\documentclass[11pt]{amsart}%
\usepackage{pdfsync}
\usepackage[english]{babel}
\usepackage[T1]{fontenc}
\usepackage[utf8]{inputenc}
\usepackage{amssymb,amsfonts,latexsym,amsmath,amsthm,amscd,graphpap,pxfonts}
\usepackage{geometry,graphicx,amsmath,amssymb,epstopdf,amsthm,mathrsfs,mathabx}   
\usepackage[usenames, dvipsnames]{color}

\usepackage{graphicx}

\newtheorem{theorem}{Theorem}[section]
\newtheorem{proposition}[theorem]{Proposition}

\newtheorem{lemma}[theorem]{Lemma}

\newtheorem{coro}[theorem]{Corollary}
\newtheorem{cor}[theorem]{Corollary}

\newtheorem{remark}[theorem]{Remark}

\newcounter{compte}
\newenvironment{enum-i}{\begin{list}{\roman{compte})} {\usecounter{compte}
\topsep=1mm \itemsep=0.2mm \leftmargin=5mm  }}{\end{list}}
\newenvironment{enum-1}{\begin{list}{\arabic{compte})} {\usecounter{compte}
\topsep=1mm \itemsep=0.2mm \leftmargin=5mm  }}{\end{list}}
\newenvironment{enum-a}{\begin{list}{\alph{compte})} {\usecounter{compte}
\topsep=1mm \itemsep=0.2mm \leftmargin=5mm  }}{\end{list}}

\renewcommand{\dim}{\operatorname{dim}}

\renewcommand{\ker}{\operatorname{Ker}}
\newcommand{\EQ}[1]{\begin{equation}\begin{split} #1 \end{split}\end{equation}}

\def\al{\alpha}
\def\f{\frac}
\def\R{\mathbb{R}}
\def\les{\lesssim}
\def\la{\langle}
\def\ra{\rangle}
\def\calH{\mathcal{H}}
\def\nn{\nonumber}
\def\eps{\varepsilon}

\def\I{\infty}
\def\p{\partial}

\def\fy{\varphi}
\def\nn{\nonumber}

\newcommand{\LR}[1]{{\langle #1 \rangle}}

\def\calL{\mathcal{L}}

\numberwithin{equation}{section}

\begin{document}

\author{N.\ Burq, G.\ Raugel, W.\ Schlag}

\address{ Nicolas Burq: Univ Paris-Sud, Laboratoire de Math\'ematiques d'Orsay,
Orsay Cedex, F-91405; CNRS, Orsay cedex, F-91405, France}
\address{ Genevi\`eve Raugel: CNRS, Laboratoire de Math\'{e}matiques d'Orsay, Orsay Cedex, 
F-91405; Univ Paris-Sud, Orsay cedex, F-91405, France} 
\address{Wilhelm Schlag: University of Chicago, Department of Mathematics, 5734 South University Avenue, Chicago, IL 60636, U.S.A.}

\title{Long time dynamics for damped Klein-Gordon equations}

\begin{abstract}
For general  nonlinear Klein-Gordon equations with dissipation we show  that 
any finite energy  radial solution either blows up in finite time or asymptotically approaches  a stationary solution in $H^1\times L^2$. In particular, any global
solution is bounded. The result applies to standard energy subcritical focusing nonlinearities $|u|^{p-1} u$, $1<p<(d+2)/(d-2)$ as well as  
any energy subcritical  nonlinearity  obeying  a sign condition
of the Ambrosetti-Rabinowitz type. The argument involves both techniques from nonlinear dispersive PDEs 
and dynamical systems (invariant manifold theory in Banach spaces and 
 convergence theorems).
\end{abstract}

\thanks{The authors thank  F.\ Merle, E.\ Hebey and M.\ Willem for fruitful discussions. In particular, they thank E.\ Hebey for having indicated them the paper of Cazenave~\cite{Caz85}. The first author was partially funded by ANR through ANR-13-BS01-0010-03 (ANA\'E).
The third author was partially supported by the NSF through DMS-1160817.}

\maketitle

\section{Introduction}

Nonlinear dispersive evolution equations such as the wave and Schr\"odinger equations have been investigated for decades.
For defocusing power-type energy subcritical or critical nonlinearities the theory is    developed, while the energy supercritical powers are wide open. 
For semilinear focusing equations the picture is less complete for  long-term dynamics. 
These equations exhibit   finite-time blowup, small data global existence and scattering, as well as time-independent solutions (solitons). 
For the energy critical wave equation
\EQ{\nn 
\Box u &= u^5 ~, \quad (t,x)\in \R^{1+3}~, \\
(u(0),  \partial_t u (0)) &\in   \dot H^1(\R^3) \times L^2(\R^3)~,
}
in the radial setting, Duyckaerts, Kenig, and Merle~\cite{DuKeMer} achieved a breakthrough by showing that all global   trajectories
can be described as a superposition of a finite number of rescalings of the  ground state $W(r)=(1+r^2/3)^{-\f12}$ plus a radiation term which is asymptotic to a free wave. 
This work introduces the novel {\em exterior energy} estimates. 
The subcritical case appears to require different techniques, however. Nakanishi and the third author~\cite{NaS} described the asymptotics of solutions provided
the energy is only slightly larger than the ground state energy. The trichotomy in forward time of (i) blowup in finite time (ii) global existence and scattering to zero (iii) global existence and
scattering to the ground state,  can be naturally formulated in terms of the center-stable manifold associated with the ground state. 

In this paper, we develop a robust approach to the problem of long-term asymptotics of the general energy subcritical Klein-Gordon equations with (arbitrarily small) dissipation. 
The focusing damped subcritical Klein-Gordon equation in $\mathbb{R}^d$, $1 \leq d \leq 6$ (for the case $d \geq 7$, see 
\cite{BRS2}), is
\begin{equation}
\label{KGspecial}
\begin{split}
& \partial_t^2 u+2\alpha  \partial_t u -\Delta u+u- |u|^{\theta -1}u=0, \cr 
&(u(0), \partial_t u (0))=(\varphi_0,\varphi_1)\in\mathcal{H}~,
\end{split}
\end{equation}
where $\mathcal{H}= H^1(\mathbb{R}^d)\times L^2(\mathbb{R}^d)$,  $\alpha \geq 0$ and  
\begin{equation}
\label{eq:theta}
1 < \theta < \theta^*~, \hbox{ with  }  \theta^*= \frac{d+2}{d-2}~. 
\end{equation}
We will limit our study to the case of radial functions 
$$
\mathcal{H}_{rad}= H^1_{rad}(\mathbb{R}^d)\times L^2_{rad}(\mathbb{R}^d)~.
$$
The energy functional $E^{\theta}$ below, also called Lyapunov functional in the dissipative case $\alpha>0$, plays an important role in the analysis of the behaviour of the solutions of \eqref{KGspecial}. This energy functional is given by
 \begin{equation}
\label{Eq:Ep}
E^{\theta}(\varphi_0,\varphi_1) = \int_{\mathbb{R}^d}\left(\frac{1}{2}|\nabla\varphi_0|^2+\frac{1}{2}\varphi_0^2+\frac{1}{2}\varphi_1^2- \frac{1}{\theta +1} |\varphi_0 |^{\theta +1}\right)\,  dx
\end{equation}
For  the  Klein-Gordon equation  \eqref{KGspecial},  it is known (see \cite{Strauss}, \cite{BeLions83I}, 
\cite{Coff72}, \cite{McLeod93} and \cite{ChenLin91} for example) that \eqref{KGspecial}  admits a unique positive radial stationary solution $(Q_g,0)$ (the ground state solution), which minimizes the energy $E^{\theta}(.,0)$ in the class of all nonzero stationary solutions $(Q,0)$ in $\mathcal{H}$, that is,
$$
0 <E^{\theta}(Q_g,0) = \min \{ E(Q,0) \, | \, Q  \in H^1(\mathbb{R}^d), Q \ne 0, -\Delta Q +Q - |Q|^{\theta -1}Q =0\}
$$
The behaviour of solutions of \eqref{KGspecial} with initial data $(\varphi_0,\varphi_1) \in \mathcal{H}$ with energy 
$E^{\theta}(\varphi_0,\varphi_1) < E^{\theta}(Q_g,0)$  is rather  well understood in the case $\alpha \geq 0$ since these solutions remain  in the so-called  Payne-Sattinger sets 
(see \cite{PaSat}) for all positive times (and also negative times for $\alpha=0$). In these Payne-Sattinger domains, the solutions either blow-up in finite time or globally exist and  scatter to $0$ when $\alpha=0$ or converge to $0$ for $\alpha > 0$, respectively. For a description of this phenomenon in the case $\alpha=0$, we refer for example to the book \cite{NaS}.

\medskip

It is also well-known that this equation has an infinite number of radial equilibrium points 
$(e_{\ell},0)$ with a prescribed number $\ell\ge1$ of zeros (these are called {\em nodal solutions}, see for example \cite{BeLions83II}).  
Unfortunately, one  knows almost nothing  about the uniqueness and the hyperbolicity of those nodal solutions. 
(\cite{CG-HY11} obtains uniqueness results for nodal solutions but for sub-linear nonlinearities). 
In  the Hamiltonian case ($\alpha=0$), this in part prevents the description of the behaviour of the solutions $\vec u(t)$ of 
\eqref{KGspecial} whose initial data $(\varphi_0,\varphi_1)$ have an energy $E^{\theta}(\varphi_0,\varphi_1)$ much larger than the one of the ground state $(Q,0)$.

In 1985 Cazenave~\cite{Caz85} established the following dichotomy for the Hamiltonian case 
$\alpha =0$:   solutions of \eqref{KGspecial} either blow up in finite  time or are global and bounded in $\mathcal{H}$, provided $1<\theta < +\infty$, if $d=1, 2$ with $\theta \leq 5$ if $d=2$ and $1 < \theta \leq \frac{d}{d-2}$ if $d \geq 3$ (the result of Cazenave should extend to the case $\alpha >0$).
In 1998 Feireisl~\cite{Feir98A}, for the dissipative case $\alpha >0$,  gave an independent proof of the boundeness of the global solutions of \eqref{KGspecial}, when $d \geq 3$ and $1< \theta < 1 + \min (\frac{d}{d-2},\frac{4}{d})$ (for the case $d=1$, see his earlier paper \cite{Feir94a}). 

Unfortunately,  the proofs of  Cazenave \cite{Caz85} and of Feireisl \cite{Feir98A} do not seem  to extend to nonlinearities satisfying $\frac{d}{d-2} < \theta < \frac{d +2}{d-2}$, when $d \geq 3$, where one needs to use Strichartz estimates in the various a priori estimates rather than Galiardo-Nirenberg-Sobolev inequalities. In this paper, we restrict our study to the dissipative radial case ($\alpha >0$)   and show the following dichotomy.

\begin{theorem} \label{ThBRS1special}
Let $\alpha >0$ and $ d \leq 6$. Then, 
\begin{enumerate}
\item either the solutions of \eqref{KGspecial} in 
$\mathcal{H}_{rad}$ blow up in finite positive time, 
\item or  they are global in positive time and converge to an equilibrium point. 
\end{enumerate}
In  particular,  all global in positive times solutions are bounded for positive times. 
\end{theorem}

We notice that this theorem is a particular case of Theorem \ref{ThBRS1} below. In \cite{BRS2}, we will partly generalise this dichotomy to non-radial solutions.

\medskip 

Actually the above dichotomy holds for some more general nonlinearities and, in this paper, we consider the damped Klein-Gordon equation in $\mathbb{R}^d$, $ d \leq 6$ (for the case $d \geq 7$, see \cite{BRS2}),
\begin{equation*}
\tag*{$(KG)_\alpha$}
\label{KGalpha}
\begin{split}
& \partial_t^2 u+2\alpha  \partial_t u -\Delta u+u-f(u)=0~, \cr
&(u(0), \partial_t u (0))=(\varphi_0,\varphi_1)\in\mathcal{H}_{rad}~,
\end{split}
\end{equation*}
where $f: y \in  \mathbb{R} \mapsto f(y) \in  \mathbb{R}$ is an odd $C^1$-function, $f'(0)=0$,  which satisfies the following Ambrosetti-Rabinowitz type condition: there exists 
a constant $\gamma >0$ such that 
\begin{equation*}
\tag*{$(H.1)_f$}
\label{H1f}
\begin{split}
\int_{\mathbb{R}^d} \big(2(1+ \gamma) F(\varphi) - \varphi(x) f( \varphi(x))\big) \, dx \leq 0~, \quad \forall \varphi \in H^1(\mathbb{R}^d)~,
\end{split}
\end{equation*}
where $F(y) = \int_0^y f(s)ds$.\\
We also need to impose a growth condition on $f$, when $d \geq 2$. We assume that,
\begin{equation*}
\tag*{$(H.2)_f$}
\label{H2f}
\begin{split}
&|f'(y)| \leq  C \max\big( |y|^\beta, |y|^{\theta -1}\big)~, \quad   \forall  y \in \mathbb{R}~,  \cr
&|f'(y_1) - f'(y_2)| \leq C \big( |y_1 - y_2 |^{\beta} + |y_1 - y_2 |^{\theta -1} \big)~, \quad   \forall  y_1, y_2
 \in \mathbb{R}~,
\end{split}
\end{equation*}
where  $1 < \theta < \theta^* $, $ 0<\beta<\theta-1$, $\theta^* = 2^*-1$ and where $2^*=\infty$ if $d=1,2$ and $2^*=\frac{2d}{d-2}$ if $d\ge3$. 
We notice that, when $d \geq 3$, $ \theta^* = \frac{d+2}{d-2}$. \\
 In other words,  the growth of $f$ is energy subcritical for large $y=0$,  and we also assume that $f'$ is  $\beta$-H\"older continuous. 
 For sake of simplicity in the proofs below, we may assume, without loss of generality, that 
 $ 0 < \beta < \min( \theta -1, \frac{2}{d-2})$.

\medskip

We remark that our argument does not depend on the existence or uniqueness of a ground state solution. Note that 
  Hypothesis \ref{H1f} alone does not imply the existence and uniqueness of a ground state solution. 
We further note that  Hypothesis \ref{H1f} may actually be replaced by the following weaker one: 
\begin{equation*}
\tag*{$(H.1bis)_f$}
\label{H1bisf}
\begin{split}
\int_{\mathbb{R}^d} \big(2(1+ \gamma) F(\varphi) - \varphi(x) f( \varphi(x))\big) dx \leq 0~, \quad  \hbox{for } \| \varphi\|_{H^1} \hbox{ large enough}.
\end{split}
\end{equation*}
But, for sake of simplicity, we assume  \ref{H1f}  throughout.
A classical example of a function $f$ satisfying  hypotheses \ref{H1f} and \ref{H2f} is as follows: 
 \begin{equation}
\label{fExample2}
\begin{split}
f(u) =  \sum_{i=1}^{m_1} a_i |u |^{p_i-1}u - \sum_{j=1}^{m_2}b_j |u |^{q_j -1}u~, 
&\hbox{ with }  1< q_j < p_i \leq \frac{d+2}{d-2}, \,\forall i,j  \cr
& \hbox{ and }  a_i, b_j \geq 0, a_{m_1}>0~.
\end{split}
\end{equation}

In Section~\ref{sec:basic}, we shall prove that the equation \ref{KGalpha} generates a local dynamical system on 
 $\mathcal{H}$ as well as on $\mathcal{H}_{rad}$, for $\alpha \geq 0$. We denote $S_{\alpha}(t)$, $\alpha \geq 0$, this local dynamical system. As in the particular case of the Klein-Gordon equation \eqref{KGspecial}, we introduce  the energy functional (also called Lyapunov functional in the case of positive damping $\alpha>0$) on  $\mathcal{H}$:
 \begin{equation}
\label{Eenergie}
E(\varphi_0,\varphi_1) = \int_{\mathbb{R}^d}\left(\frac{1}{2}|\nabla\varphi_0|^2+\frac{1}{2}\varphi_0^2+\frac{1}{2}\varphi_1^2- F(\varphi_0)\right)\,  dx~.
\end{equation}

 \medskip 

The natural first step in the study of the dynamics of the equation \ref{KGalpha} consists in studying the boundedness or unboundedness of its global (in positive times) solutions.
As already mentioned above, under restrictions on the growth rate of the nonlinearity, Cazenave~\cite{Caz85} and Feireisl~ \cite{Feir98A} established this boundedness. 
In this paper, taking advantage of the fact that all the functions are radial, we will show the boundedness of the global solutions of \ref{KGalpha}, for $\alpha >0$, by using  ``dynamical systems'' arguments. Indeed, we will show that each global solution $\vec u(t)$ converges to an equilibrium point as $t$ goes to $+\infty$.

If the equation \ref{KGalpha} admits a ground state solution and is Hamiltonian, the  functional $K_0: \varphi \in H^1(\mathbb{R}^d) \mapsto K_0(\varphi) \in \mathbb{R}$ 
defined as 
 \begin{equation}
\label{eq:DefK0}
K_0(\varphi) = \int_{\mathbb{R}^d}\left(|\nabla\varphi|^2+\varphi^2- \varphi f(\varphi)\right)\,  dx~,
\end{equation}
has played a decisive  role in the description of the dynamics of the solutions with initial energy smaller or slightly larger than
the one of the ground state (see \cite{PaSat}, \cite{NaS} for example). It will also be important in our situation. First we shall prove in Lemma \ref{eqBlowup}, that if 
$$
\vec u(t) =
S_{\alpha}(t)(\varphi_0,\varphi_1)(t) \equiv (u(t),  \partial_t u (t))
$$
satisfies  $K_0(u(t)) \leq  -\delta$ (where $\delta >0$),
on the maximal interval of existence, the solution blows up in finite time. On the other hand, 
we will see that, if $K_0(u(t)) \geq \eta$ for some finite $\eta$ on the maximal interval of existence, the solution exists and is bounded for all positive times. 

In order to prove that each global solution 
$\vec u(t) =S_{\alpha}(t)(\varphi_0,\varphi_1)(t)$  converges to an equilibrium point 
as $t$ goes to $+\infty$, we  argue by contradiction. If this trajectory $\vec u(t)$ is unbounded in positive time, then there exists  a sequence of times $t_n$, $t_n \rightarrow_{n \rightarrow +\infty} +\infty$, such that $$K_0(u(t_n)) \rightarrow_{n \rightarrow +\infty} 0$$
 Then, using this sequence of times $t_n$, we show in 
Theorem \ref{K0(tn)stern}, that the $\omega$-limit set $\omega(\varphi_0,\varphi_1)$ of $(\varphi_0,\varphi_1)$ is non-empty and contains at least one equilibrium point $(Q^*,0)$ of the equation \ref{KGalpha}. We recall that the $\omega$-limit set 
$\omega(\varphi_0,\varphi_1)$ of $(\varphi_0,\varphi_1)$ is defined as follows:
 \begin{equation}
\label{eqOmegaLim}
 \begin{split}
\omega(\varphi_0,\varphi_1) = \{ \vec w \in \mathcal{H}_{rad} \, | \, \exists &\hbox{  a sequence } \tau_n \geq 0, 
\hbox{ so that } \tau_n  \rightarrow_{n \rightarrow +\infty} +\infty~,\cr
& \hbox { and }  S_{\alpha} (\tau_n) (\varphi_0,\varphi_1)  \rightarrow_{n \rightarrow +\infty} \vec w  \}~. 
\end{split}
\end{equation}
Then, in Section~\ref{sec:conv}, taking advantage of the fact that the linearized Klein-Gordon equation around $(Q^*,0)$ in the space $\mathcal{H}_{rad}$ has a kernel which is  at most one-dimensional, we show, by using classical convergence arguments based on invariant manifold theory,  that the  trajectory converges to this equilibrium in positive infinite time, and is therefore bounded.  
  
  \begin{theorem} \label{ThBRS1}
 Let $\alpha >0$.
 Assume that $1 \leq d \leq 6$ and that $f$ satisfies the conditions \ref{H1f} and \ref{H2f}.
Let  $(\varphi_0,\varphi_1)\in\mathcal{H}_{rad}$, then
\begin{enumerate}
\item  either $S_\alpha(t)(\varphi_0,\varphi_1)$ blows up in finite time,
\item or   $S_\alpha(t)(\varphi_0,\varphi_1)$ exists globally and converges 
to an equilibrium point $(Q^*,0)$ of \ref{KGalpha}, as $t\to+\infty$.
\end{enumerate}
 \end{theorem}
For  the case $d\geq 7$,  we refer the reader to \cite{BRS2}. 

 \medskip 

To place this result into context, we now briefly recall   various related convergence theorems. Since we are considering the equation \ref{KGalpha} in the radial setting, the linearized Klein-Gordon operator around the equilibrium $(Q^*,0)$ has a kernel of dimension less than or  equal to $1$, that is, either  $0$ does not belong to the spectrum of the elliptic selfadjoint operator 
$$
\mathcal{L} \equiv -\Delta + I - f'(Q^*)
$$ 
or $0$ is a simple eigenvalue of $\mathcal{L}$ (see Section~\ref{sec:basic}, 
Lemma \ref{lem:spectre2A}). If $0$ is a simple eigenvalue of $\mathcal{L}$, then the dynamical system $S_{\alpha}(t)$ admits a $C^1$ local center manifold 
$W^c((Q^*,0))$ of dimension $1$ at $(Q^*,0)$. Since the $\omega$-limit set of any element $(\varphi_0,\varphi_1) \in \mathcal{H}_{rad}$ belongs to the set of equilibria, if the trajectory of $S_{\alpha}(t) (\varphi_0,\varphi_1) \equiv \vec u(t)$  were precompact in $\mathcal{H}_{rad}$, we could directly conclude by using the convergence results contained in 
\cite{BrP97b} or in \cite{HaRa} for example that the whole trajectory 
$S_{\alpha}(t) (\varphi_0,\varphi_1)$ converges to $(Q^*,0)$, when $t$ goes to infinity. Unfortunately, we do not know that the trajectory $S_{\alpha}(t) (\varphi_0,\varphi_1)$ is bounded and thus we do not even know that the $\omega$-limit set of $(\varphi_0,\varphi_1)$
is  bounded and connected.  However, adapting the proof of \cite[Lemma 1]{BrP97b} and using the asymptotic  phase property of the local center unstable and local center manifolds around $(Q^*,0)$ (see Appendix A for these concepts), we easily obtain that the entire  trajectory  
$S_{\alpha}(t) (\varphi_0,\varphi_1)$ converges to $(Q^*,0)$ as  $t$ goes to infinity. An alternative way for proving the convergence of the trajectory 
$S_{\alpha}(t) (\varphi_0,\varphi_1)$ towards $(Q^*,0)$ would be to use (instead of dynamical systems arguments) a \L ojasiewicz-Simon inequality (see  Sections~3.2 and 3.3 in the monograph of L.\ Simon \cite{Si96} and also \cite[Theorem 2.1]{HaJen07}) together with functional arguments as in Jendoubi and Haraux  (see \cite{HaJen99} or  
\cite{HaJen07}). The proof of the \L ojasiewicz-Simon inequality in \cite{Si96}  uses a Lyapunov-Schmidt decomposition. In the special case where the kernel of 
$\mathcal{L}$ is one-dimensional, this proof also shows 
that the set of equilibria of \ref{KGalpha} passing through $(Q^*,0)$ is a $C^1$-curve. Using this \L ojasiewicz-Simon inequality and introducing an appropriate functional like in \cite{HaJen07}, we could show that the $\omega$-limit set of every precompact trajectory converges to an equilibrium point. Unfortunately, the  trajectory $S_{\alpha}(t) (\varphi_0,\varphi_1)$ is not a priori bounded and it seems difficult to adapt the functional part of the proof of \cite[Theorem 3.1]{HaJen07}. Moreover, there is an additional difficulty in the construction of  such an appropriate functional coming from the fact that we need to use Strichartz estimates. So we have not been able to follow this route. 
 \vskip 5mm 
 
 The plan of this paper is as follows. Section~\ref{sec:basic} is devoted to basic properties of the Klein-Gordon equation~\ref{KGalpha}. In particular, we recall the local existence and uniqueness of mild solutions of the equation \ref{KGalpha}. In Section~\ref{sec:Nehari}, we introduce the functional $K_0$, which not only plays an important role in the proof of  Theorem \ref{ThBRS1} but also  defines the well-known Nehari manifold $\mathcal{N}$ as the locus of the radial zeros  of the functional $K_0$. In Lemma~\ref{eqBlowup}, we give a sufficient condition on $K_0$ for blow-up in finite time of the solutions of  \ref{KGalpha}. We end this section by describing the spectral properties of the linearized Klein-Gordon equation around a (radial) equilibrium point. Section~\ref{sec:core} is the core of this paper. In Section~\ref{sec:seq}
(see Theorem \ref{K0(tn)stern}) we show that if a solution $\vec u(t)$ does not blow up in finite positive time, then the $\omega$-limit set $\omega(\vec u(0))$ contains at least one equilibrium point. In Section~\ref{sec:conv} we show that the whole trajectory  $\vec u(t)$  converges to this equilibrium point and  is therefore bounded. 
  In Section \ref{sec:IMforKG}, we apply the classical invariant manifold theory, recalled in Appendix A, in order to construct the  local unstable, center unstable and center manifolds about equilibrium points of the Klein-Gordon equation \ref{KGalpha} and the  unstable, center unstable and center manifolds about equilibrium points of the localized Klein-Gordon equation \eqref{eq:Auphi0mod}.
In Appendix A, we recall the existence theorems for local center-stable,  local center-unstable and local center manifolds together with their foliations and exponential attraction properties with asymptotic phase in the formulation of Chen, Hale and Tan (see \cite{ChHaBT}). 
Finally, in Appendix B, we recall the classical convergence theorem (see \cite{Aulb}, 
\cite{HaMa} or \cite{HaRa})  in the generalised  form given by 
Brunovsk\'{y} and Pol\'{a}\v{c}ik in \cite{BrP97b}. 

Such a convergence theorem is needed in case the  dynamics near the equilibrium exhibits a nontrivial center manifold. As a result of dissipation and the radial condition,  this center manifold can be at most one-dimensional. For the nonlinearities~\eqref{fExample2}, it is known that the kernel of the linearized operator about  the ground state is trivial, see~\cite{ChenLin91}. But, due to the lack of precise description of the bound states,  we cannot guarantee that the local center manifold is absent about a bound state. The local strongly unstable manifold is finite-dimensional. The local strongly stable manifold is infinite-dimensional in stark contrast to the Hamiltonian 
scenario for which the local center manifold is the largest piece. The convergence theorem in 
\cite{BrP97b} then guarantees that, if the $\omega$-limit set is not a single equilibrium point $(Q^*,0)$, and if  $(Q^*,0)$ is stable for the restriction of $S_{\alpha}(t)$ to the local center manifold of  $(Q^*,0)$ (for this definition of stability, see  \eqref{eq:2.16} 
and Appendix B), then this $\omega$-limit set must contain a point on the unstable 
manifold of $(Q^*,0)$, distinct from
$(Q^*,0)$. But this contradicts the fact that, due to the properties of the Lyapunov functional \eqref{Eenergie}, the $\omega$-limit set is contained in the set of equilibrium points.  
\section{Basic properties}\label{sec:basic}

\subsection{Local  existence results}    \label{sec:linear decay}

Consider the linear equation, with $\alpha\ge0$, 
\EQ{
\label{eq:linear}
 \partial_t^2 u + 2\alpha  \partial_t u -\Delta u + u &= G, \quad (u, \partial_t u )\Big|_{t=0} = (u_0, u_1)\in H^1(\R^d) \times L^2(\R^d).
}
Since $v(t)=e^{\al t}u(t)$ satisfies 
\EQ{\label{eq:v eq}
v_{tt} -\Delta v + (1-\al^2) v= e^{\al t} G, \quad (v,v_t)\Big|_{t=0} = (u_0, u_1+\al u_0).
}

We deduce that the solution of~\eqref{eq:linear} is given by 
\EQ{
\label{eq:lin_sol}
u(t) &= e^{-\al t} \Big[ \cos(t\sqrt{-\Delta+1-\alpha^2}) + \al \f{\sin(t \sqrt{-\Delta+1-\alpha^2})}{\sqrt{-\Delta+1-\al^2}}\Big] u_0 \\
& + e^{-\al t} \f{\sin(t \sqrt{-\Delta+1-\alpha^2})}{\sqrt{-\Delta+1-\alpha^2}} u_1 
+ \int_{0}^{t} \f{\sin((t-s) \sqrt{-\Delta+1-\alpha^2})}{\sqrt{-\Delta+1-\alpha^2}} e^{-(t-s)\al}G(s)\, ds \\
& =\mathcal{S}_{1,\alpha}(t) u_0 + \mathcal{S}_{2,\alpha}(t) u_1 
+ \int_{0}^{t}
\mathcal{S}_{2,\alpha}(t-s) G(s) \, ds~.}

Clearly, the  regimes $0\le\alpha <1$, $\al=1$, and $\al>1$ exhibit quite different behaviours. 
The dispersion relation for $\al<1$ is that of Klein-Gordon (the characteristic variety is a hyperboloid), whereas for $\al=1$ it is that of the wave equation (the characteristic variety is a cone).

If $X$ is a Banach space, then we let $L^{p,\beta}_{t}(X)$ be the space with norm
\[
\|f\|_{L^{p,\beta}_{t}(X)} = \| e^{\beta t} \| f(t)\|_{X} \|_{L^{p}_{t}},\quad  \beta \in \mathbb{R}
\]
In this section, the $\beta$ in these weighted estimates has nothing to do with the regularity in~\eqref{H2f}. 

\begin{lemma}
\label{lem:free}
Let $0\le\al <1$ and assume $d\ge3$ for simplicity. Set $p=\f{2d}{d-2}$ and $\sigma=\f12-\f1d$, $\sigma'=1-\sigma$. 
The solution $u$ of \eqref{eq:linear} satisfies 
 the following Strichartz-type estimates  for any $0\le\beta\le\al$, 
\EQ{\label{eq:Strich1} 
\| u \|_{L^{2,\beta}_{t} B^{\sigma}_{p,2}\cap L^{\infty,\beta}_{t} H^{1}_{x}} \le C(\al)\Big[ \| (u_{0},u_{1}) \|_{H^{1}\times L^{2}} + \|G\|_{L^{2,\beta}_{t}B^{\sigma'}_{p',2}+L^{1,\beta}_{t}L^{2}_{x}} \Big]
}
where $C(\al)$ is uniform on compact intervals of $[0,1)$.   
\end{lemma}

\begin{proof}
This follows from \eqref{eq:v eq} and the Keel-Tao endpoint for the Klein-Gordon equation, see for example Lemma~2.46 in \cite{NaS}. 
\end{proof}

Lemma~\ref{lem:free} does not hold for $\al\ge1$. Indeed, for $\al=1$ we would need to replace the Strichartz estimates 
for Klein-Gordon in~\eqref{eq:Strich1} with those for the wave equation.   We set
$
\beta(\alpha)=\alpha
$ if $0\le\alpha\le1$ and $$\beta(\alpha)=\alpha-\sqrt{\alpha^{2}-1}$$ if $\al>1$.  Exploiting the exponential decay in~\eqref{eq:lin_sol}
we can now state the following  space-time averaged estimates. 

\begin{lemma}
\label{lem:free2}
Let $\al>0$. In all dimensions $d\ge1$ the solution $u$ of \eqref{eq:linear} satisfies the following energy bounds with decay
\EQ{\label{eq:ener dec}
\sup_{t\ge 0} e^{t\beta(\al)} \| (u, \partial_t u )(t) \|_{H^1\times L^2} \le C(\al) \Big [ \| (u_{0},u_{1}) \|_{H^{1}\times L^{2}}  +  \int_0^\infty e^{s\beta(\al)} \| G(s)\|_2 \, ds \Big]
}
as well as the exponentially weighted Strichartz  estimates, in dimensions $d\ge2$,   and 
with $0\le \beta<\beta(\al)$, 
\EQ{\label{eq:Strich2} 
\|  u\|_{L^{q,\beta}_t L^p_x} \le C(\al,\beta)\big[ \| (u_{0},u_{1}) \|_{H^{1}\times L^{2}} + \|  G\|_{L^{\tilde q',\beta}_t L^{\tilde p'}_x   }  \big]
}
where $\f{1}{q}+\f{d}{p}=\f{d}{2}-1= \f{1}{\tilde q'}+\f{d}{\tilde p'}-2$, $2\le p,\tilde p<\infty$, $2\le q,\tilde q$, and $\f1q+\f{d-1}{2p}\le\f{d-1}{4}$,  $\f1{\tilde q}+\f{d-1}{2\tilde p}\le\f{d-1}{4}$. 
The constant $C(\al,\beta)$ is uniform on compact subsets of 
$$
\{(\al,\beta)\:|\: \al \in (0,\infty), \,0\le \beta<\beta(\al)\}
$$
\end{lemma}

\begin{proof}
Taking the Fourier transform of \eqref{eq:lin_sol}
yields 
\[
\hat u(t,\xi) =  m_{\al}(t,\xi) \widehat{u_{0}}(\xi) + \tilde m_{\al}(t,\xi) \widehat{u_{1}} (\xi) + \int_{0}^{t} \tilde m_{\al}(t-s,\xi) e^{-(t-s)\al}\widehat{G}(s,\xi)\, ds
\]
The multipliers satisfy the estimates
\[
|  m_{\al}(t,\xi)|  + |  \tilde m_{\al}(t,\xi)|  \le C(\al) e^{-\beta(\alpha)t}
\]
which proves~\eqref{eq:ener dec}.   For \eqref{eq:Strich2} we introduce the Littlewood-Paley decomposition
\[
1 = P_{\les \al} + \sum_j P_j = P_{\les \al} + P_{> \al}
\]
where the $P_j$ are associated to frequencies $2^j>\alpha$ and $P_{\les \al} f = f$ for all Schwartz functions with support in $\{|\xi|\le 1+2\alpha\}$. 
Let $K_\lambda ^{\pm} (t)$ be the propagator defined by, cf.~\eqref{eq:lin_sol},  
\[
[K_\lambda^{\pm} (t) f](x) = e^{-\al t}\int_{\R^d} e^{\pm it\sqrt{\xi^2+1-\alpha^2}}  e^{ix\cdot\xi} \chi(\xi/\lambda) \hat{f}(\xi)\, d\xi 
\]
where $\chi$ is the usual Littlewood-Paley bump function supported on an annulus, and $\lambda >\alpha+1 $ (and ignoring multiplicative constants). Then the root is smooth, and 
we may apply stationary phase to conclude that 
\[
\| K_\lambda^\pm (t)\|_\infty \le  e^{-\al t} \lambda^{d} \la t\lambda \ra^{-\f{d-1}{2}}   \les e^{-\al t} t^{-\f{d-1}{2}} \lambda^{\f{d+1}{2}}
\]
for all $t>0$. Proceeding as for the wave equation (see Keel-Tao), and ignoring the exponential decay for the frequencies $\gtrsim \al$,  yields the Strichartz estimates 
\eqref{eq:Strich2} for $P_{>\al} u$ with $\beta=0$.  On the other hand, by the same logic we can also derive Strichartz estimates for the transformed equation~\eqref{eq:v eq}
which yields~\eqref{eq:Strich2} with $\beta=\al$ for the piece $P_{>\al} u$. Interpolating between these two cases we obtain Strichartz inequalities for all $0\le\beta\le \al$
for those frequencies. Smaller frequencies require smaller $\beta$. Indeed, for the remaining piece $P_{\les \al} u$ we use the energy bound~\eqref{eq:ener dec} and Bernstein's inequality. 
To be precise,  the energy estimate
 \[
 \| P_{\les \al} u(t) \|_2  \le C(\al) \Big [ e^{-t\beta(\al)}\| (u_{0},u_{1}) \|_{H^{1}\times L^{2}}  +  \int_0^t e^{-(t-s)\beta(\al)} \| P_{\les \al} G(s)\|_2 \, ds \Big]
 \]
 implies via Bernstein's inequality that 
  \[
 e^{\beta t} \| P_{\les \al} u(t) \|_q  \le C(\al) \Big [ e^{-t(\beta(\al)-\beta)}\| (u_{0},u_{1}) \|_{H^{1}\times L^{2}}  +  \int_0^t e^{-(t-s)(\beta(\al)-\beta)} e^{\beta s}\| P_{\les \al} G(s)\|_{\tilde q'} \, ds \Big]
 \]
 Taking $L^p_t$ norms on both sides, and applying Young's inequality to the Duhamel integral yields \eqref{eq:Strich2} for all frequencies. 
\end{proof}

We now turn to the nonlinear equation~\ref{KGalpha}. We write $\vec u=(u, \partial_t u )$. 

\begin{theorem} \label{thm:WP}
Let $d \leq 6$. 
Let $f:\R\to\R$ be a $C^1$ odd  function, satisfying the assumption \ref{H2f}.  Then for every data $\vec u_0$ in $\calH=H^1(\R^d) \times L^2(\R^d)$ ( resp.\   in 
$\mathcal{H}_{rad}$)
  the equation \ref{KGalpha} has a unique strong solution 
  \[
  u \in X \equiv X_T :=C([0,T], H^1(\R^d))\cap C^1([0,T], L^2(\R^d))
\]
  ( resp.\  in $C([0,T], H_{rad}^1(\R^d))\cap C^1([0,T], L_{rad}^2(\R^d))$ ),
  where $T$ only depends on $\| \vec u_0\|_{\calH}$. \\
   Moreover, if $3 \leq d \leq 6$, the solution belongs to  
  $$
  L^{\theta^*}((0,T), L^{2\theta^*}(\mathbb{R}^d))
  $$
  where 
  $\theta^*= \frac{d+2}{d-2}$ and the estimate \eqref{eq:fixed} below holds. \\
   Furthermore, the following properties hold.
  \begin{enumerate}
  \item  The solution 
  $$
  (t, \vec u_0) \in [0,T] \times \mathcal{H} \mapsto  \vec u (t) \equiv (u(t),\partial u(t))
  \in \mathcal{H}
  $$ 
  is continuous.
  \item  For any $0 \leq \tau \leq T$, the map  $ \vec u_0 \in  \mathcal{H} \mapsto S_{\alpha}(\tau) \vec u_0 \equiv \vec u(\tau) \in 
  \mathcal{H}$ is Lipschitz continuous on the bounded sets of $\mathcal{H}$ (see \eqref{eq:Lipuv}).
  \item   The map  $\vec u_0  \in \mathcal{H} \mapsto u(t) \in X \cap
 L^{\theta^*}((0,T), L^{2\theta^*}(\mathbb{R}^d))$ is a $C^1$-map.
  \item Let $T^*$ be the maximal time of existence. If $T^*<\infty$, then 
  $$
  \limsup_{t\to T^*} \|\vec u(t)\|_{\calH}=+\infty
  $$
  \item If  
  $\vec u_{0}\in H^{2}(\R^d)\times H^1(\R^d)$, then 
  \[
  u \in C([0,T), H^{2}(\R^d))\cap C^1([0,T), H^1(\R^d))
  \]
  \item The energy~\eqref{Eenergie} 
decreases:
\begin{equation}
\label{Et1t2}
E(\vec u(t_2)) - E(\vec u(t_1))=-2\alpha \int_{t_1}^{t_2} \| \partial_t u (s)\|_{L^2}^2 \, ds 
\end{equation}
and, in particular,
\begin{equation}
\label{Et2}
E(\vec u(t_2)) + 2\alpha \int_{0}^{t_2} \| \partial_t u (s)\|_{L^2}^2\, ds \leq   E(\vec u(0))
\end{equation} 
\item If $\|\vec u(0)\|\ll 1$, then the solution exists globally, and $\| \vec u(t)\|_{\calH}$ converges exponentially to
$0$ as $t\to\infty$.  
  \end{enumerate}
  \end{theorem}
\begin{proof}
We have $|f(u)|\les |u|^2 + |u|^\theta$. We begin with dimensions $d=1,2$. In that case the Sobolev embeddings and \eqref{eq:ener dec} imply,  for any $\beta$,
\EQ{
\label{eq:weight} 
e^{\beta t} \|\vec u(t)\|_{\calH} &\les \| \vec u(0)\|_{\calH} + \int_0^t e^{-\beta(t-s)} \big( \|u(s)\|_4^2 + \|u(s)\|_{2\theta}^\theta\big)\, ds\\
&\les \| \vec u(0)\|_{\calH} + \beta^{-1}(1-e^{-\beta T}) \max_{0\le s\le T}  e^{\beta s} \big( \|\vec u(s)\|_\calH^2 + \|\vec u(s)\|_{\calH}^\theta\big) 
}
where the implicit constant is of the form $C(\alpha)$ as above.  For $T$ small we discard the exponential weight whence 
\[
  \|\vec u(t)\|_{\calH} \les \| \vec u(0)\|_{\calH} + T \max_{0\le s\le T}   \big( \|\vec u(s)\|_\calH^2 + \|\vec u(s)\|_{\calH}^\theta \big) 
\]
This immediately shows that we can set up a contraction in the space $X$ and that $T$ only depends on $\| \vec u(0)\|_{\calH}$. 
Moreover, global existence for small data follows from~\eqref{eq:weight} by the method of continuity. This also implies  the exponential decay. 

We shall establish the persistence of $H^{2}\times H^{1}$ regularity later in this proof. Taking it for granted for now, 
and using the density of $H^2(\mathbb{R}^d)\times H^1(\mathbb{R}^d)$ in $H^1(\mathbb{R}^d)\times L^2(\mathbb{R}^d)$, 
one shows that
\begin{equation}
\label{dEdt}
E(\vec u(t))\in C^1((-\tilde{T},T)), \text{ and }
\frac{d}{dt}E(\vec u(t))=-2\alpha\| \partial_t u (t)\|_{L^2}^2.
\end{equation}
Integrating this implies the  above identities for the energy. 

We now continue with the dimensions $d\ge3$. If $\theta \le \f{2^{*}}{2}=\f{d}{d-2}$, then the same energy bounds suffice. 
As usual,  larger $\theta$ requires the Strichartz bounds.  

The local wellposedness for small times does not require the exponential weights in the Strichartz estiimates and are identical to standard proofs for the wave equation. 

We first recall the main lines of the proof of the local existence and uniqueness of the solution. The local existence is proved by using the classical strict contraction fixed point theorem with parameters. In the fixed point argument below, we will use the Strichartz inequality 
\eqref{eq:Strich2} given in Lemma \ref{lem:free2}. 
Let 
$\theta^*=2^*-1= \frac{d+2}{d-2}$, $(\tilde{p}', \tilde{q}')= (2,1)$ and $(p,q)=
(2\theta^*,\theta^*)$. We remark that these pairs satisfy the conditions of  Lemma 
\ref{lem:free2} and in particular $q \geq 2$ if $d \leq 6$.

 Let $K_0>0$ be a fixed constant.  In what follows, we denote 
$B_{\mathcal H}(0,K_0)$ the ball of center $0$ and radius $K_0$ in $\mathcal H$.
Using the notation of the previous lemma, we set 
\begin{equation}
\label{eq:M0}
M_0 \equiv M_0(\alpha) =  4(C(\alpha) + 
C(\alpha,0)) K_0 \equiv 4 C_1(\alpha) K_0
\end{equation} 
and $T>0$ will be a positive constant, to be determined later.

\noindent 
We  introduce the following space
\begin{equation}
\begin{split}
\label{spaceY}
Y \equiv  Y_{T} \equiv \{ \vec u \in L^{\infty}((0,\tau_0), \mathcal{H}) \hbox{ with } 
& u \in L^{\theta^*}((0,\tau_0),L^{2\theta^*}(\mathbb{R}^d)) \cr 
&\, | \, \| u\|_{L^{\infty}(H^1) \cap W^{1,\infty}(L^2) \cap L^{\theta^*}(L^{2\theta^*})} \leq M_0\}~.
\end{split}
\end{equation}
We consider the mapping 
$$
\mathcal{F} : (\vec u_0, \vec u) \in B_{\mathcal H}(0,K_0) \times  Y \mapsto 
\mathcal{F}(\vec u_0, \vec u ) \equiv (\mathcal{F}_1, \mathcal{F}_2)(\vec u_0, \vec u ) \in Y~,
$$
defined by
\begin{equation}
\label{eq:definFronde}
  (\mathcal{F}_1(\vec u_0, \vec u ))(t) = \mathcal{S}_{1,\alpha}(t) u_0 + \mathcal{S}_{2,\alpha}(t) u_1 
+ \int_{0}^{t} \mathcal{S}_{2,\alpha}(t-s) f(u(s))\,  ds~,
\end{equation}
 and $\mathcal{F}_2(\vec u_0, \vec u ) = \partial_{t}\mathcal{F}_1(\vec u_0, \vec u )$, where 
$\vec u_0=(u_0, u_1)$ and $\vec u =(u, \partial_t u)$ . Fix some $\vec u_0 \in \mathcal{H}$  with $\|\vec u_0 \|_{\mathcal{H}} < K_0$. Consider the map $\mathcal{F}(\vec u_0, .): \vec u \in Y \mapsto 
\mathcal{F}(\vec u_0, \vec u ) \in Y$ and simply write $\mathcal{F}(\vec u_0, \vec u ) = \mathcal{F} \vec u$.

An application of Lemma \ref{lem:free2}   implies
\begin{equation}
\label{eq:calF0}
 \| \mathcal{F}(u_{0},0)\|_{Y} \leq C_1(\alpha) K_0 \leq \frac{M_0}{4}~.
\end{equation}
 Applying again Lemma 
\ref{lem:free2} and using the hypothesis \ref{H2f}, we get 
\begin{equation}
\begin{split}
\label{eq:calFaux1}
\| \mathcal{F} \vec u - \mathcal{F} \vec v \|_{Y} &\leq C_1(\alpha) \int_0^T \| f(u(s)) - f(v(s))\|_{L^2}\,  ds \cr
& \leq C_1(\alpha) \int_0^T \| \int_0^1 f' (v(s) +\lambda (u(s) -v(s)) ) (u(s) -v(s)) \,
d\lambda\|_{L^2} \, ds \cr
& \leq C_1(\alpha) C \int_0^T \| ( 1 + |u(s)|^{\theta -1} + |v(s)|^{\theta -1}) 
|u(s) - v(s)| \,\|_{L^2}\, ds
\cr
& \leq C_1(\alpha) C \big[ T \| u-v\|_{L^{\infty}(L^2)} + \int_0^T \| |u(s)|^{\theta -1}
|u(s) - v(s)| \,\|_{L^2}\, ds \cr
&\hphantom{+ \leq C_1(\alpha) C} + \int_0^T \| |v(s)|^{\theta -1}|u(s) - v(s)| \,\|_{L^2}\, ds
\big] 
\end{split}
\end{equation}
where $C=C(f)$. 
We next estimate the term 
$$
B = \int_0^T \| |u(s)|^{\theta -1}
|u(s) - v(s)| \,\|_{L^2}\, ds~.
$$ 
Applying the H\"older inequality, we obtain
\begin{equation}
\label{eq:calFaux2}
B  \leq \int_0^T \| u(s)\|_{L^{2\theta}}^{\theta -1} \| u(s) -v(s)\|_{L^{2\theta}} \, ds~. 
\end{equation}
We next write $2\theta$ as
$2\theta = 2\eta +2(1-\eta) \theta^*$
which is equivalent to
\begin{equation}
\label{eq:eta}
\eta = \frac{d+2 -\theta (d-2)}{4}
\end{equation}
The condition  $1 < \theta <\theta^*$ implies $0 < \eta <1$. Using the above decomposition of $\theta$ in \eqref{eq:calFaux2} together with a H\"older inequality, we get
\begin{equation}
\label{eq:calFaux3}
\begin{split}
B & \leq \int_0^T \| u(s)\|_{L^2}^{\frac{(\theta -1)\eta}{\theta}} 
\| u(s) \|_{L^{2\theta^*}}^{\frac{\theta^*(\theta -1)(1-\eta)}{\theta}} 
 \| u(s)-v(s)\|_{L^2}^{\frac{\eta}{\theta}} 
\| u(s)-v(s) \|_{L^{2\theta^*}}^{\frac{\theta^*(1-\eta)}{\theta}}\,  ds \cr
& \leq \| u(s)\|_{L^{\infty}(L^2)}^{\frac{(\theta -1)\eta}{\theta}} 
 \| u(s)-v(s)\|_{L^{\infty}(L^2)}^{\frac{\eta}{\theta}}
 \int_0^T \| u(s) \|_{L^{2\theta^*}}^{\frac{\theta^*(\theta -1)(1-\eta)}{\theta}} 
 \| u(s)-v(s) \|_{L^{2\theta^*}}^{\frac{\theta^*(1-\eta)}{\theta}} \, ds ~. 
\end{split}
\end{equation}
Applying again the H\"older inequality to the integral term, we obtain,
\begin{equation}
\label{eq:calFaux4}
\begin{split}
 \int_0^T \| u(s) \|_{L^{2\theta^*}}^{\frac{\theta^*(\theta -1)(1-\eta)}{\theta}} 
 \| u(s)-v(s) \|_{L^{2\theta^*}}^{\frac{\theta^*(1-\eta)}{\theta}} \, ds
 \leq &T^{\eta}\Big(\int_0^T \| u(s)-v(s) \|_{L^{2\theta^*}}^{\theta^*} \, ds\Big)^{\frac{1-\eta}{\theta}}\cr
 & \times \Big(\int_0^T \| u(s) \|_{L^{2\theta^*}}^{\theta^*}\, ds\Big)^{\frac{(\theta -1)(1-\eta)}{\theta}} ~.
 \end{split}
\end{equation}
The estimates \eqref{eq:calFaux3} and \eqref{eq:calFaux4} together with the Young inequality give
\begin{equation}
\label{eq:calFaux5}
\begin{split}
B \leq C T^{\eta} M_0^{\frac{\theta -1}{\theta}(\theta^*(1-\eta) +\eta)} 
\big[ \| u-v\|_{L^{\infty}(L^2)} +  \| u-v \|_{L^{\theta^*}(L^{2\theta^*})} \big]~.
\end{split}
\end{equation}
  We next choose $T_0>0$ so that 
 \begin{equation}
\label{eq:defiT0}
C_1(\alpha) C \big[ T_0 
+ 2  T_0^{\eta} M_0^{\frac{\theta -1}{\theta}(\theta^*(1-\eta) +\eta)} \big] = \frac{1}{4}~.
\end{equation} 
The estimates \eqref{eq:calFaux1} to \eqref{eq:calFaux5} imply 
that,  for $0 < T \leq T_0$,
\begin{equation}
\label{eq:calFaux6}
\begin{split}
\| \mathcal{F}\vec u - \mathcal{F} \vec v \|_{Y} \leq C_1(\alpha) C \big[ T 
+ 2  T^{\eta} M_0^{\frac{\theta -1}{\theta}(\theta^*(1-\eta) +\eta)} \big] 
 \| \vec  u- \vec v\|_{Y} \leq \frac{1}{4}  \| \vec u- \vec v\|_{Y}~.
 \end{split}
\end{equation}
{}From the estimates \eqref{eq:calF0} and \eqref{eq:calFaux6}, we   deduce that 
$\mathcal{F}$ is a strict contraction and thus has a unique fixed point  $\vec u \equiv \vec u(\vec u_0)$ in  $Y$
satisfying
\begin{equation}
\label{eq:fixed}
 \| \vec u(\vec u_0)\|_{Y} \leq  C_1(\alpha) \| \vec u_0\|_{\mathcal{H}}
\end{equation}
The fact that $\vec u(t) = (u(t), \partial_t u (t))$ also belongs to $C([0,T], \mathcal{H})$ is standard and left to the reader.  Likewise, we leave it to the reader to verify that the map
$(t, \vec u_0) \in [0,T] \times \mathcal{H} \mapsto  \vec u (t)\in \mathcal{H}$ is  jointly continuous.

 We now turn to property (2). To show that $\vec u_0 \in \mathcal{H} \to \vec u(\tau) \equiv
S_{\alpha}(\tau) \vec u_0 \in \mathcal{H}$ is Lipschitz continuous on the bounded sets of 
$\mathcal{H}$,  we  choose $\vec u_0$ and $\vec v_0$ in the ball $B_{ \mathcal{H}}(0,K_0)$.  Let $T_0>0$ be given by \eqref{eq:defiT0}  and $M_0$ be defined in 
\eqref{eq:M0}. Arguing as above (see the inequality \eqref{eq:calFaux6}), we obtain 
 the following inequality for $0 \leq T \leq T_0$, 
  \begin{equation}
\label{eq}  
 \| \mathcal{F}(\vec u_0, \vec u) - \mathcal{F}(\vec v_0, \vec v) \|_{Y_T}
 \leq  C_1(\alpha) \| \vec u_0 - \vec v_0\|_{\mathcal{H}} + \frac{1}{4}
\| \vec u - \vec v\|_{Y_T}~,
\end{equation}
and thus, the fixed points $\vec u(\vec u_0)$ and $ \vec v(\vec v_0)$ satisfy:
\begin{equation}
\label{eq:Lipuv}
\| \vec u(\vec u_0) - \vec  v(\vec v_0)\|_{Y_T} \leq  \frac{4}{3}C_1(\alpha) 
\| \vec u_0 - \vec v_0\|_{\mathcal{H}}~.
\end{equation} 
If the solutions $ \vec u(\vec u_0)$ and $ \vec v(\vec v_0)$ exist on a time interval $[0, T^*)$, where $T^* >T_0$, 
we repeat the above proof by considering now the ball in $\mathcal{H}$ of center $\vec u(\vec u_0)(T_0)$ and  
radius $K_1>0$ large enough so that $v(\vec v_0)(T_0)$ also belongs to this new ball and replacing the non-linearity $f(.)$ by 
$f(. + u(\vec u_0)(T_0)) - f(u(\vec u_0)(T_0))$.  Repeating this process a finite number of times shows that the map is Lipschitz continuous up to any time~$\tilde T<T^{*}$
and therefore on all of $[0,T^{*})$. 
 The above inequality also implies the uniqueness of the solution of \ref{KGalpha}.

\smallskip

  We next want to show the property (3), namely that the map  $$\vec u_0  \in \mathcal{H} \mapsto u(\vec u_0) \in X \cap L^{\theta^*}((0,T), L^{2\theta^*}(\mathbb{R}^d))$$ is a $C^1$-map. To this end, we will first go back to the mapping 
$$
\mathcal{F} : (\vec u_0, \vec u) \in B_{\mathcal H}(0,K_0) \times  Y \mapsto 
\mathcal{F}(\vec u_0,  \vec u ) \in Y
$$ 
which has been defined by \eqref{eq:definFronde}. And then, for $t \geq T_0$, proceed like in the proof of the property (2). Clearly the map 
$\mathcal{F}(\vec u_0, \vec u)$ is differentiable with respect to the variable $\vec u_0$ since it is a linear map in $\vec u_0$. The differentiability with respect to the variable 
$\vec u \in Y$ is proved as follows (we only indicate the main arguments and leave the details to the reader). Let $\vec h = (h, k) \in Y$ be small. Applying Lemma \ref{lem:free2}, one sees that the proof of the differentiability reduces to proving that
\begin{equation}
\label{eq:differf1}
 \| f(u+h) - f(u) -f'(u) h\|_{L^1((0,T),L^2)} = o(\| \vec  h\|_{Y})~.
\end{equation}
As above, using the hypothesis \ref{H2f}, using the fact that $0 <\beta \leq \frac{2}{d-2}$ and the classical Sobolev embeddings, we may write 
\begin{equation}
\label{eq:differf2}
\begin{split}
\| f(u+h) - f(u) -f'(u) h\|_{L^1((0,T),L^2)} &= \Big \| \int_0^1 (f' (u(s) +\lambda h(s) )
-f'(u(s)))h(s)\, d\lambda \Big\|_{L^1((0,T),L^2)}  \cr
& \leq  C \int_0^T \| (  |h(s)|^{\beta} + |h(s)|^{\theta -1}) 
|h(s)| \,\|_{L^2}\, ds
\cr
& \leq C \big[ T \| h\|_{L^{\infty}(H^1)}^{1+\beta} +  \int_0^T \| |h(s)|^{\theta -1}
|h(s)| \,\|_{L^2}\, ds \big] ~.
\end{split}
\end{equation}
We remark that the last term in the right-hand side of the inequality \eqref{eq:differf2} 
can be estimated as in the inequalities \eqref{eq:calFaux3} and  \eqref{eq:calFaux4}. 
We thus deduce from the inequalities \eqref{eq:calFaux3}, \eqref{eq:calFaux4} and \eqref{eq:differf2} that
\begin{equation}
\label{eq:differf3}
 \| f(u+h) - f(u) -f'(u) h\|_{L^1((0,T),L^2)} = O(\| \vec h\|_{Y}^{1+\delta})~,
\end{equation}
where $\delta= \min (\beta,\eta)$ and $\eta >0$ is defined in \eqref{eq:eta}. 
We thus have proved the property \eqref{eq:differf1}, which implies that 
$\mathcal{F}(\vec u_0, \vec u)$ 
is differentiable with respect to the variable $\vec u \in Y$. 
The derivative of 
$\mathcal{F}(\vec u_0, \vec u)$ with respect to $(\vec u_0, \vec u)$ is given by
$D\mathcal{F}(\vec u_0, \vec u) = (D\mathcal{F}_1,
D\mathcal{F}_2)(\vec u_0, \vec u)$, where $D\mathcal{F}_2(\vec u_0, \vec u)
= \partial_t D\mathcal{F}_1(\vec u_0, \vec u)$ and
\begin{equation}
\label{eq:differf4}
(D\mathcal{F}_1(\vec u_0, \vec u)(\vec v_0, \vec v)) (t) 
= \mathcal{S}_{1,\alpha}(t) v_0 + \mathcal{S}_{2,\alpha}(t) v_1 
+ \int_{0}^{t} \mathcal{S}_{2,\alpha}(t-s) f'(u(s))v(s)\,  ds~.
\end{equation}
We let to the reader to check that this derivative is continuous with respect to 
$(\vec u_0, \vec u)$.
Finally, we remark that, with the choice of the time $T_0$ made in \eqref{eq:defiT0}, 
the mapping $\mathcal{F}(\vec u_0,.): \vec u \in Y_{T} \mapsto
 \mathcal{F}(\vec u_0, \vec u)\in Y_{T}$
is a uniform contraction on $B_{\mathcal H}(0,K_0)$. We may thus apply the
 {\em uniform contraction principle} as stated for example in \cite[Theorem 2.2 on Page 25]{ChowHale}, 
 which implies that $\vec u_0 \in B_{\mathcal H}(0,K_0) \mapsto \vec u(\vec u_0) \in Y_T$ is of class $C^1$. 

\smallskip

We now return to the $H^{2}\times H^{1}$-regularity question, that is, prove the regularity property (5).  Assuming this regularity for now, 
taking a derivative of \ref{KGalpha} yields
\EQ{\label{formder}
\partial_t^2 v+2\alpha  \partial_t v -\Delta v+v-f'(u)v=0
}
where $v$ stands for any of the derivatives $\p_{x_{j}}u$, $1\le j\le d$.  The data for \eqref{formder} belong to $\mathcal{H}$
by assumption. We now perform the same estimates as in \eqref{eq:calFaux1}-\eqref{eq:calFaux5} to conclude that 
\[
\| \vec v\|_{Y}\le C\|(u_{0},u_{1})\|_{H^{2}\times H^{1}} + \f12 \| \vec v\|_{Y},
\]
see especially \eqref{eq:calFaux5}, \eqref{eq:calFaux6}. As above, these estimates require $T$ to be sufficiently small. 
To be precise, the smallness here is determined by $u$ alone through the constant $M_{0}$, see \eqref{eq:calFaux5}. 
It follows  that 
\[
\| \vec v\|_{Y}\le 2C\|(u_{0},u_{1})\|_{H^{2}\times H^{1}}
\]
which is the desired regularity estimate.  In order to pass from an a priori bound to a regularity statement we follow
a standard procedure involving difference quotients: letting $\vec e_{j}$ be the coordinate vectors in $\R^{d}$ we
define  with $h>0$
\[
v_{j}^{(h)}(x):= h^{-1}(u(x+h\vec e_{j}) - u(x))~.\]
By the argument leading to the a priori estimate we obtain  
\[
\big\| \vec v_{j}^{(h)} \big\|_{Y}\le 2C\|(u_{0},u_{1})\|_{H^{2}\times H^{1}}
\]
uniformly in $h>0$.  Passing to suitable weak limits, we obtain the $H^{1}\times L^{2}$ regularity of the derivatives of $u$,
as desired.
\smallskip

Finally, we  turn to the case of small data. We will only provide a sketch of the main argument. In the hypothesis \ref{H2f}, we can choose $\beta >0$ arbitrarily small. In particular, we choose  $0 < \beta <1$. We recall that, for any $y \in \mathbb{R}$,
\begin{equation}
\label{eq:estimfu}
|f(y)| \leq C ( |y|^{\beta} + |y|^{\theta -1}) |y| \leq C ( |y|^{1+ \beta}  + 
|y|^{\theta^*})~.
\end{equation}
 Proceeding as before, applying Lemma \ref{lem:free2}, using the inequality  
 \eqref{eq:estimfu}, one gets, for $t \geq 0$, 
 \begin{equation*}
\begin{split}
 \|  u\|_{L^{\theta^*,\beta}((0,t),L^{2\theta^*})} + 
 \| e^{\beta s} \vec u \|_{L^{\infty}((0,t), \mathcal{H})}  \leq C  [ & 
 \| (u_{0},u_{1}) \|_{H^{1}\times L^{2}} + 
 \| u^{1+\beta} \|_{L^{1,\beta}((0,t), L^{2})}  \cr
& + \| |u|^{\theta^*}  \|_{L^{1,\beta}((0,t), L^{2})}]~.
\end{split}
\end{equation*}
Applying the H\"older inequality, one deduces from the above inequality that, for 
 $t \geq 0$,
\begin{equation}
\label{eq:petit}
\begin{split}
\|  u\|_{L^{\theta^*,\beta}((0,t),L^{2\theta^*})} 
+ \| e^{\beta s} \vec u \|_{L^{\infty}((0,t), \mathcal{H})}   \leq C  
[ & \| (u_{0},u_{1}) \|_{H^{1}\times L^{2}} 
+ \| e^{\beta s}\vec u \|_{L^{\infty}((0,t), \mathcal{H})}^{1 +\beta} \\
& + \| u \|_{L^{\theta^*,\beta}((0,t),L^{2\theta^*})}^{\theta^*} ]~, 
\end{split}
\end{equation}
where we used that $\beta>0$. For small data the method of continuity implies global existence and smallness of the norms on the left-hand side. 
In particular, we have exponential convergence to zero in the energy (see also \cite{Keller83b}).
 \end{proof}

  In Section 3, we will linearize the equation \ref{KGalpha} around an equilibrium point. More generally, we can linearize the Klein-Gordon equation \ref{KGalpha} along any solution of the equation \ref{KGalpha}. This leads us to consider the following affine equation 
\begin{equation}
\label{eq:fprimeuwLin}
w_{tt} + 2\alpha w_t - \Delta w +w -f'(u^*(t,x)) w = G~, \quad 
(w,w_t)(0) \equiv \vec w(0) =  \vec w_0 \in\mathcal{H}~,
\end{equation}
where  $u^*(t,x) \in X_{\tau_0}\cap L^{\theta^*}((0,\tau_0), L^{2\theta^*}(\mathbb{R}^d))$, $\tau_0 >0$, and 
$G \in L^{1}((0,\tau_0), L^2( \mathbb{R}^d))$. The existence (and uniqueness) of a solution
 $\vec w \equiv (w,\partial_t w) \in
C([0,\tau_0), \mathcal{H})$ is classical if the dimension $d$ is equal to $1,2$. So we will state 
this existence result and the corresponding Strichartz estimates only in the case where 
$d \geq 3$.

\begin{proposition} \label{AStrichphi0}
Let $d \geq 3$ and  $\alpha \geq 0$. Assume that $u^*(t,x) \in X_{\tau_0} \cap L^{\theta^*}((0,\tau_0), L^{2\theta^*}(\mathbb{R}^d))$  and that 
$G \in L^{1}((0,\tau_0), L^2( \mathbb{R}^d))$. Then the equation
\eqref{eq:fprimeuwLin} admits a unique solution $\vec w \equiv (w,\partial_t w) \in C([0,\tau_0), \mathcal{H})$. Moreover, the solution $\vec w$ of \eqref{eq:fprimeuwLin} satisfies the following bound, for $0 \leq \tau < \tau_0$,
\begin{equation}
\label{eq:AStrichphi0}
 \| \vec w \|_{L^{\infty}((0,\tau), \mathcal{H})} + \| w \|_{L^{q}((0,\tau), L^p_x)}
  \leq C(\alpha ,\tau) \big [ \| \vec w_0 \|_{\mathcal{H}}  +   \| G\|_{L^{1}((0,\tau), L^2_x)}   \big]~,
\end{equation}
where $$\f{1}{q}+\f{d}{p}=\f{d}{2}-1,\quad 2\le p <\infty,\quad q \geq 2,$$
and 
$\f1q+\f{d-1}{2p}\le\f{d-1}{4}$.
The constant $C(\alpha, \tau) \equiv C(\alpha, \tau, u^*) \geq 1$  depends only on $\alpha$, 
$\tau$ and the norm of $u^*$ in the space $X_{\tau} \cap L^{\theta^*}((0,\tau), 
L^{2\theta^*}(\mathbb{R}^d))$.\\
If $u^*$, $G$ and the initial data are radial functions, then $\vec w$ is a radial solution.
\end{proposition}

\begin{proof} This proposition can be proved in the same way as Theorem \ref{thm:WP}, by    considering the term $f'(u^*(t,x)) w+ G$ as a non-linearity. The changes are  minor in the fixed point argument used in the proof of Theorem \ref{thm:WP}. Here $Y$ and 
$\mathcal{F} = (\mathcal{F}_1,\mathcal{F}_2) = (\mathcal{F}_1, 
\partial_t \mathcal{F}_1)$ simply become:

$$ 
Y \equiv  Y_{T} \equiv \{ \vec w \in L^{\infty}((0,\tau_0), \mathcal{H}) \hbox{ with } 
w \in L^{\theta^*}((0,\tau_0),L^{2\theta^*}(\mathbb{R}^d)) \}~.
$$
and 
$$
(\mathcal{F}_1(\vec w_0, \vec w ))(t) = \mathcal{S}_{1,\alpha}(t) w_0 
+ \mathcal{S}_{2,\alpha}(t) w_1 + \int_{0}^{t} \mathcal{S}_{2,\alpha}(t-s) (f'(u^*(s)) w(s) + G(s))\,  ds~.
$$
We obtain estimates similar to \eqref{eq:calFaux6}, where now $M_0$ is replaced by the 
norm of $u^*$ in $X_{\tau} \cap L^{\theta^*}((0,\tau), L^{2\theta^*}(\mathbb{R}^d))$. 
If the time $T_0$ defined in \eqref{eq:defiT0} is larger than $\tau_0$, then we have proved the existence (and uniqueness) of the solution $\vec w(\vec w_0) \in Y_T$ and the estimates 
\eqref{eq:AStrichphi0} follow from Lemma \ref{lem:free2}. If $T_0 < \tau_0$, we repeat the above proof by taking as initial data $(\vec w(\vec w_0))(T_0)$ and by replacing
$$f'(u^*(t,x)) w(t,x)+ G(t,x)$$ by $$f'(u^*(t + T_0,x))w(t +T_0,x) + G(t+T_0,x)$$
We repeat this argument a finite number of times till we reach the time $\tau_0$.
  \end{proof}

 
\subsection{Definition of the functional $K_0$ and the Nehari manifold}
\label{sec:Nehari}

We introduce the functional $K_0: \varphi \in  H^1(\mathbb{R}^d) \mapsto K_0(\varphi) \in \mathbb{R}$, defined by
$$
K_0 (\varphi) = \int_{\mathbb{R}^d}( |\nabla \varphi |^2 + \varphi^2 - \varphi f(\varphi)) \, dx~,
$$
and introduce the Nehari manifold
\begin{equation}
\label{Nehari}
 \mathcal{N} = \{ \varphi \in H^1_{rad}({\bf R}^d) \, | \, K_0(\varphi) =0 \}~.
\end{equation}
The Nehari manifold arises naturally in the study of elliptic equations.
The ``Ambrosetti-Rabinowitz" hypothesis $(H.1)_f$ allows to prove the following lemmas, which will be used along this paper. The first one is trivial. 
\begin{lemma}
\label{SolBorne}
Assume that Hypothesis $(H.1)_f$ holds. 
Then, for any $(\varphi, \psi)\in H^1(\mathbb{R}^d) \times L^2(\mathbb{R}^d)$, we have
\begin{equation}
\label{eqSolBorne}
\gamma ( \| \varphi\|_{H^1}^2 + \| \psi\|_{L^2}^2 ) \leq 2(1 + \gamma) E((\varphi, \psi)) -K_0(\varphi)~. 
\end{equation}
\end{lemma}
\begin{proof}
We simply write
\begin{equation}
\label{eqEgalK0E}
\begin{split}
\gamma ( \| \varphi\|_{H^1}^2 + \| \psi\|_{L^2}^2 ) = & 2(1 + \gamma) E((\varphi, \psi)) 
- K_0(\varphi) - (1 + \gamma) \| \psi\|_{L^2}^2 \cr
 &+ \int_{\mathbb{R}^d} \big(2(1+ \gamma) F(\varphi) - \varphi(x) f( \varphi(x))\big) \, dx \cr
 \leq &  2(1 + \gamma) E((\varphi, \psi)) - K_0(\varphi)~,
\end{split}
\end{equation}
where the integral is nonpositive by \ref{H1f}. 
\end{proof}

\begin{cor}\label{lem:u global}
Suppose $\vec u(t)= (u(t), \partial_{t} u(t))$ is a strong solution of \ref{KGalpha} defined on the maximal interval $0\le t< T^*$. Assume 
$$
\inf_{0\le t<T^*} K_0(u(t))> -\infty ~.
$$ 
Then $T^*=\infty$, i.e., the solution is global. 
\end{cor}
\begin{proof}
By Lemma~\ref{SolBorne}, we have for some finite $M$ and all $0\le t<T^*$   
\EQ{ \nn
\|\vec u(t)\|_{\calH}  &\le 2(1+\gamma)E(u(t), \partial_t u (t))  + M \\
&\le 2(1+\gamma)E(u(0), \partial_t u (0))  + M
}
where the second line holds by the decrease of the energy.  Since finite time blowup means that $\|\vec u(t)\|_{\calH} $ goes to infinity in finite time along some subsequence,
we obtain the result. 
\end{proof}

The proof of the next lemma uses a convexity argument and follows the lines of the proof of \cite{PaSat} and
\cite[Corollary 2.13]{NaS}.   
We denote the nonlinear evolution by $S_\al(t)$. 

\begin{lemma}
\label{eqBlowup}
Assume that the hypotheses $(H.1)_f$ and $(H.2)_f$ hold.
Assume that $(u(t),  \partial_t u (t))$ is a solution of \ref{KGalpha} defined on $[0,T^*)$ where $T^*\in (0,\I]$ is maximal. If  
$K_0(u(t)) \leq -\delta$ (where $\delta >0$), for $t_0\le t <T^*$, then $T^*<\infty$, i.e.,  the solution blows up in finite time.
\end{lemma}
{}From Lemmas~\ref{SolBorne} and~\ref{eqBlowup} we immediately  deduce the following result.

\begin{coro}\label{coro-blow}
Assume that the initial energy $E(\vec u_0) $ is negative. Then the solution blows-up in finite time $T^* <+\infty$.
\end{coro}

\begin{proof}[Proof of Lemma~\ref{eqBlowup}] 
We assume without loss of generality that $t_0=0$. We also assume towards a contradiction that $T^*=\infty$. 
In order to show that $S_\alpha(t)(u_0,u_1)$ blows up in finite time, we use a convexity argument as in~\cite{PaSat}. Assume that $S_\alpha(t)(u_0,u_1)$ exists for all $t\geq0$ and let
$$y(t)=\frac{1}{2}\|u(t)\|_{L^2}^2+\alpha\int_0^t\|u(s)\|_{L^2}^2\,ds.$$
We have
\begin{equation}\label{eq:2.3}
\begin{aligned}
\dot{y}(t)&=(u(t),\dot{u}(t))+\alpha\|u(t)\|_{L^2}^2\\
&=(u(t),\dot{u}(t))+\alpha\|u(0)\|_{L^2}^2+2\alpha\int_0^t(u(s),\dot{u}(s))\,ds
\end{aligned}
\end{equation}
and
\begin{equation}\label{eq:2.4}
\begin{aligned}
\ddot{y}(t)&=\|\dot{u}(t)\|_{L^2}^2+(u(t),\ddot{u}(t)+2\alpha\dot{u}(t))\\
&=\|\dot{u}(t)\|_{L^2}^2+(u(t),(\Delta u-u+f(u))(t))\\
&=\|\dot{u}(t)\|_{L^2}^2-K_0(u(t))~.
\end{aligned}
\end{equation}
Thus,
\begin{equation}
\label{eq:2.6}
\ddot{y}(t)\geq\|\dot{u}(t)\|_{L^2}^2+\delta\geq\delta.
\end{equation}
We deduce from \eqref{eq:2.6} that $\lim_{t\rightarrow + \infty}\dot{y}(t) = + \infty$, and therefore  $\lim_{t\rightarrow + \infty} y(t)=+\infty$.

Next, we note that 
\EQ{ 
\label{eq:2.7}
\ddot{y}(t)&= \|\dot{u}(t)\|_{L^2}^2-K_0(u(t)) \\
& =(2+\gamma)\|\dot{u}(t)\|_{L^2}^2+\gamma \|u(t)\|_{H^1}^2-2(1+\gamma)E(t) \\
&\qquad - \int_{\R^d} \big(2(1+\gamma) F(u(t)) - u(t) f(u(t)) \big)\, dx
}
where we have set for simplicity $E(t)=E((u(t),\dot{u}(t)))$. But, we have
$$\dot{E}(t)=-2\alpha\|\dot{u}(t)\|_{L^2}^2$$
and
$$
E(t)=E(0)+\int_0^t \dot{E}(s)\,ds=E(0)-2\alpha\int_0^t\|\dot{u}(s)\|_{L^2}^2\,ds.
$$
Using \ref{H1f}  and the definition of $y(t)$, we can also write, for $t\geq0$,
\begin{align}
\label{eq:2.8}
\ddot{y}(t)\ge (2+\gamma)\|\dot{u}(t)\|_{L^2}^2+\gamma\|u(t)\|_{H^1}^2-2(1+\gamma)E(0)+4 \alpha (1+\gamma)\int_0^t\|\dot{u}(s)\|_{L^2}^2\,ds.
\end{align}
For the sake of illustration, assume first that $\al=0$.
Since $y(t)\to\infty$, we infer from \eqref{eq:2.8} that for large $t$ 
\EQ{\label{eq:ddoty}
\ddot{y}(t)\ge (2+\gamma)\|\dot{u}(t)\|_{L^2}^2 
}
 Then 
$|\dot y(t)|\le \| u(t)\|_{L^2} \|\dot{u}(t)\|_{L^2}$ whence 
\[
\ddot{y}(t)\ge \f{2+\gamma}{2} \f{\dot y^2(t)}{y(t)}
\]
This implies that $\f{d^2}{dt^2} (y^{-\delta} (t))<0$ where $\delta=\gamma/2$.  Since $y^{-\delta}(t)\to0$ as $t\to\infty$ we 
must have $\f{d}{dt} (y^{-\delta})(t)<0$ for some $t=t_1>0$ whence also  $\f{d}{dt}(y^{-\delta})(t)\le \f{d}{dt} (y^{-\delta})(t_1)<0$ for all $t\ge t_1$. But then $y^{-\delta}(t_2)=0$
for some $t_2>t_1$ which is a contradiction. 

For $\al>0$, we claim that there exists $c>1$ so that for large times
\EQ{\label{eq:c claim}
\ddot y(t) y(t) - c\dot y(t)^2> 0
}
If so, then 
\[
\f{d^2}{dt^2} \big( y^{-(c-1)} \big)(t) = -(c-1) y^{-c-1}(t) (\ddot y (t)y (t)  - c\dot y^2(t) )<0
\]
which leads to a contradiction as before. 

It remains to verify \eqref{eq:c claim}. Using the Cauchy-Schwarz inequality we obtain
\begin{align}
\label{eq:2.10}
& y(t)\ddot{y}(t)-c\dot{y}^2(t)\geq \Bigl(\frac{1}{2}\|u\|_{L^2}^2+\alpha\int_0^t\|u(s)\|_{L^2}^2\,ds\Bigr)\\
&\cdot\Bigl(    (2+\gamma)\|\dot{u}(t)\|_{L^2}^2+\gamma\|u(t)\|_{H^1}^2-2(1+\gamma)E(0)+4 \alpha (1+\gamma)\int_0^t\|\dot{u}(s)\|_{L^2}^2\,ds     \Bigr)\nonumber\\
&-c\Bigl[\|u\|_{L^2}\|\dot{u}\|_{L^2}+2\alpha\Bigl(\int_0^t\|u(s)\|_{L^2}^2\,ds\Bigr)^{\frac{1}{2}}\Bigl(\int_0^t\|\dot{u}(s)\|_{L^2}^2\,ds\Bigr)^{\frac{1}{2}}+\alpha\|u(0)\|_{L^2}^2\Bigr]^2.\nonumber
\end{align}
But, for any $\eps>0$, we estimate the term in brackets as follows: 
\begin{align*}
c&\Bigl[\|u\|_{L^2}\|\dot{u}\|_{L^2}+2\alpha\Bigl(\int_0^t\|u(s)\|_{L^2}^2\,ds\Bigr)^{\frac{1}{2}}\Bigl(\int_0^t\|\dot{u}(s)\|_{L^2}^2\,ds\Bigr)^{\frac{1}{2}}+\alpha\|u(0)\|_{L^2}^2\Bigr]^2\\
&\leq c(1+\eps)\Bigl(\|u\|_{L^2}\|\dot{u}\|_{L^2}+2\alpha\Bigl(\int_0^t\|u(s)\|_{L^2}^2\,ds\Bigr)^{\frac{1}{2}}\Bigl(\int_0^t\|\dot{u}(s)\|_{L^2}^2\,ds\Bigr)^{\frac{1}{2}}\Bigr)^2\\
&\qquad+c \Bigl(1+\frac{1}{\eps}\Bigr)\alpha^2\|u(0)\|_{L^2}^4\\
&\leq c(1+\eps)\Bigl(\frac{1}{2}\|u\|_{L^2}^2+\alpha\int_0^t\|u(s)\|_{L^2}^2\,ds\Bigr)\Bigl(2\|\dot{u}\|_{L^2}^2+4\alpha\int_0^t\|\dot{u}(s)\|_{L^2}^2\,ds\Bigr)\\
&\qquad+c \Bigl(1+\frac{1}{\eps}\Bigr) \alpha^2\|u(0)\|_{L^2}^4.
\end{align*}
Setting  $b=c(1+\eps)$, $C=c\alpha^2(1+\frac{1}{\eps})\|u(0)\|_{L^2}^4$,  we may replace  the right-hand side of this inequality  by
\[
\leq y(t) \Big(2b \|\dot{u}\|_{L^2}^2+4b\alpha\int_0^t\|\dot{u}(s)\|_{L^2}^2\,ds \Big) + C
\]

{} From the last inequality and from \eqref{eq:2.10}, we deduce that
\begin{align}
y\ddot{y}(t)-c\dot{y}^2(t) &\geq y(t)\Big\{ (2+\gamma-2b)\|\dot{u}(t)\|_{L^2}^2+4\alpha(1+\gamma-b)\int_0^t\|\dot{u}(s)\|_{L^2}^2\,ds+\gamma\|u(t)\|_{H^1}^2 \nn \\
&\qquad -2(1+\gamma)E(0)\Big\}-C\nonumber\\
& = y(t)\Psi(t)-C  \label{eq:2.11}
\end{align}
where $\Psi(t)$ is defined by the term in braces. 

We now adjust the constants $c>1$ and $\eps>0$ so that $2+\gamma-2b>0$, $1+\gamma-b>0$.   We now pick $\eta>0$ so small that
\[
2+\gamma-2b> \eta, \qquad  \gamma - \f{\eta}{2} -\al\eta>0
\]
This allows us to bound $\Psi(t)$ from below: 
\begin{align*}
\Psi(t)&=\left[\left(2+\gamma-2b-\frac{\eta}{2}\right)\|\dot{u}(t)\|_{L^2}^2+4\alpha(1+\gamma-b)\int_0^t\|\dot{u}(t)\|_{L^2}^2\,ds+\gamma \|\nabla u(t)\|_{L^2}^2\right.\\
&\qquad\left.+\left(\gamma -\frac{\eta}{2}-\alpha\eta\right)\|u(t)\|_{L^2}^2 +\eta\dot{y}(t)-2(1+\gamma)E(0)\right]\\
&\geq\eta\dot{y}(t)- 2(1+\gamma)E(0)+q(t)
\end{align*}
where $q(t)\geq0$. {}From \eqref{eq:2.11}, we infer that, for $t\geq0$,
\begin{align}
\label{eq:2.12}
y(t)\ddot{y}(t)-c\dot{y}^2(t)\geq y(t)[\eta\dot{y}(t)- 2(1+\gamma)E(0)+q(t)]-C.
\end{align}
Since $y(t), \dot y(t)\to\infty$ as $t\to\infty$, we are done. 
\end{proof}

\subsection{Spectral properties}

Suppose we have a  stationary solution $\fy_0\in H^1(\R^d)$ to \ref{KGalpha}, namely,
\[
-\Delta \varphi_0 + \varphi_0  - f(\varphi_0)=0
\]
By elliptic theory, see for example~\cite{BeLions83I, BeLions83II}, these solutions are exponentially decaying, and lie in $C^{3,\beta}$ for some $\beta>0$. 
Solving \ref{KGalpha} for $u=\fy_0+v$ yields
\EQ{\label{eq:vN}
v_{tt} + 2\alpha v_t -\Delta v + v -f'(\fy_0) v = N(\fy_0,v)
} 
where $N(\fy_0,v) = f(\fy_0+v)-f(\fy_0)-f'(\fy_0) v$. Set $\calL=-\Delta + I - f'(\fy_0)$. 
Rewrite \eqref{eq:vN} in the form
\EQ{\label{Sys1}
\p_t\binom{v}{v_t} = \left(\begin{matrix} 0 & 1\\ - \calL & -2\alpha\end{matrix}\right) \binom{v}{v_t} + \binom{0}{N(\fy_0,v)}
}
Denoting the matrix operator on the right-hand side by $A_\alpha$, and setting $\vec v:=\binom{v}{v_{t}}$, we may write~\eqref{Sys1} in the form
\[
\p_{t}\vec v = A_{\alpha}\vec v + \vec N
\]
The spectral properties of $\calL$ stated in the following lemma are standard, see for example~\cite{IS} and the references cited here. 

\begin{lemma}\label{lem:Lspec}
The operator $\calL$ is self-adjoint with domain $H^{2}(\R^{d})$. The spectrum $\sigma(\calL)$ consists of an essential part $[1,\infty)$, which is absolutely continuous,  
and finitely many eigenvalues of finite multiplicity all of which fall into $(-\infty, 1]$. The eigenfunctions are $C^{2,\beta}$ with $\beta>0$ and the ones associated with eigenvalues below $1$ are exponentially decaying.  Over the radial functions, all eigenvalues are simple. 
\end{lemma}
\begin{proof}
The essential spectrum equals $[1,\infty)$ by the Weyl criterion. The Agmon-Kuroda theory on asymptotic completeness guarantees that 
there are no imbedded eigenvalues and no singular continuous spectrum. Thus, the spectral measure restricted to $[1,\infty)$ is purely absolutely continuous. 
The Birman-Schwinger criterion shows (due to the rapid decay of the potential $f'(\fy_{0})$) that there are only finitely many eigenvalues of $\calL$ which are $\le 1$, 
counted with multiplicity. 
The $C^{2,\beta}$ property of the eigenfunctions is standard elliptic regularity (Schauder estimates) since $\fy_{0}$ is smooth, and so $f'(\fy_{0})$ is H\"older regular. 

For the sake of completeness we remark that the threshold $1$ may be an eigenvalue or a resonance.  To illustrate what this means, consider $\R^{3}$. Then this distinction refers to the to fact that solutions to $\calL \psi=\psi$ either decay like $|x|^{-2}$ (which means $\psi\in L^{2}$ is an eigenfunction) 
or like $|x|^{-1}$, the latter implying that $\psi\not\in L^{2}(\R^{3})$ (this is the resonant case). We remark that over the radial functions only the resonant case 
can occur. However, none of this finer analysis at energy $1$ is  relevant for our purposes. 

The exponential decay of the eigenfunctions with eigenvalues below $1$ is known as Agmon's estimate.
The simplicity of the radial eigenfunctions is immediate from the reduction to an ODE on $(0,\infty)$ with a Dirichlet condition at $r=0$. 
Let us elaborate on the kernel of $\calL$, since it is important in our construction. We set 
$\mathcal{L} v =0$, $v \ne0$ radial and in $H^{1}$. Then 
 $$
 -\Delta v + v - f'(\varphi_0) v=0
 $$
 We already notes that $v \in C^{2,\beta}(\R^{d})$, and that $v(r)$ decays exponentially. 
 Set $u(r)=r^{\f{d-1}{2}}v(r)$. Then $u(0)=0$, $u(r)\to0$ as $r\to\infty$ (exponentially in fact), and it satisfies the equation 
 \EQ{\label{eq:uode}
 -u''(r) + u(r) - (\frac{d-1}{2})(\frac{d-3}{2}) \frac{u(r)}{r^2} -  f'(\varphi_0) u(r)=0~,
 \quad r>0
 }
 This ODE has a fundamental system consisting of a solution growing like $e^{r}$ and one decaying like $e^{-r}$
 as $r\to\infty$.  Only the latter can lie in the kernel and it does so if and only if it satisfies the boundary
 condition $u(0)=0$. In this case the kernel has dimension~$1$ otherwise it consists only of $\{0\}$. 
\end{proof}

We now analyze the spectral properties of the matrix operator $A_{\alpha}$. 

\begin{lemma} \label{lem:spectre2A} 
\begin{itemize}
\item The  operator $A_\alpha$ has discrete spectrum if and only if $\mathcal{L}$ does. The essential spectrum of $A_{\al}$ lies
strictly to the left of the imaginary axis, i.e., in $\Re(z)<-\delta(\alpha)$ for some 
$\delta(\alpha)>0$.  The  spectrum of $A_{\al}$ on the imaginary axis is either empty 
or~$\{0\}$.  
In the latter case,  $0$ is an
eigenvalue of $A_\alpha$ and this occurs if and only if $0$ is an eigenvalue of $\calL$. Then $\dim(\ker(\calL))=1$, in which case $0$ is a simple eigenvalue. The  eigenvalues of $A_\alpha$ are precisely
\[
-\al \pm \sqrt{\al^2 - \mu}
\]
where $\mu \in \sigma(\calL)$  is an eigenvalue.  
\begin{itemize}
\item If $\alpha \ge 1$, then the discrete spectrum of $A_\alpha$ lies only on the real axis.
\item If $0<\alpha<1$, in addition to real eigenvalues, there may also be eigenvalues on the line $\Re(z)=-\alpha$ resulting from eigenvalues of $\calL$
in the gap $(0,1]$. 
\end{itemize}

\item The essential spectrum of $\calL$ gives rise to essential spectrum $\sigma_{\mathrm{ess}}(A_\al)$ of $A_{\al}$ as follows: 
\begin{itemize}
\item 
If  $0<\al\le 1$,  $\sigma_{\mathrm{ess}}(A_\al)$ is contained in  the line 
$\Re(z)=-\alpha$  and 
consists of $-\al\pm i \beta$, $\beta \ge \sqrt{1-\al^2}$.  
\item If  $\alpha >1$,  $\sigma_{\mathrm{ess}}(A_\al)$ consists of the entire line  $\Re(z)=-\alpha$ and of the interval $$[-\alpha - \sqrt{\alpha^2 - 1},
-\alpha + \sqrt{\alpha^2 - 1}]$$
\end{itemize}
\end{itemize}
\end{lemma}
\begin{proof}
We need to address the solvability of the system 
\EQ{
\nn
A_\alpha \binom{u_1}{u_2} = z \binom{u_1}{u_2}
}
over the domain 
 $H^2_{rad}(\R^d) \times H^1_{rad}(\R^d)$  of $A_\alpha$.  This means that
\EQ{\nn
u_2 &= z u_1 \\
-\calL u_1 - 2\alpha u_2 &= z u_2
}
which is the same as 
\EQ{\nn
u_2 &= z u_1 \\
(\calL  + 2\alpha z +z^2)  u_1 &= 0
}
There exists a solution in the domain of $A_\alpha$ if and only if 
\[
2\alpha z + z^2  \in \sigma(-\calL)
\]
Taking $\lambda \in \sigma(\calL)$, this means that 
\EQ{\label{eq:z al}
z = -\alpha \pm \sqrt{\alpha^2 - \lambda},\quad \lambda \in \sigma(\calL)~.
}
This relation establishes all the claims concerning the point spectrum of $A_{\alpha}$. 
 Let now $\tau$ belong to the resolvent set $\rho(A_\alpha)$ of $A_\alpha$.
 Then, for any $(0,v_2) \in \mathcal{H}_{rad}$, the system  
 \begin{equation}
\label{resoluble}
(A_\alpha - \tau Id) \binom{u_1}{u_2} =  \binom{0}{v_2}
\end{equation}
 has a unique solution $(u_1,u_2)$ in $\mathcal{H}_{rad}$, which implies that
 \begin{equation*}
-\mathcal{L} u_1 - (\tau^2 +2\alpha \tau) u_1 =  v_2
\end{equation*}
has a unique solution $u_1$ and thus $\tau^2 +2\alpha \tau \equiv -\lambda$ does not belong to the spectrum of $-\mathcal{L}$, that is, 
$$\tau \ne -\alpha \pm \sqrt{\alpha^2 - \lambda},\quad 
 \lambda \in  \sigma(\calL)$$ 
 and we are done. 
  \end{proof}
 

The discrete spectrum of $A_{\al}$ (and therefore of $\calL$) is  important to our analysis.
In fact, the strongly unstable manifold  of the linear evolution $e^{tA_{\al}}$ 
 as $t\to\infty$  corresponds exactly
to spectrum of $A_{\al}$ in the right-half plane which occurs if and only  if $\calL$ exhibits
negative eigenvalues.  In the generality we assume here we cannot determine whether this is the case or not,
and so our arguments need to be flexible enough to account for both possibilities. 

However, consider the following additional condition,  where $\gamma$ is as in \ref{H1f}: 
for any $\phi\in H^1$, 
\EQ{\label{H5f}
 \int_{\R^d} \big[  \phi^2(x)f'(\phi(x))  - (1+2\gamma) \phi(x) f(\phi(x))  \big] \, dx \ge 0
}
Let  $\varphi_0\ne0$ be a stationary solution as before.  Then it follows from \eqref{H5f} that 
 \EQ{\label{eq:L+neg}
 \LR{\calL \varphi_0,\varphi_0} &= \int_{\R^d} ( |\nabla \varphi_0|^2 + \varphi_0^2 - f'(\varphi_0)\varphi_0^2)\,  dx  \\
 & = -2\gamma \int_{\R^d} f(\varphi_0)\varphi_0\,  dx  + \int_{\R^d} [(1+2\gamma)f(\varphi_0)\varphi_0 - f'(\varphi_0)\varphi_0^2]\,  dx  \\
 & \leq  -2\gamma \|\varphi_0\|_{H^1}^2 <0
 }
where we used that $K_0(\varphi_0)=0$. Therefore, $\calL$ has negative eigenvalues. 
We leave it to the reader to check that the class of nonlinearities $f$ given by a sum and difference of pure powers as in~\eqref{fExample2} satisfy~\eqref{H5f}.  Hence, for such nonlinearities all nonzero 
stationary solutions are linearly unstable. In other words, under the additional condition 
\eqref{H5f} all nonzero equilibria give rise to a strongly
unstable manifold of $e^{tA_{\al}}$. 

\begin{figure}[ht]
\begin{center}
\includegraphics[width=0.55\textwidth]{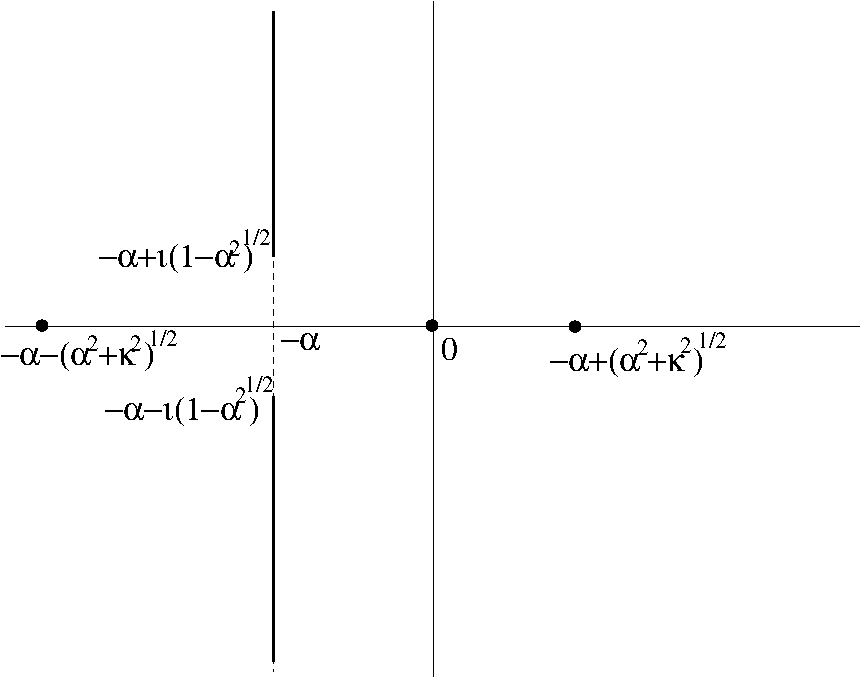}
\end{center}
\caption{The spectrum of $A_\al$ for $0<\al<1$}\label{fig:1.1} 
\end{figure}
\section{Proof of Theorem \ref{ThBRS1}}
\label{sec:core}

In this section, we are going to prove  Theorem \ref{ThBRS1}. To this end, given 
$(\varphi_0,\varphi_1) \in \mathcal{H}_{rad}$, we will first show that, if 
$S_{\alpha}(t)(\varphi_0,\varphi_1)$ does not blow up in finite time, then there exists a sequence of times $t_n$ going to $+\infty$ such that $S_{\alpha}(t_n)(\varphi_0,\varphi_1)$ converges to an equilibrium point $(Q^*,0)$. 

\subsection{Convergence to an equilibrium $(Q^*,0)$ along a subsequence} \label{sec:seq}
Denote the evolution operator of \ref{KGalpha} by $S_{\alpha}(t)$ and for $(\varphi_0,\varphi_1) \in \mathcal{H}_{rad}$, let $\vec u(t) := S_{\alpha}(t)(\varphi_0,\varphi_1)$.  
We have the following trichotomy for the forward evolution of \ref{KGalpha}:
\begin{enumerate}
\item[(FTB)]  $\vec u(t)$ blows up in finite positive time.  
\item[(GEB)]   $\vec u(t)$ exists globally and the trajectory 
$\{\vec u(t), t\geq 0\}$ is bounded in $\mathcal{H}_{rad}$,
\item[(GEU)]   $\vec u(t)$ exists globally and the trajectory 
$\{\vec u(t), t\geq 0\}$ is unbounded in $\mathcal{H}_{rad}$. 
 \end{enumerate}
 
 \begin{remark}\label{rem:3}
\upshape
Several remarks have to be made at this stage.
\begin{description}
\item[(i)] From Corollary~\ref{coro-blow}, we know that if $E(\varphi_0, \varphi_1) <0$, then $S_{\alpha}(t) (\varphi_0,\varphi_1)$  blows up in finite time. Thus, in the study of the cases (GEB) and (GEU), we only need to consider solutions $\vec u(t) \equiv S_{\alpha}(t) (\varphi_0,\varphi_1)$ such that, for any $t \geq 0$,
\begin{equation}
\label{Epositive}
E(u(t), \partial_t u (t)) \geq 0~.
\end{equation}

\item[(ii)] Assume now that a solution $\vec u(t) \equiv S_{\alpha}(t) (\varphi_0,\varphi_1)$ of
\ref{KGalpha} satisfies the properties \ref{H1f}, \ref{H2f} and  \eqref{Epositive}. Assume moreover, that the exponent $\theta$ in \ref{H2f} satisfies the bound
\begin{equation}
\label{bornetheta}
 \theta < 1 + \frac{4}{d}~.
\end{equation}
 Then, arguing exactly as in \cite[Lemma 4.2]{Feir98A}, one can prove that every global solution $S_{\alpha}(t)(\varphi_0,\varphi_1)$ is bounded in $\mathcal{H}$. In this proof, the upper bound \eqref{bornetheta} of $\theta$ plays a crucial role.

\item[(iii)] Now, let us turn to the case where $1 + \frac{4}{d} \leq \theta 
\leq \frac{d}{d-2}$.  We consider a global solution $(u(t), \partial_t u (t))= S_{\alpha}(t) (\varphi_0,\varphi_1)$. In this case, arguing as in \cite[Page 59]{Feir98A} by introducing the auxiliary equation satisfied by $\partial_{t} \vec u(t):=( \partial_t u (t),  \partial_t^2 u(t))$, one  shows that 
 $\partial_{t} \vec u(t)$ converges to $(0,0)$ in  $L^{2}(\mathbb{R}^d) \times 
 H^{-1}(\mathbb{R}^d)$. From this convergence property, we   deduce that 
$K_0(u(t))$ converges to $0$ as $t$ goes to infinity.
\end{description}
\end{remark} 
 We first make a simple observation concerning the case (GEU).  Later in Section~\ref{sec:conv}, we shall show that (GEU) cannot occur. 
 \begin{lemma}\label{lem:GEU seq} 
 Assume  that the hypothesis $(H.1)_f$ and $(H.2)_f$ hold.
 In the case (GEU), we may assume that there exist a sequence of times $t_n$ and a sequence of numbers $\delta_n$, $\delta_n \leq 0$, such that $t_n \rightarrow+\infty$  
as $n \rightarrow +\infty$ and that
\begin{equation}
\label{K0n0}
K_0(u(t_n))  =\delta_n~, \hbox{ with } \lim_{n\rightarrow +\infty}\delta_n = 0~.
\end{equation}
\end{lemma}
\begin{proof}
If $K_0(u(t))\ge0$ for all times $T_0<t$, then the trajectory is bounded by Lemma~\ref{SolBorne}. 
So there exists a sequence $\tau_n\to\infty$ with $K_0(u(\tau_n))<0$.  If $K_0(u(t))<-\delta$ for all times $T_0<t<\infty$, where $\delta>0$ is fixed,  then
by Lemma~\ref{eqBlowup} finite time blowup occurs. This contradicts (GEU). 
\end{proof}
For the case (GEB) we shall now also construct such a sequence, albeit with $\delta_n= K_0 (u) (t_n)$ possibly being positive. 
\begin{proposition}\label{lem:GEB seq}
In the case (GEB) there exists a sequence $t_n\to\I$ with $K_0(u(t_n))\to0$ and $\|  \partial_t u (t_{n})\|_{L^2}\to 0$ as $n\to\infty$. 
\end{proposition}
\begin{proof}
Taking the inner product in $L^2$ of the equation \ref{KGalpha} with $u$ and integrating by parts yields
\begin{align}
\label{eq:6}
\int\left(|\nabla u|^2+u^2-f(u)u\right)\,dx+2\alpha\int u \partial_t u \,dx=\int  (\partial_t u) ^2\,dx-\frac{d}{dt}\int  \partial_t u u\,dx.
\end{align}
Notice that for smooth (or, more precisely, $H^2\times H^1$) initial data, this integration by parts is justified. Moreover,  for $H^1\times L^2$ initial data, we conclude that the map 
$$ (\varphi_0, \varphi_1 ) \in H^1 \times L^2  \mapsto  \int  \partial_t u u\,  dx \in \mathbb{R}
$$ is $C^1$ with derivative given by~\eqref{eq:6}. In the sequel, we shall take~\eqref{eq:6} as a definition for 
$$ \frac{d}{dt}\int  \partial_t u u\,dx.$$
We wish to choose a sequence $t_n\to+\infty$ so that each term on the right-hand side of \eqref{eq:6}, when evaluated at $t_n$, tends to $0$ as $n\to+\infty$. First, we rewrite \eqref{eq:6} as
\begin{align}
\label{eq:7}
\int_{\mathbb{R}^d}\left(|\nabla u|^2+u^2+ (\partial_t u) ^2-f(u)u\right)\,dx+2\alpha\int_{\mathbb{R}^d}u \partial_t u \,dx=2\int_{\mathbb{R}^d} (\partial_t u) ^2\,dx-\frac{d}{dt}\int  \partial_t u u\,dx.
\end{align}
We now make the following {\bf Claim:} 

{\em 
There exists a sequence $t_n\to+\infty$ such that
\begin{align*}
\begin{cases}
\lim_{n\rightarrow +\infty} \| \partial_t u (t_n)\|_{L^2}=0 \\
\lim_{n\rightarrow + \infty} 2\| \partial_t u (t_n)\|_{L^2}^2-\frac{d}{dt}\int( \partial_t u u)(t_n,x)\,dx =0.
\end{cases}
\end{align*}
}

Since the energy is nonnegative, we have that for any $0<T<\tilde{T}$,
\begin{align}
\nn
2\alpha\int_T^{\tilde{T}}\| \partial_t u (s)\|_{L^2}^2\,ds= E(\vec{u}(T))- E(\vec{u}(\tilde{T}))\leq K= E(\vec {u} (0))~,
\end{align}
and
\begin{align}
\nn
\left|\int_T^{\tilde{T}}\frac{d}{dt}\int( \partial_t u u)(s,x)\,dxds\right|=\left| \LR{ \partial_t u (\tilde{T}),u(\tilde{T})} - \LR{ \partial_t u (T),u(T)}\right|\leq\tilde{K}~,
\end{align}
where $\tilde{K}>0$ is independent of $T$, and $\tilde{T}$.
We now distinguish  two cases:

{\bf Case 1}: If there exists $T_0>0$ such that, for $t\geq T_0$, $$2\| \partial_t u (t)\|_{L^2}^2-\frac{d}{dt} \LR{ \partial_t u (t),u(t)}$$ does not change sign (and, for example, is nonnegative), then for any $T_0\leq T<\tilde{T}$,
\begin{align*}
\int_T^{\tilde{T}}&\left(\| \partial_t u (t)\|_{L^2}^2+\left|2\| \partial_t u (t)\|_{L^2}^2-\frac{d}{dt} \LR{ \partial_t u (t),u(t)} \right|\right)dt\\
&=\int_T^{\tilde{T}}\left(\| \partial_t u (t)\|_{L^2}^2+\left(2\| \partial_t u (t)\|_{L^2}^2-\frac{d}{dt} \LR{ \partial_t u (t),u(t)} \right)\right)dt\leq2K+\tilde{K}~.
\end{align*}
This allows us to show that there exists a sequence $t_n\to+\infty$, such that
\begin{align}
\label{eq:10}
\| \partial_t u (t_n)\|_{L^2}^2+\left| 2\| \partial_t u (t_n)\|_{L^2}^2-\frac{d}{dt}|\LR{ \partial_t u (t_n),u(t_n)} \right|\to0\quad\text{ as }t_n\to+\infty.
\end{align}

{\bf Case 2}: There exists a sequence of times $\tau_m\to+\infty$ such that
\begin{align}
\nn
A(\tau_m):=2\| \partial_t u (\tau_m)\|_{L^2}^2-\frac{d}{dt} \LR{ \partial_t u (\tau_m),u(\tau_m)}=0.
\end{align}
To conclude, we need 
\begin{lemma}\label{lem.unif}
There exists a subsequence $\tau_{m_j}$ and $\eta_0 >0$ such that the function $A(t)$ is uniformly continuous on 
$$\mathcal{I} = \bigcup_{\tau_{m_j}}[\tau_{m_j} -\eta_0,\tau_{m_j} +\eta_0].$$
\end{lemma}
\begin{proof}
We write 
\begin{align}
\label{eq:12}
A(t)=2 E(\vec{u}(t)) + \int_{\mathbb{R}^d} \big[2F(u(t,x)) - f(u)u(t,x)\big]\,dx+2\alpha \mu(t)
\end{align}
where $\mu(t)=\LR{u(t), \partial_t u (t)}$.
Since $E(\vec{u}(t))$ is continuous and has a limit as $t\to+\infty$, $E(\vec{u}(t))$ is uniformly continuous on $[0,+\infty)$. 
Since  $\tau_m \to +\infty$, there exist a subsequence (that we still denote $\tau_m$ for ease of notation)  and  $e \in H^1_{rad}(\R^d)$ such that $u(\tau_m)$ converges weakly to $e$ in $H^1(\R^d)$ as 
$m$ goes to infinity. Thus, using the fact that the injection $H^1_{rad}(\R^d)
 \to L^p(\R^d)$, 
$2 < p< 2^*$, is compact, we deduce that   $u(\tau_m)$ converges strongly  to $e(x)$ in $L^p(\R^d)$ as $\tau_m$ goes to infinity. Furthermore, since $t\mapsto\int f(u)u(t,x)\,dx$ is continuous 
and \[ \int \big[2F(u(\tau_m,x)) -  f(u)u(\tau_m,x)\big]\,dx \to  
\int \big[2F(e(x)) -  f(e)e(x)\big]\,dx  
\]
 as $\tau_m\to+\infty$, we obtain the uniform continuity on $\mathcal{I}$ of the middle term in~\eqref{eq:12}. 

Integration by parts shows that 
\begin{equation*}
\label{eq:13}
\mu(t+\delta)=\mu(t)+\int_{t}^ {t+ \delta} \int _{\mathbb{R}^d}\left( (\partial_t u) ^2-2\alpha(u \partial_t u )-|\nabla u|^2-u^2+f(u)u\right)\,dxds 
\end{equation*}
Since $\vec u(t) = S_{\alpha}(t) (\varphi_0,\varphi_1)$, $t \geq 0$,  is bounded in $\mathcal{H}$,  we deduce that $\mu(t)$ is uniformly continuous on $[0,+\infty)$.
\end{proof}
Now, the construction of a sequence $t_n\to+\infty$ such that \eqref{eq:10} holds
 follows by a standard inductive procedure. Indeed, assume that we have constructed a sequence $\{ t_1< \dots< t_N \} $ such that 
 $$ \forall 1\leq n \leq N, \quad \| \partial_t u (t_n) \|_{L^2} \leq 2^{-n},\quad |A(t_n )| \leq 2^{-n}.$$
 Let $\epsilon = 2^{-(N+1)}$.  Since $A(\tau_{m_j})=0$, according to Lemma~\ref{lem.unif} there exists $\eta>0$ such that for  
 any $t\in [\tau_{m_j}- \eta, \tau_{m_j}+ \eta]$ one has  $$|A(t) |\le 2^{-(N+1)}.$$  Then, since 
 $$ \lim_{j\rightarrow + \infty} \int_{\tau_{m_j}- \eta}^{ \tau_{m_j}+ \eta} \|\partial_t u\|_{L^2}^2 (s) \, d s =0,$$
 we obtain that for $j$ large enough, there exists $s_j\in  [\tau_{m_j}- \eta, \tau_{m_j}+ \eta]$ such that 
 $$ \|\partial_t u (s_j) \|_{L^2}\leq 2^{-(N+1)}.$$
 Choosing  $t_{N+1}=s_j$ for $j$ large enough ensures that $t_N +1  < t_{N+1}$ and  
 $$ \| \partial_t u (t_{N+1}) \|_{L^2} \leq 2^{-(N+1)}, \quad |A(t_{N+1} )| \leq 2^{-(N+1)}.$$
 From the {\em Claim} above,  and \eqref{eq:7}, we deduce that 
\begin{align}
\nn
\lim_{n\to+\infty}\int_{\mathbb{R}^d}\left(|\nabla u|^2+u^2-f(u)u\right)(t_n,x)\,dx=0
\end{align}
i.e., $\lim_{n\to+\infty}K_0(u(t_n))=0$
as desired. 
\end{proof}

Next, by means of these vanishing results for $K_{0}$, we deduce the convergence to an equilibrium
along a subsequence. 

\begin{theorem} 
\label{K0(tn)stern} 
Let $\alpha >0$ and $\vec u_0:=(\varphi_0,\varphi_1)\in \mathcal{H}_{rad}$ 
so that the solution $\vec u(t)$ exists for all times $t >0$. Let   $t_n$ be  a sequence of times such that 
$K_0(u(t_n))= \delta_n$ converges to $0$, then  there exists an equilibrium point 
$\vec u^* = (Q^*,0) \in \mathcal{H}_{rad}$ such that (after possibly extracting a subsequence), $\vec u(t_n)$ converges to $(Q^*,0)$ in 
$\mathcal{H}$.
\end{theorem}
\begin{proof}
{}From Lemma~\ref{SolBorne} we conclude that
\[
\sup_{n\ge 0} \| (u(t_n),  \partial_t u (t_n)) \|_{\calH} <\I
\]
 We recall that without loss of generality, we may assume that
\begin{align}
\nn
E(u(t), \partial_t u (t))\geq 0,\quad\forall t\geq0.
\end{align}
Since the left-hand side is non-increasing,  there exists $\ell \geq 0$ such that
\begin{equation}
\label{eq:ell}
\lim_{t\to+\infty}E(u(t), \partial_t u (t))=\ell \geq 0.
\end{equation}
In fact, from the equality valid for any $t_1\leq t_2$,
$$
E(u(t_1), \partial_t u (t_1))-E(u(t_2), \partial_t u (t_2))=2\alpha\int_{t_1}^{t_2}\| \partial_t u (s)\|_{L^2}^2\,ds~,
$$
we deduce that $\int_{t_1}^{t_2}\| \partial_t u (s)\|_{L^2}^2\,ds$ tends to 0, as $t_1, t_{2}\to\infty$.

 We consider the equations
\begin{align}
\tag*{$(KG)_\alpha^n$}
\label{eq:kgn}
\begin{cases}
\partial_{tt} u_{n}+2\alpha \partial_t u_{n}-\Delta u_n+u_n-f(u_n)=0\\
(u_n(0), \partial_t u_{n}(0))=(u(t_n), \partial_t u (t_n))
\end{cases}
\end{align}
By Theorem~\ref{thm:WP},  there exists $T>0$ and $C>0$  such that, for any $n$,  the solution $(u_n(t), \partial_t u_{n}(t))$ is in $ C^0([-T,T],\mathcal{H})$ and, for $-T\leq t\leq T$,
\begin{align}
\label{eq:10'}
\|(u_n(t), \partial_t u_{n}(t))\|_{\mathcal{H}}\leq C~.
\end{align}
  In the case $d=1$ or $d=2$, the inequality \eqref{eq:10'} implies that 
$\|u_n\|_{L^{\infty}((-T,T), L^p)}\leq C$, for any $2 \leq p <+\infty$. In the case 
$3 \leq d \leq 6$, the estimate \eqref{eq:fixed} in Theorem~\ref{thm:WP} implies that
\begin{equation}
\label{eq:Strichartz3.6}
 \| u_n\|_{L^{\theta^*}((0,T), L^{2\theta^*})}\leq C~.
\end{equation}
where $\theta^* = \frac{d+2}{d-2}$. 
By uniqueness,  $u_n(t)= u(t_n+t)$.
For any $s,t\in[-T,T]$,
\begin{align*}
\int_{\mathbb{R}^d}|u_n(t)-u_n(s)|^2\,dx=&\int_{\mathbb{R}^d}\left|\int_s^t \p_{t}u_{n}(\sigma)\,d\sigma\right|^2\,dx\\
&\leq|t-s|\int_{\mathbb{R}^d}\int_s^t|\p_{t}u_{n}(\sigma)|^2\,d\sigma dx\\
&\leq|t-s|\int_{s+t_n}^{t+t_n}\| \partial_t u (\sigma)\|_{L^2}^2\,d\sigma
\end{align*}
whence 
\begin{align}
\label{eq:11}
\|u_n(t)-u_n(s)\|_{L^2}^2&\leq|t-s|\int_{s+t_n}^{t+t_n}\| \partial_t u (\sigma)\|_{L^2}^2\,d\sigma\\
&\leq2T\int_{t_n-T}^{t_n+T}\| \partial_t u (\sigma)\|_{L^2}^2\,d\sigma\longrightarrow0\quad\text{ as }n\to+\infty.\nonumber
\end{align}
For $s,t\in[-T,T]$, and fixed $p\in (2,2^{*})$,  interpolation gives the existence of $a \in (0,1)$ such that 
\begin{align}
\label{eq:12'}
\|u_n(t)-u_n(s)\|_{L^p}&\leq\|u_n(t)-u_n(s)\|_{L^{2^{*}}}^{a}
\|u_n(t)-u_n(s)\|_{L^2}^{1- a}\\
&\les |t-s|^{\frac{1- a}{2}}\left(\int_{t_n-T}^{t_n+T}\| \partial_t u (\sigma)\|_{L^2}^2\,d\sigma\right)^{\frac{1-a}{2}}\nonumber
\end{align}
with a uniform constant in $n$.  Fix $2<p_{0}<p_{1}<2^{*}$ and set $X:= L^{p_{0}}(\R^{d})
\cap L^{p_{1}}(\R^{d})$. 
The choice of $p_{0}, p_{1}$ depends on the nonlinearity $f(u)$ through the parameters $\beta, \theta$ in~\ref{H2f}. 

We consider the family of functions $(u_n(t))_n$ in $C^0([-T,T];X)$. By the property \eqref{eq:10'}, 
$$\bigcup_{\substack{n\in\mathbb{N},\\t\in[-T,T]}}u_n(t)\subset\text{ bounded set of }H^1_{rad}(\mathbb{R}^d).$$
Due to the compact embedding of $H^1_{rad}(\mathbb{R}^d)$ into $X$, we deduce that 
$$\bigcup_{\substack{n\in\mathbb{N},\\t\in[-T,T]}}u_n(t)\subset\text{ compact set of }X$$
Moreover, by \eqref{eq:12'}, the family $(u_n(t))_n$ is equicontinuous in $C^0([-T,T];X)$. 
Thus, by the theorem of Ascoli, (after possibly extracting a subsequence) the sequence $u_n(t)$ converges in $C^0([-T,T];X)$ to a function 
$u^*(t)\in C^0([-T,T];X)$.

Moreover, by \eqref{eq:11} and \eqref{eq:12'}, $u^*(t)$ is constant on the time interval $[-T,T]$. 
We shall simply write $u^*(t)\equiv u^*$. Remark that we deduce from $K_{0}(u_{n}(0))\to0$  and the convergence of $u_n(t)$ towards $u^*$ in $C^0([-T,T];X)$ that
\begin{align}
\label{eq:13'}
\lim_{n\to+\infty}\|u_n(0)\|_{H^1}^2=\int_{\mathbb{R}^d} f(u^{*})u^{*}\,dx.
\end{align}
For this implication we need to choose $p_{0},p_{1}$ close to $2, 2^{*}$, respectively, depending on~\ref{H2f}. 

To summarize, we  know that 
\begin{itemize}
\item $u_n(t)\to u^*$ as $n\to+\infty$ in $C^0([-T,T];X)$ and $u^* :=  u^*(t)$
\item $\p_{t}u_{n}(t)\to0$ as $n\to+\infty$ in $L^2((-T,T);L^2(\mathbb{R}^d))$
\item $(u_n(t),\p_{t}u_{n}(t))_n$ is uniformly bounded in $n$ in $L^\infty((-T,T);\mathcal{H})$ and, in particular in $L^2((-T,T);\mathcal{H})$.
\end{itemize}
Taking  these properties into account, one   shows that $(u_n,\partial_t u_{n})$ converges in the sense of distributions (i.e., $\mathcal{D}'((-T,T)\times\mathbb{R}^d)$) towards $(u^*,0)$ as $n\to+\infty$ and that $(u^*,0)$ is an equilibrium point of \ref{KGalpha}.
Since $(u_n(0), \partial_t u_{n}(0))$ is uniformly bounded in $\mathcal{H}$, with respect to $n$, there exists a subsequence (that we still label by $n$) such that $u_n(0)
\rightharpoonup u^*$ as $n\to+\infty$ weakly in $H^1(\mathbb{R}^d)$.

Since $u^*$ is an equilibrium point of \ref{KGalpha}, the following equality holds:
\begin{align}
\label{eq:14}
\int_{\mathbb{R}^d} f(u^*)u^{*}\,dx=\int_{\mathbb{R}^d}(|\nabla u^*|^2+(u^*)^2)\,dx.
\end{align}
The equalities \eqref{eq:13'} and \eqref{eq:14} imply that 
\begin{align}
\label{eq:15}
\lim_{n\to+\infty}\|u_n(0)\|_{H^1}^2=\|u^*\|_{H^1}^2
\end{align}
and thus, since $u_n(0) \rightharpoonup u^*$ as $n\to+\infty$ weakly in 
$H^1(\mathbb{R}^d)$, the convergence of $u_n(0)$ towards $u^*$ takes place in the strong sense in $H^1(\mathbb{R}^d)$. Moreover,  the strong convergence of 
$u_n(0)$ towards $u^*$ in $L^2(\mathbb{R}^d)$ and the property \eqref{eq:11} imply the
strong convergence of $u_n(s)$ towards $u^*$ in $L^2(\mathbb{R}^d)$, uniformly in $s \in [-T,T]$. In summary, 
$$u_n(.)\to u^* \text{ in }C^{0}((-T,T),L^2(\R^{d})).$$
To finish the proof of Theorem~\ref{K0(tn)stern} it remains to prove 
\begin{equation}\label{ii}
\p_{t}u_n(0)\to 0 \text{ in }L^{2}(\R^{d}).
\end{equation}
As a first step towards the proof of  property \eqref{ii}, we consider the equation satisfied by $\tilde u_{n}:=u_n-u^*$, namely
\begin{align}
\label{eq:18}
\begin{cases}
\partial_{tt} \tilde u_n -\Delta \tilde u_n+\tilde u_n =f(u_n)-f(u^{*})-2\alpha\partial_t \tilde u_n\\
\tilde u_n(0)=u_n(0)-u^*\to0\quad\text{ as }n\to+\infty\quad\text{ in }H^1(\mathbb{R}^d)\\
\partial_t \tilde u_n(0)=\partial_t u_{n}(0)
\end{cases}
\end{align}
We write $u_n-u^*=w_n+v_n$ where $w_n$ and $v_n$ are solutions of the following equations:
\begin{align}
\label{eq:19}
\begin{cases}
\partial_{tt} w_{n}-\Delta w_n+w_n=f(u_n)-f(u^{*})-2\alpha \partial_t u_n \\
w_n(0)=u_n(0)-u^*\\
\partial_t w_{n}(0)=0
\end{cases}
\end{align}
and
\begin{align}
\label{eq:20}
\begin{cases}
\partial_{tt} v_{n}-\Delta v_n+v_n=0\\
v_n(0)=0\\
\partial_t v_{n}(0)= \partial_t u_{n}(0).
\end{cases}
\end{align}
The classical energy estimates for the  Klein-Gordon equation imply that, for $-T\leq t\leq T$,
\EQ{ 
\label{eq:21}
\|(w_n, \p_{t}w_{n})(t)\|_{\calH}\leq C\Big[ \|u_n(0)-u^*\|_{H^1}+&2\alpha\sqrt{2T}\left(\int_{-T}^T\|\partial_t u_{n}(s)\|_{L^2}^2\,ds\right)^{\frac{1}{2}}\\
&+\int_{-T}^T\|f(u_n)(s)-f(u^{*})\|_{L^2}\,ds\Big] 
}
Taking into account Hypothesis \ref{H2f}, one has 
\EQ{\label{eq:beta theta}
&\int_{-T}^T\| f(u_n)(s)-f(u^{*})\|_{L^2}\,ds\\
&\le C \int_{-T}^T\| (u_n(s)-u^{*})(|u_{n}|^{\beta}  + |u^{*}|^{\beta} + |u_{n}|^{\theta-1}  + |u^{*}|^{\theta-1} ) \|_{L^2}\,ds  
}  
Here $0<\beta<\theta-1$ can be taken arbitrarily small, which only affects the constant 
$C$. 
Actually, we can choose $0<\beta<\theta-1$ so that 
$2 \leq2\beta p_0/(p_0 -2) \leq p_1$. Applying the H\"older inequality, we obtain,
\EQ{\label{eq:32}
&\int_{-T}^T\| (u_n(s)-u^{*})(|u_{n}|^{\beta}  + |u^{*}|^{\beta}  ) \|_{L^2}\,ds  \\
&\le CT   \| u_{n}-u^{*}\|_{L^{\I}(I,L^{p_{0}})}  (\| u_{n}\|_{L^{\I}(I,L^{p_2})}^{\beta}
+ \| u^{*}\|_{L^{\I}(I,L^{p_{2}})}^{\beta})\\
& \leq CT   \| u_{n}-u^{*}\|_{L^{\I}(I,L^{p_{0}})}  (\| u_{n}\|_{L^{\I}(I, H^1)}^{\beta}
+ \| u^{*}\|_{L^{\I}(I, H^1)}^{\beta})~.
}
where $2\leq p_2\leq p_1 <2^{*}$ is fixed.
 Since $u_{n}\to u^{*}$ in $C(I,X)$, we conclude that the right-hand side of~\eqref{eq:32}
vanishes in the limit $n\to\I$. We next estimate the term 
\begin{multline}
\label{eq:thetaLinfiniAux}
 \int_{-T}^T \| (u_n-u^{*}) |u^{*}|^{\theta-1} ) \|_{L^2}\,ds  \leq 2T 
\| u_n-u^{*} \|_{L^{\infty}(L^2)} \|u^{*} \|^{\theta-1}_{L^{\infty}(L^{\infty})} \\
\leq 
C \| u_n-u^{*} \|_{L^{\infty}(L^2)}~,
\end{multline}
which tends to $0$ as $n\to\infty$.
To bound the remaining term in~\eqref{eq:beta theta}, we argue as in the proof of Theorem \ref{thm:WP}. Indeed, from the estimates \eqref{eq:calFaux1} to \eqref{eq:calFaux4}, we deduce that
\EQ{\label{eq:Xpq}
 \int_{-T}^T\|(u_n(s)-u^{*}(s))|u_{n}(s)|^{\theta-1}\|_{L^2}\,ds \leq & (2T)^{\eta} C
\| u_n\|_{L^{\infty}(I,L^2)}^{\frac{(\theta -1)\eta}{\theta}} 
 \| u_n- u^{*}\|_{L^{\infty}(I,L^2)}^{\frac{\eta}{\theta}}\\
& \times  \big[\| u_n \|_{L^{\theta^*}(I,L^{2\theta^*})}^{(1-\eta)\theta^*}
+  (2T)^{1-\eta}\| u^*\|_{L^{\infty}(I,L^{2\theta^*})}^{(1-\eta)\theta^*}\big] \cr
\leq & C(1+ T) \| u_n- u^{*}\|_{L^{\infty}(I,L^2)}^{\frac{\eta}{\theta}}~,
}
where, by \eqref{eq:eta}, $\eta = \frac{d+2 -\theta (d-2)}{4}$. The right-hand side of the 
inequality \eqref{eq:Xpq} tends to $0$ as $n$ goes to infinity. 

Finally, in view of \eqref{eq:21}, 
\eqref{eq:beta theta}, \eqref{eq:32}, \eqref{eq:thetaLinfiniAux} and \eqref{eq:Xpq}, 
we conclude that 
 \EQ{\label{eq:23}
 \|(w_n(t), \partial_t w_{n}(t)\|_\mathcal{H}\to0 \text{\ \ as\ \ }n\to+\infty,
 }
  uniformly in  $-T\leq t\leq T$.

By construction,  $v_n=(u_n-u^*)-w_n$ and, in particular, $\partial_t  v_{n}=\partial_t  u_{n}- \partial_t  w_{n}$. From \eqref{eq:23} and the properties of $\|\partial_t  u_{n}\|_{L^2(I;L^2(\mathbb{R}^d))}$, we infer that
\begin{align}
\label{eq:24}
\|\partial_t  v_{n}\|_{L^2((-T,T);L^2(\mathbb{R}^d))}\leq\|\partial_t  u_{n}\|_{L^2((-T,T);L^2(\mathbb{R}^d))}+\sqrt{2T}\|\partial_t  w_{n}\|&_{C^0([-T, T];L^2(\mathbb{R}^d))}\to0
\end{align}
as $n\to\I$. 

In the final step of the proof we shall turn this $L^{2}_{t}$ averaged vanishing of $\| \partial_t  v_{n}(t)\|_{L^{2}_{x}}$ as $n\to\I$
into vanishing in the uniform sense in~$t$. The main tool for this is 
the  following ``observation inequality'' for equation~\eqref{eq:20}. 

\begin{lemma}
For any $T_0>0$, there exists a positive constant $c(T_0)>0$, independent of $n$, such that
\begin{align}
\label{eq:25}
\|\partial_t  v_{n}(0)\|_{L^2(\mathbb{R}^d)}^2\leq c(T_0)\int_{-T_0}^{T_0}
\int_{\mathbb{R}^d}|\partial_t  v_{n}|^2\,dxds.
\end{align}
\end{lemma}

\begin{proof}
For sake of simplicity, we set:
$$\partial_t  v_{n}(0)\equiv \partial_t  u_{n}(0)=v_{n1}.$$
If $\hat{v}_n$ denotes the Fourier transform of $v_n$, we have
$$\hat{v}_n(t,\xi)=\frac{\sin{\left(t\sqrt{|\xi|^2+1}\right)}}{\sqrt{|\xi|^2+1}}\hat{v}_{n1}(\xi)$$
and therefore
$$\|\partial_t\hat{v}_n(t,\cdot)\|_{L^2}^2=\int_{\mathbb{R}^d}\left|\sin{\left(t\sqrt{|\xi|^2+1}\right)}\right|^2|\hat{v}_{n1}(\xi)|^2\,d\xi$$
as well as
\begin{equation}
\label{eq:26}
\begin{split}
\int_{-T_0}^{T_0}\|\partial_t\hat{v}_n(t,\cdot)\|_{L^2}^2\,dt&=
\int_{-T_0}^{T_0}\int_{\mathbb{R}^d}\left|\sin{\left(t\sqrt{|\xi|^2+1}\right)}\right|^2|\hat{v}_{n1}(\xi)|^2\,d\xi dt\\
&=\int_{\mathbb{R}^d}\left(\int_{-T_0}^{T_0} \left|\sin{\left(t\sqrt{|\xi|^2+1}\right)}\right|^2\,dt\right)|\hat{v}_{n1}(\xi)|^2\,d\xi\\
&\geq\tilde{c}(T_0)| \int_{\mathbb{R}^d} \hat{v}_{n1}(\xi)|^2\,d\xi~,
\end{split}
\end{equation}
where $\tilde{c}(T_0)>0$, since $T_0>0$.
Indeed
\begin{align*}
\int_{-T_0}^{T_0} \left|\sin{\left(t\sqrt{|\xi|^2+1}\right)}\right|^2\,dt&=
\int_{-T_0}^{T_0} \left(\frac{1-\cos{\left(2t\sqrt{|\xi|^2+1}\right)}}{2}\right)\,dt\\
&=T_0 -\frac{\sin{\left(2T_0\sqrt{|\xi|^2+1}\right)}}{2\sqrt{|\xi|^2+1}} ~.
\end{align*}
One easily sees that, for any $T_0 >0$, there exists $\tilde{c}(T_0)>0$ such that, for any $|\xi|$,
\begin{equation}
\label{eq:27}
T_0 -\frac{\sin{\left( 2 T_0\sqrt{|\xi|^2+1}\right)}}{2\sqrt{|\xi|^2+1}}  \geq   \tilde{c}(T_0)~.
\end{equation}
The estimate \eqref{eq:25} is then a direct consequence of \eqref{eq:26}, \eqref{eq:27} and Plancherel's theorem.
\end{proof}
{}From the property  \eqref{eq:25} and the estimate \eqref{eq:24}, one deduces that
\begin{align}
\label{eq:28}
\|\partial_tu_{n}(0)\|_{L^2}\leq c(T)\left[\|\partial_tu_{n}\|_{L^2((-T,T);L^2)}+\sqrt{2T}\|\partial_tw_{n}\|_{C^0([-T,T];L^2(\mathbb{R}^d))}\right]&\to0\\
&\text{ as }n\to+\infty\nonumber
\end{align}
and the theorem is proved.
\end{proof}

\subsection{Convergence property}\label{sec:conv}

 Let $\vec u_0 = (\varphi_0,\varphi_1) \in  \mathcal{H}_{rad}$ be so that the solution 
 $\vec u(t) = S_{\alpha}(t)\vec u_0 \equiv (u(t), \partial_t u (t))$ exists globally and may be unbounded. Theorem~\ref{K0(tn)stern} asserts that there exists a sequence of times 
 $t_n \to +\infty$ such that $\vec u(t_n) \to  (Q^*,0)$ strongly in $\mathcal{H}_{rad}$, where  $Q^*$ is an equilibrium of \ref{KGalpha}. We shall now show by contradiction
 that then necessarily $\vec u(t)\to  (Q^*,0)$ strongly in $\mathcal{H}_{rad}$ as 
 $t\to\I$ and hence the trajectory is bounded. In other words, Theorem~\ref{K0(tn)stern} implies that the $\omega$-limit set $\omega(\vec u_0)$ is not empty and contains an equilibrium point $(Q^*,0) \in \mathcal{H}_{rad}$. We recall that the $\omega$-limit set of $\vec u_0$ is defined as
\begin{equation*}
 \begin{split}
\omega(\vec u_0) = \{ \vec w \in \mathcal{H}_{rad} \, | \, \exists & \hbox{  a sequence } s_n \geq 0, 
\hbox{ so that } s_n \rightarrow_{n \rightarrow +\infty} +\infty~,\cr
& \hbox { and }  S_{\alpha} (s_n) \vec u_0 \rightarrow_{n \rightarrow +\infty} \vec w  \}~. 
\end{split}
\end{equation*} 
 Below we will show that the $\omega$-limit set  $\omega(\vec u_0)$ reduces to the singleton $(Q^*,0)$, and that the entire trajectory converges to this point in the strong sense. And this concludes the proof of Theorem~\ref{ThBRS1}. 
\medskip

 Before proving that the entire trajectory $\vec u(t) = 
S_{\alpha}(t)\vec u_0$ converges to $(Q^*,0)$, we will emphasize that the $\omega$-limit set $\omega(\vec u_0)$ is contained in the set $\mathcal{E}_{rad}$  of radial equilibrium points of \ref{KGalpha}.

\begin{lemma} The $\omega$-limit set $\omega(\vec u_0)$  satisfies the property
\begin{equation}
\label{eq:Ecalrad}
\omega(\vec u_0) \subset \mathcal{E}_{rad}~.
\end{equation} 
\end{lemma}
\begin{proof}
Let $\vec v_0 =(v_0,v_1) \in \omega(\vec u_0)$. Then, there exists a sequence $s_n \rightarrow_{n \rightarrow +\infty} +\infty$ such that 
$S_{\alpha} (s_n) \vec u_0 \equiv \vec u(s_n) \rightarrow_{n \rightarrow +\infty} 
\vec v_0$. \\
On the one hand, we know by \eqref{eq:ell} that the energy satisfies 
$$
E(\vec u(s_n ))\to \ell =E(({Q}^{*},0))
$$
as $n\to+\infty$, and
$$
E(\vec u(s_n) )\to E(\vec v_0).
$$
If $\vec v_0$ is not an equilibrium point, then for some time $\sigma>0$,
\begin{align}
\label{eq:2.15}
E(S_\alpha(\sigma) \vec v_0) \leq E(\vec v_0)-\delta = \ell -\delta
\end{align}
where $\delta>0$.
Since $$
E(\vec u(s_n+\sigma))\to \ell
$$ 
and 
$$
E(\vec u(s_n +\sigma))\to E(S_\alpha(\sigma)\vec v_0)~,
$$ 
we arrive at a contradiction and \eqref{eq:Ecalrad} holds.
\end{proof}

\begin{remark} \label{rem:omegatau}
Let us fix a positive time $\tau >0$ and introduce the $\omega$-limit set 
$\omega_{\tau}(\vec u_0)$ of the discrete dynamical system defined by the iterates 
$S_{\alpha}(\tau)^m$, $m \in \mathbb{N}$, that is,
\begin{equation*}
 \begin{split}
\omega_{\tau}(\vec u_0) = \{ \vec w \in \mathcal{H}_{rad} \, | \, \exists & \hbox{  a sequence } k_n \geq 0, 
\hbox{ so that } k_n \rightarrow_{n \rightarrow +\infty} +\infty~,\cr
& \hbox { and }  S_{\alpha} (\tau)^{k_n} \vec u_0 \rightarrow_{n \rightarrow +\infty} \vec w  \}~. 
\end{split}
\end{equation*} 
Obviously, $\omega_{\tau}(\vec u_0) \subset
\omega(\vec u_0)$. Using the fact that $\omega(\vec u_0)$ is contained in 
$\mathcal{E}_{rad}$ and that the Lipschitz property of $S_{\alpha}(t): \vec v_0 \in \mathcal{H} \to S_{\alpha}(t)\vec v_0 \in \mathcal{H}$, which is uniform with respect to 
$t \in [0,\tau]$ (see the arguments in Step 1 of Section 4 and especially the estimates \eqref{eq:AStricAux1}, \eqref{eq:AStricAux2}, and \eqref{eq:AStricAux3}), one can show 
that 
\begin{equation}
\label{eq:omega-tau}
\omega_{\tau}(\vec u_0) = \omega(\vec u_0) ~.
\end{equation} 
\end{remark} 
\vskip 4mm

 To prove that the $\omega$-limit $\omega(\vec u_0)$ is a singleton and that the entire trajectory converges to this point, we will apply a generalization  of the classical convergence theorem of Aulbach \cite{Aulb}, Hale-Massat \cite{HaMa} and Hale-Raugel \cite{HaRa}, due to Brunovsk\'{y} and P.\ Pol\'{a}\v{c}ik
 \cite{BrP97b}, which uses local invariant manifold theory. For more details on these convergence theorems, we refer the reader to Appendix~B   and especially to Lemma  \ref{lemmeB1} that we shall apply below.
 The behaviour of $S_{\alpha}(t)\vec u_0=\vec u(t)$ heavily depends on the spectral properties of the  linearized operator 
$\mathcal{L}$ about $Q^*$ and the linearized operator $\tilde{\Sigma}_{\alpha}(t)
 = e^{A_{\alpha}t}$ about  $(Q^*,0)$ (see the definitions  \eqref{eq:vN}, \eqref{Sys1} or \eqref{AgL0} with 
$\varphi=Q^*$). Lemma~\ref{lem:spectre2A} describes the spectrum of the operator 
$A_{\alpha}$.\\
Before proving this convergence result, we need to recall some notation given in Section 
\ref{sec:IMforKG}.
There we introduce the  modified (localized) Klein-Gordon equation 
\eqref{eq:Auphi0mod} and show that this localized equation defines a globally defined flow 
$\bar{S}_{\alpha}(t)$ on $\mathcal{H}_{rad}$, such that,
\begin{equation}
\label{32}
\vec u (t) = S_{\alpha}(t) ((Q^*,0) +\vec v_0) = (Q^*,0) + \bar{S}_{\alpha}(t) \vec v_0~, \hbox{ as long as }  \vec u(t) \in B_{r_1}~,
\end{equation}
where $B_{r_1} \equiv B((Q^*,0),r_1)$ is the open ball of center $(Q^*,0)$ and radius $r_1 >0$,  with $r_1 \leq (8 C(\alpha,\tau_0))^{-1}r_0$ (see Remark \ref{rem:SEgalSbarre}).  In other terms, if  we set 
$$
S^*_{\alpha}(t) \vec u_0= (Q^*,0) + \bar{S}_{\alpha}(t) (\vec u_0  - (Q^*,0))~,
$$ 
then $S_{\alpha}(t) \vec u_0$ and $S_{\alpha}^*(t) \vec u_0$ coincide as long as 
$S_{\alpha}(t) \vec u_0 \in B_{r_1}$. \\
In Section \ref{sec:IMforKG}, we define  the (global) stable, unstable, center stable, center unstable, and center manifolds $W^{i*}((Q^*,0))$ of $S^*_{\alpha}(t)$ about $(Q^*,0)$, where
$i= s, u, cs, cu, c$ respectively. Since $S_{\alpha}(t) \vec u_0$ and 
$S_{\alpha}^*(t) \vec u_0$ coincide as long as $S_{\alpha}(t) \vec u_0 \in B_{r_1}$, 
we may define the local stable, unstable, center stable, center unstable, and 
center  manifolds $W^{i}_{loc}((Q^*,0))$ of $S_{\alpha}(t)$ about $((Q^*,0))$ 
as follows:
 \begin{equation}
\label{32Wsuc}
W^{i}_{loc}((Q^*,0)) = W^{i*}((Q^*,0)) \cap B_{r_1}~, \quad i= s,  u, cs,  cu, c~.
\end{equation}

We begin our proof with the particular case where $(Q^*,0)$ is the (hyperbolic) trivial equilibrium $(0,0)$ of \ref{KGalpha}.  
We remark that in that case $\mathcal{L}= - \Delta+ I$ and the entire spectrum of $A_\al$ lies in a half-plane of the form $\Re z<-\delta<0$. 
In the terminology of Section \ref{sec:IMforKG} and of Appendix~A,  this means that the local stable manifold 
$W^{u}_{loc}((0,0))$ is a whole neighborhood of $(0,0)$ and that then necessarily $(0,0)$ is an isolated equilibrium, and the perturbative equation \eqref{eq:vN} around $(0,0)$ exhibits exponential  decay of solutions in $\mathcal{H}_{rad}$  for small data. Actually, this exponential  decay to zero had already been proved in Theorem 
\ref{thm:WP}. In particular, $\vec u(t)\to  (0,0)$ in that case as 
$t\to\I$. 

\medskip

  Let us come back to the general case. If $Q^*\ne0$, then Lemma~\ref{lem:spectre2A} states that $A_\al$ has either a trivial kernel, or a one-dimensional kernel. The former case means that the dynamics near $(Q^*,0)$ is {\em hyperbolic}, whereas in the latter case it is not. 
In the hyperbolic scenario, we have no central part, which means that the 
 invariant manifolds constructed in Section \ref{sec:IMforKG} and in Appendix~A only involve stable and unstable manifolds
 $W^{s}_{loc}((Q^*,0))$ and $W^{u}_{loc}((Q^*,0))$.
In both cases, the (local) unstable manifold $W^{u}_{loc}(Q^*,0)$ is finite-dimensional since 
$\calL$ has only finitely many eigenvalues  (and thus only finitely many eigenvalues with positive real part). 

\medskip

 In the non-hyperbolic case, the kernel of 
 $A_{\alpha}$ is one-dimensional, the local center manifold $W^{c}_{loc}((Q^*,0))$ is a   $C^1$-curve containing $(Q^*,0)$. We 
 notice that we can also  choose $r_1 >0$ small enough so that $W^{c}_{loc}((Q^*,0)) 
 = W^{c*}((Q^*,0)) \cap \overline B_{r_1}$ is a connected curve. Moreover, as remarked above, the (local) unstable manifold $W^{u}_{loc}(Q^*,0)$ is finite-dimensional. 
  In order to prove the convergence  to $(Q^*,0)$, we would like to directly apply the classical convergence theorem of \cite{BrP97b} or \cite{HaRa}, which is the case (1) of Theorem \ref{th:ConvB2}. However, we do not know that the trajectory $\vec u(t)$ is bounded and thus 
 we also cannot ascertain that  the $\omega$-limit set $\omega(\vec u_0)$ is connected.   So we will apply the more general convergence 
 Theorem \ref{th:ConvB1} of Brunovsk\'{y} and Pol\'{a}\v{c}ik, and more precisely their local Lemma \ref{lemmeB1}, which are  recalled in Appendix B. To this end, we need to show that $(Q^*,0)$ is stable for $S_{\alpha}(t)$ restricted to the local center manifold (see the definition \eqref{eq:2.16} below). In order to prove this stability, we shall use the same arguments as  Brunovsk\'{y} and 
Pol\'{a}\v{c}ik in the proof of Lemma \ref{lemmeB1}.  Like them, we will make use of  the attraction of the center unstable manifold with asymptotic 
phase of Section \ref{sec:IMforKG} (see also Appendix~A). 
 Notice that the hyperbolic case can be considered as a special case, where the local center unstable (respectively,  center) manifold reduces to the local unstable manifold (respectively, 
to $(Q^*,0)$).  In the non-hyperbolic case, the center manifold is present and the  dynamics is more delicate to analyze.  

We proceed by contradiction and assume that $\vec u(t)\not\to( Q^*,0)$. 
Since $\vec u(t)$ does not converge to $( {Q}^*,0)$, there exists $\beta_0>0$, $\beta_0<\frac{r_1}{2}$  with the
following property:  for any $0<\beta\leq\beta_0$, if 
$\vec u(t_0)\in B_\mathcal{H}((Q^{*},0),\beta)$, there exists a first time $\tau_0>0$ such that $\vec u(t_0+\tau) \in B_\beta$, for $0\leq\tau<\tau_0$, and $\vec u(t_0+\tau_0) \not\in B_\beta$. In other words, $\vec u(t_0+\tau_0)$ belongs to the sphere
 $S(({Q}^{*},0),\beta)$.

We first fix $\beta>0$, $\beta\leq\beta_0$. By Theorem \ref{K0(tn)stern}, there exists 
$n(\beta)$ such that, for $n\geq n(\beta)$, $\vec u(t_{n})\in B_\beta$. 
Moreover, there exists a first time $\tau_n(\beta)>0$ such that
\begin{equation}
\label{eq:sphere}
\begin{split}  
\vec u(t_n+\tau) &\in B_\beta\quad\text{ for }0\leq\tau <\tau_n(\beta)\\
\vec u(t_n+\tau) & \not\in B_\beta\quad\text{ for }\tau=\tau_n(\beta)~.
\end{split}
\end{equation}
Since $\vec u(t_n)\to({Q}^{*},0)$ as $n\to+\infty$, we remark that $\tau_n(\beta)\to+\infty$ as $n\to+\infty$. 
 We now invoke the asymptotic phase property of the center-unstable manifold, see~\eqref{eq:AfoliaLya} (or also \eqref{eq:Aattracexp1} in Theorem \ref{AVarieteCenInst}).
Thus, there exists $\xi_n:=\xi(\vec u(t_n)) \in W^{cu}_{loc}(Q^*,0)$ such that, for 
$t\geq0$,
\begin{align}
\label{PCU}
\|S^*_\alpha(t)\vec u(t_{n})- S^*_\alpha(t)\xi_n\|_\mathcal{H}\leq c_0 
\rho_0^t\|\vec u(t_n)-\xi_n\|_\mathcal{H},
\end{align}
where $0<\rho_0<1$. And, by continuity of the map $\xi(\cdot)$,
$$
\xi_n\to({Q}^{*},0)\quad\text{ as }n\to+\infty.
$$
In particular, \eqref{PCU} implies that
\begin{align}
\label{eq:2.13}
\|S_\alpha(\tau_n(\beta))\vec u(t_{n})-S^*_\alpha(\tau_n(\beta))\xi_n\|_\mathcal{H}\to 0\quad\text{ as }n\to+\infty.
\end{align}
Since $W^{cu*}((Q^*,0))$ is finite-dimensional and  by \eqref{eq:2.13},
$S^*_\alpha(\tau_n(\beta))\xi_n$ is bounded, the sequence $S^*_\alpha(\tau_n(\beta))\xi_n$, $n \in \mathbb{N}$,  contains a convergent subsequence. 
We conclude that up to passing to a subsequence one has 
\EQ{
\label{eq:2.14}
\vec u(t_n+\tau_n(\beta)) = S_\alpha(\tau_n(\beta))\vec u(t_n) \to 
(\tilde{u}_0, \tilde{u}_1) \in \bar{B}_{\beta} \text{ as }n\to+\infty.\nonumber
}
By the invariance property of $W^{cu*}((Q^*,0))$ and by~\eqref{eq:2.13}, 
\begin{equation}
\label{eq:limWCU}
(\tilde{u}_0, \tilde{u}_1)\in  W^{cu}_{loc}((Q^*,0))~.
\end{equation}
We remark that, by \eqref{eq:Ecalrad} and \eqref{eq:sphere},
\EQ{\label{ClaimQ}
(\tilde{u}_0, \tilde{u}_1)\text{\ \  is an equilibrium point\ \ } (\tilde{Q},0) \equiv
 (\tilde{Q}(\beta),0) \text{\ \  and\ \ } \|(\tilde{Q}(\beta),0)-(Q^*,0)\|_\mathcal{H}=\beta.
}
 If $({Q}^{*},0)$ is an isolated equilibrium point, then \eqref{ClaimQ} with $\beta\leq\frac{r_1}{2}$ leads to a contradiction. We remark that, in the hyperbolic case, $({Q}^{*},0)$ is necessarily an isolated equilibrium which ends the proof in this case.

 \medskip 

 Let us now focus  on the case where $({Q}^{*},0)$ is not isolated. 
Before completing the proof in this case, we recall a definition of Brunovsk\'{y} and Pol\'{a}\v{c}ik, see Appendix~B. We say that 
$(Q^{*},0)$ is {\bf stable for} $S_\alpha(t)\vert_{W_{loc}^c((Q^{*},0))}$ if, 
$\forall \epsilon>0$, $\exists \theta>0$ such that, for any $\vec v_0\in 
W^{c}_{loc}(({Q^*,0))}$, 
$$
\|\vec v_0-({Q}^{*},0)\|_\mathcal{H}\leq\theta
$$
 implies that, for $t\geq0$,
\begin{align}
\label{eq:2.16}
\|S_\alpha(t) \vec v_0-({Q}^{*},0)\|_\mathcal{H}\leq\epsilon.
\end{align}
We now complete our proof. 
By construction and~\eqref{eq:limWCU}, the element $(\tilde{Q}(\beta),0)$ belongs to 
$W^{cu}_{loc}((Q^*,0))$.  Since $(\tilde{Q}(\beta),0)$ is an equilibrium point, it necessarily belongs to the local center manifold $W^{c}_{loc}((Q^*,0))$ (see Section \ref{sec:IMforKG} and Appendix A for more explanations), which, as we saw earlier, is a $C^1$  one-dimensional embedded manifold passing through $(Q^{*},0)$.

Since \eqref{ClaimQ} holds for any small $\beta>0$, 
we see that  this curve segment  contains equilibria   in the $omega$-limit set
$\omega( \vec u_0)$ which are arbitrarily close to, but distinct from, 
$(Q^{*},0)$. 
In fact, we can say even more than that. First, 
we place an order on the curve $\tilde{W}_{r_1}^c(({Q}^{*},0))$ if $r_1>0$ is small enough. 
We adopt the notation $v^-<({Q}^{*},0)<v^+$ if 
$v^-$ (respectively $v^+$) is to the ``left'' ( resp.\  ``right'') of $({Q}^{*},0)$ on 
the curve segment $\tilde{W}_{r_1}^c(({Q}^{*},0))$.
By intersecting the tangent
line to this curve at $(Q^{*},0)$ with the spheres of radius $\beta$ for all small $\beta$, 
we see that there are two possibilities: 
\begin{enumerate}
\item Either  there exist two families of equilibria $(Q_m^-,0)$ and $(Q_m^+,0)$ with 
$(Q_m^-,0)<({Q}^{*},0)<(Q_m^+,0)$ such that
\begin{align}
\label{eq:2.17}
(Q_m^\pm,0)\to({Q}^{*},0)\quad\text{ as }m\to+\infty.
\end{align}
  A simple dynamical argument based on \eqref{eq:2.17} implies  that 
 $S_\alpha(t)\vert_{W^{c}_{loc}((Q^*,0))}$ is in fact stable.  We can now directly apply Lemma \ref{lemmeB1} of Brunovsk\'{y} and Pol\'{a}\v{c}ik  to 
 the time $1$ map $S_{\alpha}(1)$, which implies that  the $\omega$-limit set 
 $\omega_{1}(\vec u_0)$ and thus the $\omega$-limit set 
 $\omega(\vec u_0)$ contain an element of $W^u_{loc}(({Q}^{*},0))\backslash({Q}^{*},0)$. This contradicts the fact that $\omega(\vec u_0) \in \mathcal{E}_{\ell}$.
 Instead of directly applying Lemma \ref{lemmeB1} to  the map $S_{\alpha}(1)$, we can 
 also argue for the flow $S_{\alpha}(t)$ as at the end of the proof  of \cite[Lemma 1]{BrP97b}  of Brunovsk\'{y} and Pol\'{a}\v{c}ik  and  directly show  that 
 $(\tilde Q(\beta),0) \in W^u_{loc}(({Q}^{*},0))\backslash({Q}^{*},0)$, where 
 $\tilde{Q}(\beta)$ is as in~\eqref{ClaimQ}. But this contradicts the fact that 
 $(\tilde{Q}(\beta),0)$ is an equilibrium  and so we again obtain the desired convergence. 

\medskip

\item   Or there exists $\beta_2>0$ such that there is no equilibrium point from the family $(\tilde Q(\beta),0)$ on the ``left" (say) of $(Q^*,0)$ in 
$W^c_{loc}(({Q}^{*},0)) \cap B_{2\beta_2}$. But then, the above arguments  (and in particular the properties \eqref{ClaimQ}) imply that, for every $0 \leq \beta \leq \beta_{2}$, there exists an equilibrium 
$(\tilde{Q}^{+}(\beta),0)$ in $\omega(\vec u_0)$ satisfying the properties \eqref{ClaimQ}. This implies that on the right of $(Q^*,0)$, $W^c_{loc}(({Q}^{*},0))$ consists only of equilibria and that the $\omega$-limit set $\omega(\vec u_0)$ contains a curve 
$\mathcal{C}$ of equilibria with end point $(Q^*,0)$ (as for an interval).  
We then choose an equilibrium $(\tilde{Q}^{+}(\beta),0)$ in the interior of 
$\mathcal{C}$ and    close to $(Q^*,0)$. 
We repeat  the above proof  with $(Q^*,0)$ replaced by $(\tilde{Q}^{+}(\beta),0)$. And we  again obtain the same contradiction as in Case (1).  
\end{enumerate}

\begin{remark}  \label{3LojaSimon}
\upshape
  In the particular case of a wave type or reaction-diffusion equation, the proof of 
the \L ojasiewicz-Simon inequality (see Sections~3.2 and 3.3 in the monograph of L.~Simon \cite{Si96} and also \cite[Theorem 2.1]{HaJen07}) shows that, when the kernel of 
$\mathcal{L}$ is one-dimensional,  the set of equilibria of \ref{KGalpha} passing through 
$(Q^*,0)$ is a $C^1$-curve. We could have used this property in the proof
above to avoid the last arguments and apply Theorem \ref{th:ConvB1}. However, in view of possible extensions, we chose not to use this property.
\end{remark} 


\section{Invariant manifold theory for the Klein-Gordon equation}
\label{sec:IMforKG}

 In Section~\ref{sec:conv}, in order to prove the convergence of any global solution (in positive time) towards an equilibrium point $(\varphi_0,0)$ of \ref{KGalpha}, we used the properties of the local unstable, local center unstable and local center manifolds
$W^i_{loc}((\varphi_0,0))$, $i=u, cu, c$ about $(\varphi_0,0)$ for the flow $S_{\alpha}(t)$. There, 
we defined these local manifolds as the intersections of the global manifolds  
$W^{i*}((\varphi_0,0))$, $i=u, cu, c$ about $(\varphi_0,0)$ for the global flow 
$S^{*}_{\alpha}(t)$,  with the ball of center $(\varphi_0,0)$ and radius $r_1>0$, where 
$r_1>0$ is small enough. We recall that the global flow $S^{*}_{\alpha}(t)$ was defined 
by 
$$
S^*_{\alpha}(t) \vec u_0= (\varphi_0,0) + \bar{S}_{\alpha}(t) (\vec u_0  - (\varphi_0,0))~,
$$ 
where $\bar{S}_{\alpha}(t)$ is the global flow defined by the localized  Klein-Gordon equation \eqref{eq:Auphi0mod} below. \\
In this section, we construct the global invariant manifolds $W^i((0,0))$, $i=u,cu,c$, for the global flow $\bar{S}_{\alpha}(t)$ and 
obtain the attraction property of $W^{cu}((0,0))$ by applying the general invariant manifold theory recalled in Appendix A.
\vskip 4mm

Let $(\varphi_{0},0) \in \mathcal{H}_{rad}$ be an equilibrium point  of \ref{KGalpha}, that is, $\varphi_0$ is a radial solution of the elliptic equation
\begin{equation}
\label{Aelliptic}
-\Delta \varphi_0 +\varphi_0 -f(\varphi_0) =0~.
\end{equation}
Solving the equation \ref{KGalpha} in the neighborhood of $(\varphi_{0},0)$ leads one to solve the equation
\begin{equation}
\label{Auphi0}
v_{tt} + 2\alpha v_t + \mathcal{L}v - g_0(v)= 0~, \quad (v,v_t)(0) \equiv \vec v(0) \in \mathcal{H}_{rad}~.
\end{equation}
where 
\begin{equation}
\label{AgL0}
\begin{split}
&\mathcal{L}  = -\Delta + I -  f'(\varphi_0)~, \cr
& g_0(v)  = f(\varphi_0+v)- f(\varphi_0) - f'(\varphi_0) v~.
\end{split}
\end{equation}
The equation \eqref{Auphi0} can be written in  matrix form as follows
\EQ{\label{ASys1}
\p_t\binom{v}{v_t} = \left(\begin{matrix} 0 & 1\\ - \mathcal{L} & -2\alpha\end{matrix}\right) \binom{v}{v_t} + \binom{0}{g_0(v)} \equiv A_{\alpha} \vec v
+ \binom{0}{g_0(v)}
}
We denote by $\tilde{\Sigma}_{\alpha}(t)= e^{A_{\alpha}t}$ the linear group 
generated by $A_{\alpha}$ and $\tilde{S}_{\alpha}(t)$ the local flow defined by the equation \eqref{Auphi0}. We notice that 
\begin{equation}
\label{eq:SStilde}
S_{\alpha}(t) \vec u_0 = S_{\alpha}(t)( (\varphi_0,0) + \vec v_0) = 
(\varphi_0,0) + \tilde{S}_{\alpha}(t) \vec v_0~, \quad \hbox{ where } \vec v_0 = 
\vec u_0 - (\varphi_0,0) ~.
\end{equation}

\medskip

\noindent When $\alpha >0$, according to Lemma \ref{lem:spectre2A}, the radius 
$\rho (\sigma_{ess}(\tilde{\Sigma}_{\alpha}(\tau)))$ of the essential spectrum of 
$\tilde{\Sigma}_{\alpha}(\tau)$ satisfies:
$$
\rho (\sigma_{ess}(\tilde{\Sigma}_{\alpha}(\tau))) \leq  \delta(\alpha,\tau) <1
$$
The operator $A_{\alpha}$ can have a finite number of negative eigenvalues 
$\mu_j^{-}(\alpha) <0$ (resp.\  a finite number of positive eigenvalues $\mu_j^{+}(\alpha) >0$), in which case, $\lambda_j^{-}(\tau,\alpha) \equiv 
\exp ( \mu_j^{-}(\alpha) \tau) <1$ (resp.\  $\lambda_j^{+}(\tau,\alpha) \equiv 
\exp ( \mu_j^{+}(\alpha) \tau) >1$). 

In addition, if $1$ is an eigenvalue of $\tilde{\Sigma}_{\alpha}(\tau_0)$,$\tau_0 >0$, it is a simple eigenvalue (and is a simple eigenvalue of
$\tilde{\Sigma}_{\alpha}(\tau)$ for any $\tau >0$). Since this case plays an important role in the proof of Section~\ref{sec:conv}, we assume that $1$ is an eigenvalue of 
$\tilde{\Sigma}_{\alpha}(\tau_0)$,$\tau_0 >0$. In this case, we will construct a local center unstable manifold $W^{cu}_{loc}((0,0))$ of the equilibrium $(0,0)$  of 
$\tilde{S}_{\alpha}(t)$, a foliation of a neighborhood of $(0,0)$ in $\mathcal{H}_{rad}$ over $W^{cu}_{loc}((0,0))$ as well as a local center manifold $W^{c}_{loc}((0,0))$ 
by applying Theorems  \ref{th:FoliaA2} and \ref{th:FoliaA4} to   
$\tilde{S}_{\alpha}(t)$. We choose $\tau_0 >0$ small enough ($\tau_0$ will be made more precise later). And we set 
$$
L = \tilde{\Sigma}_{\alpha}(\tau_0)~.
$$
The spectrum $\sigma(L)$ can be decomposed as in Hypothesis (HA.5.1) and one can define constants $C_1 \geq 1$, $C_2 \geq 1$, $\eta>0$ and $\varepsilon>0$ satisfying the estimates \eqref{eq:SousEnsheta}, \eqref{eq:estcus}, \eqref{eq:AL1L2moins'}.
Unfortunately, $\tilde{S}_{\alpha}(t)$ is only a local flow and thus 
$\tilde{S}_{\alpha}(\tau_1)$ will not satisfy the hypothesis (HA.3). Moreover, we need to 
show that the Lipschitz-constant $\mathrm{Lip}(R)$ can be chosen as small as needed, which is not true for $\tilde{S}_{\alpha}(t)$.
Therefore, we need to make a localization in the following way, for instance.
Let $r_0 >0$ be a small constant, which will  be made more  precise later. We introduce a smooth cut-off function $\chi: \R \to [0,1]$ such that 
\begin{equation}
\label{Chi}
\begin{split}
\chi(s)   = \begin{cases}
1 & \text{if } |s| \leq 1, \\
0 & \text{if } |s| \geq 2~. \\
\end{cases}
\end{split}
\end{equation}
And, we consider the modified Klein-Gordon equation,
\begin{equation}
\label{eq:Auphi0mod}
v_{tt} + 2\alpha v_t + \mathcal{L}v - g_0(v)
\chi \big(\frac{\| \vec v\|_{ \mathcal{H}}^2}{r^2}\big)=0~, \quad 
 \vec v(0) =  \vec v_0 \in\mathcal{H}_{rad}~,
\end{equation}
where $0 < r \leq r_0$ is fixed. To simplify the notation, we set
$$
h( \vec v) = g_0(v) \chi \big(\frac{\| \vec v\|_{ \mathcal{H}}^2}{r^2}\big)~.
$$ 
We first show that, for any  $\vec v_0 \in \mathcal{H}$, the equation 
\eqref{eq:Auphi0mod} admits a unique solution 
$\vec v(t) \equiv  \bar{S}_{\alpha}(t) \vec v_0 \in  C^0([0,+\infty),\mathcal{H})$ 
(we leave to the reader to show that  $\bar{S}_{\alpha}(t) \vec v_0$ also belongs to  
$C^0((-\infty,0],\mathcal{H})$). 
To this end, it is sufficient to show that, for any $\vec v_0 \in \mathcal{H}$, the solution 
$\vec v(t) \equiv 
\bar{S}_{\alpha}(t) \vec v_0$ of  \eqref{eq:Auphi0mod} exists on the time interval $[0,\tau_0]$ and remains bounded there. \\
We will do that in two steps. We will give the proof only in the case where $d \geq 3$, the case $d \leq 2$ being easier.
We first recall that the solution $\vec v(t)$ of \eqref{eq:Auphi0mod} is given by the Duhamel formula,
\begin{equation}
\label{eq:ADuhamel}
\vec v(t) =  \tilde{\Sigma}_{\alpha}(t) \vec v_0 + \int_{0}^{t} \tilde{\Sigma}_{\alpha}(t-s) (0,g_0(v(s)) \chi \big(\frac{\| \vec v(s)\|_{ \mathcal{H}}^2}{r^2}\big) )^t
\, ds~,
\end{equation}
and also remark that, as long as $\vec v(s) \notin B_{\mathcal{H}}(0, \sqrt{2}r)$, the term $h(\vec v(s))$ vanishes. 

\medskip

{\bf Step 1: } Let $\vec v_0 \in \mathcal{H}$ so that $\| \vec v_0\|_{ \mathcal{H}} \leq m r$ with 
$ (8C(\alpha,\tau_{0}))^{-1} \leq m \leq 2$ for example. We set: $M_0 \equiv M_0(mr) = 
4C(\alpha,\tau_{0}) mr$, where $C(\alpha,\tau_{0}) \geq 1$ is the constant given in Proposition 
\ref{AStrichphi0}.  To show the local existence of the solution $\vec v(t)$ on the time interval $[0, \tau_0]$, we argue as in the proof of Theorem \ref{thm:WP} and introduce the space
\begin{equation*}
\begin{split}
Y \equiv  \{ \vec v \in L^{\infty}((0,\tau_0), \mathcal{H}) \hbox{ with } 
& v \in L^{\theta^*}((0,\tau_0),L^{2\theta^*}(\mathbb{R}^d)) \cr 
&\, | \, \| v\|_{L^{\infty}(H^1) \cap W^{1,\infty}(L^2) \cap L^{\theta^*}(L^{2\theta^*})} \leq M_0(mr)\}~.
\end{split}
\end{equation*}
Like there we introduce the mapping $\mathcal{F} : Y \to Y$ defined by
$$
(\mathcal{F}\vec v)(t) = \tilde{\Sigma}_{\alpha}(t) \vec v_0  
+ \int_{0}^{t} \tilde{\Sigma}_{\alpha}(t-s) (0,h(\vec v(s)))^t ds~.
$$
The application of Proposition \ref{AStrichphi0} implies
\begin{equation}
\label{eq:calF0*}
 \| \mathcal{F}(0)\|_{Y} \leq C(\alpha, \tau_0) mr  \leq \frac{M_0(mr)}{4}~.
\end{equation}
We next show that $\mathcal{F}$ is a strict contraction from $Y$ into $Y$. Due to the hypothesis \ref{H2f}, we may write, for $v_1, v_2 $ in $H^1(\R^d)$,
\begin{equation}
\label{Agy1y2}
\begin{split}
 |(g_0(v_1) - g_0(v_2))(x)|& = |f(\varphi_0(x) + v_1(x)) - f(\varphi_0(x) + v_2(x))  - 
f'(\varphi_0(x)) (v_1(x) - v_2(x)) | \cr
& = | \int_{0}^{1}(f'(\varphi_0 +v_2 + \sigma (v_1 - v_2)) - f'(\varphi_0)) (v_1- v_2) 
d\sigma | \cr
& \leq C |( |v_1|^{\beta} + |v_2|^{\beta} + |v_1|^{\theta - 1} + |v_2|^{\theta -1}) 
(v_1- v_2) |~,
\end{split}
\end{equation}
where $0 <\beta < \min(\theta - 1, \frac{2}{d-2})$   and $C \equiv C(f,\varphi_0)$ is a constant depending only on $f$ and on $\varphi_0$. For $\vec v_i \in Y$, $i=1,2$,  Proposition \ref{AStrichphi0} and  the inequality \eqref{Agy1y2} imply,
\begin{equation}
\label{eq:AStricAux1}
\begin{split}
\| \mathcal{F}\vec v_1 - \mathcal{F} &\vec v_2 \|_{Y} \leq C(\alpha, \tau_0) 
\int_0^{\tau_0}
\| h(\vec v_1(s)) - h( \vec v_2(s))\|_{L^2} ds \cr
\leq  C(\alpha, \tau_0) & \int_0^{\tau_0}\|  (g_0(v_1)- g_0(v_2)) 
\chi \big(\frac{\| \vec v_2\|_{ \mathcal{H}}^2}{r^2}\big) + g_0(v_1)
 \big(\chi \big(\frac{\| \vec v_1\|_{ \mathcal{H}}^2}{r^2}\big)
- \chi \big(\frac{\| \vec v_2\|_{ \mathcal{H}}^2}{r^2}\big) \big)\|_{L^2} ds \cr
 \leq   C(\alpha, \tau_0) & C \big[ \int_0^{\tau_0} \| ( |v_1(s)|^{\beta} + |v_2(s)|^{\beta}) 
|v_1(s) - v_2(s)| \,\|_{L^2}ds \cr
&+ \int_0^{\tau_0} \| ( |v_1(s)|^{\theta -1} + |v_2(s)|^{\theta -1}) 
|v_1(s) - v_2(s)| \,\|_{L^2}ds \cr
&+ \int_0^{\tau_0} \| ( |v_1(s)|^{\beta +1} + |v_1(s)|^{\theta})\|_{L^2}
 | \big(\chi \big(\frac{\| \vec v_1\|_{ \mathcal{H}}^2}{r^2}\big)
- \chi \big(\frac{\| \vec v_2\|_{ \mathcal{H}}^2}{r^2}\big) \big)| ds \big]
\cr
& \equiv B_1 +B_2 +B_3 ~.
\end{split}
\end{equation}
Arguing as in the proof of Theorem \ref{thm:WP}, by using the Sobolev embeddings, the 
 H\"older inequality and  the fact that $0 < \beta < \frac{2}{d-2}$, we obtain the following inequality for $B_1$:
\begin{equation}
\label{eq:AStricAux2}
\begin{split}
B_1 &\leq   C(\alpha, \tau_0)  C \int_0^{\tau_0} (\| v_1\|_{H^1}^{\beta} 
+ \| v_2\|_{H^1}^{\beta})\| v_1 -v_2\|_{H^1}ds 
 \leq  2  C(\alpha, \tau_0)\tau_0  C  M_0 (rm)^\beta\| v_1 -v_2\|_{L^{\infty}(H^1)}
\end{split}
\end{equation}
The bound of the term $B_2$ is obtained as in the proof of Theorem \ref{thm:WP} 
 (see \eqref{eq:calFaux5}): 
\begin{equation}
\label{eq:AStricAux3}
B_2 \leq 2  C(\alpha, \tau_0)  C^2 \tau_0^{\eta} 
M_0(rm)^{\frac{\theta -1}{\theta}(\theta^*(1-\eta) +\eta)} 
\big[ \| v_1 - v_2\|_{L^{\infty}(L^2)} +  \| v_1 - v_2 \|_{L^{\theta^*}(L^{2\theta^*})} \big]~.
\end{equation}
where $\eta >0$ is given in the formula \eqref{eq:eta}. It remains to bound the term $B_3$. 
We first remark that, since $ \chi' \big(\frac{\| \vec w\|_{ \mathcal{H}}^2}{r^2}\big)$ vanishes if $\| \vec w\|_{ \mathcal{H}} \geq \sqrt 2 r$, we may write
\begin{equation}
\label{eq:AStricAux4}
\begin{split}
 | \big(\chi \big(\frac{\| \vec v_1\|_{ \mathcal{H}}^2}{r^2}\big)
- \chi \big(\frac{\| \vec v_2\|_{ \mathcal{H}}^2}{r^2}\big) \big)| 
&\leq \int_0^1 |\chi' \big(\frac{\| \vec v_2 +\sigma (\vec v_1 -\vec v_2)\|_{ \mathcal{H}}^2}{r^2}\big) \big(\frac{\vec v_2 +\sigma (\vec v_1 -\vec v_2)}{r^2},(\vec v_1 -\vec v_2))_{\mathcal{H}} |d\sigma \cr
& \leq \frac{\sqrt 2}{r} \| \vec v_1 - \vec v_2\|_{\mathcal{H}}~.
\end{split}
\end{equation}
The estimate \eqref{eq:AStricAux4}, together with the estimates 
\eqref{eq:AStricAux2} and \eqref{eq:AStricAux3} with $v_2=0$, imply that
\begin{equation}
\label{eq:AStricAux5}
B_3 \leq 4 \sqrt 2 m  C^2C(\alpha, \tau_0)^2  [\tau_0 M_0 (rm)^\beta + 2 C(\alpha, \tau_0) 
\tau_0^{\eta} M_0(rm)^{\frac{\theta -1}{\theta}(\theta^*(1-\eta) +\eta)}]
\| \vec v_1 - \vec v_2\|_{L^{\infty}(\mathcal{H})}~.
\end{equation}
Choosing $r_0>0$ small enough so that
\begin{equation}
\label{eq:AStricAux6}
\begin{split}
K(r_0,\tau_0) \equiv & 2 C(\alpha, \tau_0)\tau_0  C  M_0 (2r_0)^\beta + 4 C(\alpha, \tau_0)  C^2 \tau_0^{\eta} 
M_0(2r_0)^{\frac{\theta -1}{\theta}(\theta^*(1-\eta) +\eta)} \cr
 +& 8 \sqrt 2  C^2  C(\alpha, \tau_0)^2  [\tau_0 M_0 (2r_0)^\beta + 2 C(\alpha, \tau_0) 
\tau_0^{\eta} M_0(2r_0)^{\frac{\theta -1}{\theta}(\theta^*(1-\eta) +\eta)}] \leq \frac{1}{4}~,
\end{split}
\end{equation}
we deduce from the inequalities \eqref{eq:AStricAux1} to \eqref{eq:AStricAux6} that
\begin{equation}
\label{eq:AStricAux7}
\| \mathcal{F}\vec v_1 - \mathcal{F} \vec v_2 \|_{Y} \leq \frac{1}{4} \| \vec v_1 - 
 \vec v_2 \|_{Y}~,
\end{equation}
 which implies with \eqref{eq:calF0*}, that, for any $\vec v_1 \in Y$,
 \begin{equation}
\label{eq:AStricAux8}
\| \mathcal{F}\vec v_1\|_{Y} \leq \frac{M_0(mr)}{2} ~.
\end{equation}
Therefore, $ \mathcal{F}$ is a strict contraction and admits a unique fixed point 
$\vec v(\vec v_0)$ in $Y$. The uniqueness of the solution 
$\vec v$ of the equation \eqref{eq:Auphi0mod} on the time interval $[0,\tau_0]$ is proved as in the proof of Theorem \ref{thm:WP}. \\
 Let next $\vec v_{0,i}$, $i=1,2$, be so that $\| \vec v_{0,i}\|_{ \mathcal{H}} \leq m r$, and let $\vec v_{i}$, $i=1,2$, be the corresponding solutions of 
the equation \eqref{eq:Auphi0mod} on the time interval $[0,\tau_0]$; by the above proof, they belong to $Y$. 
Applying Proposition \ref{AStrichphi0}  and repeating the above proof, we show that
 \begin{equation}
\label{eq:AStricAux9}
\| \vec v_1 - \vec v_2\|_{Y} \leq \frac{4}{3} C(\alpha,\tau_0) \| \vec v_{0,1} - \vec v_{0,2}\|_{Y} ~.
\end{equation} 
As in the proof of Theorem  \ref{thm:WP}, one also shows that 
$\vec v_0 \in B_{ \mathcal{H}}(0, mr) \mapsto \vec v(\vec v_0) \in Y$ is a 
$C^1$-function.\\
 In the remaining part of the proof, we set $m=2$.

\medskip

{\bf Step 2 :}   We begin by showing  that for every $\vec v_0 \in \mathcal{H}$, 
$\vec v(t) = \bar{S}_{\alpha}(t)\vec v_0$ exists on $[0,+\infty)$. Let first $\vec v_0 \in \mathcal{H}$ satisfying
$\| \vec v_{0}\|_{\mathcal{H}} \leq 2 r$, then, by Step 1, $\vec v(t)$ stays in the ball $B_{ \mathcal{H}}(0, M_0(2r))$  for $0 \leq t \leq \tau_0$. Let next  $\vec v_0 \in \mathcal{H}$ be such that 
$\| \vec v_{0}\|_{\mathcal{H}} \geq 2 r$ and let 
$\vec v(t)= \bar{S}_{\alpha}(t) \vec v_0$ be the mild local solution of \eqref{eq:Auphi0mod}. By continuity of this solution, there exists a time 
$t_1>0$ so that $\vec v(t) \notin B_{\mathcal{H}}(0, \sqrt 2 r)$, for $0 \leq t \leq t_1$. We have, for 
$0 \leq t \leq t_1$,
 \begin{equation}
\label{eq:AStricAux10}
\vec v(t) =  \tilde{\Sigma}_{\alpha}(t) \vec v_0~,
\end{equation}
and, in particular, for $0 \leq t \leq \inf(t_1,\tau_0)$,
\begin{equation}
\label{eq:AStricAux11}
\| \vec v(t)\|_{\mathcal{H}} + \| v\|_{L^{\theta^*}((0,t),L^{2\theta^*})} \leq 
C(\alpha,\tau_0) \| \vec v_0\|_{ \mathcal{H}}~.
\end{equation}
If at a time $t_1$, $\vec v(t_1)$ enters into the ball $B_{ \mathcal{H}}(0, 2r)$, then, according to Step 1, for $t_1\leq t \leq t_1 + \tau_0$, $\vec v(t)$ still exists, stays in the ball $B_{ \mathcal{H}}(0, M_0(2r))$ and satisfies the estimates given in Step~1.  We thus have proved that, for every $\vec v_0 \in \mathcal{H}$, $\vec v(t)$ exists on the time interval $[0,\tau_0]$. Consequently, for every $\vec v_0 \in \mathcal{H}$, 
$\bar{S}_{\alpha}(t)\vec v_0$ exists on $[0,+\infty)$. Likewise, one shows that 
$\bar{S}_{\alpha}(t)\vec v_0$ exists on $(-\infty,0])$. Arguing as in the proof  
of Theorem  \ref{thm:WP}, one shows the continuity properties of 
$\bar{S}_{\alpha}(t)\vec v_0$ with respect to $(t, \vec v_0)$ and the fact that, for any 
$t \in \R$, $\vec v_0 \in \mathcal{H} \mapsto \bar{S}_{\alpha}(t)\vec v_0 
\in \mathcal{H}$ is a $C^1$-map.

We are now able to prove that $\bar S_{\alpha}(t)$ satisfies the assumptions $(HA.3)$, 
$(HA.5.2)$, and $(HA.5.3)$.   We first prove the last part of assumption $(HA.3)$. 
 $\bar S_{\alpha}(t)$ is Lipschitz continuous, with a Lipschitz constant which is  uniform in $0 \leq t
\leq \tau_0$. The idea is that it is true if  $\vec v_{0,1}$ and 
$\vec v_{0,2}$ belong to $B_{ \mathcal{H}}(0, 2r)$ by \eqref{eq:AStricAux9}. If $\vec v_{0,2}\in B_{ \mathcal{H}}(0, 2r)$ and $\vec v_{0,1}\not \in B_{ \mathcal{H}}(0, 2r)$, we estimate the difference up to the first time $t_1\leq \tau_0$ when 
$\vec v_{1}(t)$ enters the ball $ B_{ \mathcal{H}}(0, 2r)$, and then apply the estimate proved in the first case up to time 
$\tau_0$. Finally, if both initial data are outside $B_{ \mathcal{H}}(0, 2r)$, we apply the linear estimates up to the first time when one solution enters $B_{ \mathcal{H}}(0, 2r)$ and then the estimate of the second case. As a consequence,  to conclude,
it remains to show that, if $ \| \vec v_{0,1}\|_{\mathcal{H}} \leq 2r$ and 
$\| \vec v_{0,2}\|_{\mathcal{H}} \geq 2r$ so that $ \| \vec v_2 (t) \|_{\mathcal{H}} \geq 2r$ for any $t \geq 0$, then $\vec v_1 - \vec v_2$ satisfies the estimate 
\eqref{eq:AStricAux9}. 
Using Proposition \ref{AStrichphi0}, the inequalities \eqref{Agy1y2}, 
\eqref{eq:AStricAux1}, and \eqref{eq:AStricAux5},  we obtain, for 
$0 \leq t \leq \tau_0$,
\begin{equation}
\label{eq:AStricAux12}
\begin{split}
&\|  \vec v_1 - \vec v_2\|_{Y} \cr
&\leq C(\alpha,\tau_0) \big[ \| \vec v_{0,1} - \vec v_{0,2} \|_{\mathcal{H}} +
\int _0^{\tau_0} \| h(\vec v_1(s)) ds \big] \cr
&\leq C(\alpha, \tau_0) \big[ \| \vec v_{0,1} - \vec v_{0,2} \|_{\mathcal{H}} +
 \int_0^{\tau_0}\|   g_0(v_1)
 \big(\chi \big(\frac{\| \vec v_1\|_{ \mathcal{H}}^2}{r^2}\big)
- \chi \big(\frac{\| \vec v_2\|_{ \mathcal{H}}^2}{r^2}\big) \big)\|_{L^2} ds \big]\cr
& \leq C(\alpha, \tau_0)  \| \vec v_{0,1} - \vec v_{0,2} \|_{\mathcal{H}} 
+  B_3~,
\end{split}
\end{equation}
 where $B_3$ had already been defined and used in \eqref{eq:AStricAux1}.
As before, the inequality \eqref{eq:AStricAux4} holds. Therefore, we deduce from the estimates \eqref{eq:AStricAux12}, \eqref{eq:AStricAux5} and the condition 
\eqref{eq:AStricAux6} that, for $0 \leq t \leq \tau_0$,
\begin{equation}
\label{eq:AStricAux13}
\begin{split}
\|  \vec v_1 - \vec v_2\|_{Y}
\leq C(\alpha,\tau_0) \| \vec v_{0,1} - \vec v_{0,2} \|_{\mathcal{H}} 
+ \frac{1}{4} \| \vec v_1 - \vec v_2\|_{Y}~.
\end{split}
\end{equation}
And thus the inequality  \eqref{eq:AStricAux9} holds. {}From all the above results, one 
infers that $\bar S_{\alpha}(t)$ is Lipschitz continuous and that
\begin{equation}
\label{eqLipschitzSbar}
  \sup_{0 \leq t \leq \tau_0} \hbox{ Lip }(\bar S_{\alpha}(t)) = D \leq \frac{16}{9} 
C^3(\alpha,\tau_0)~.
\end{equation}
Likewise, one shows that this estimate also holds for $-\tau_0, \leq t \leq 0$. Thus, Hypothesis $(HA.3)$ is satisfied.

\medskip

We next show that the hypotheses $(HA.5.2)$ and $(HA.5.3)$ hold. To this end, we set
\begin{equation}
\label{A:decompbar}
\begin{split}
\bar S_{\alpha}(\tau_0) &= \tilde{\Sigma}_{\alpha}(\tau_0) + R(\tau_0) \equiv L(\tau_0) + 
R(\tau_0) \cr
\bar S_{\alpha}(-\tau_0) &= \tilde{\Sigma}_{\alpha}(-\tau_0) + \tilde{R}(\tau_0) \equiv 
L(\tau_0)^{-1} + \tilde{R}(\tau_0)~.
\end{split}
\end{equation}
Let  $\vec v_0 \in \mathcal{H}$ and $\vec v(t) = \bar S_{\alpha}(t) \vec v_0$; then,  
$R(\tau_0)$ writes 
\begin{equation}
\label{eq:AR0}
R(\tau_0) = \int_0^{\tau_0}\tilde{\Sigma}_{\alpha}(t-s) (0, h(v(s)))^t ds~.
\end{equation}
To prove that the conditions \eqref{eq:RHcus},   \eqref{eq:RtildeHcsu}, and 
\eqref{eq:HA5dernier} hold, we will show that Lip$(R(\tau_0))$ and
 Lip$(\tilde{R}(\tau_0))$ go to zero as $r_0$ goes to zero (we will only show it for $R(\tau_0)$, since the proof is similar for 
$\tilde{R}(\tau_0)$).
To show this property, we are going back to the three cases considered above. If 
$ \vec v_{0,1}$ and $ \vec v_{0,2}$ belong to $B_{ \mathcal{H}}(0, 2r)$, then the estimates 
\eqref{eq:AStricAux1} to \eqref{eq:AStricAux9} imply that
\begin{equation}
\label{eq:ALipR1}
 \|R(\tau_{0})\vec v_{0,1} - R(\tau_{0}) \vec v_{0,2} \|_{Y}  \leq \frac{4}{3} 
K(r_0, \tau_0) C(\alpha,\tau_0) \| \vec v_{0,1} - \vec v_{0,2}\|_{ \mathcal{H}}~.
\end{equation}
The estimate \eqref{eq:AStricAux12} shows that the same property \eqref{eq:ALipR1}
holds if $\vec v_{0,1}$ belongs to $B_{ \mathcal{H}}(0, 2r)$ and $\vec v_{0,2}$ is so that  $ \| \vec v_2 (t) \|_{\mathcal{H}} \geq 2r$ for any $0 \leq t \leq \tau_0$. Finally, we remark that if 
$\vec v_{i} (t) \notin B_{ \mathcal{H}}(0, 2r)$, $i=1,2$, for $0 \leq t \leq \tau_0$, then 
$R(\tau_{0})\vec v_{0,1} - R(\tau_{0}) \vec v_{0,2}=0$. Combining all the above cases and using the estimate \eqref{eqLipschitzSbar}, we finally  obtain that, in every case,
\begin{equation}
\label{eq:ALipR1BIS}
 \|R(\tau_{0})\vec v_{0,1} - R(\tau_{0}) \vec v_{0,2} \|_{Y}  \leq \frac{16}{9} 
K(r_0, \tau_0) C^3(\alpha,\tau_0) \| \vec v_{0,1} - \vec v_{0,2}\|_{ \mathcal{H}}~.
\end{equation}
Since $K(r_0, \tau_0)$ goes to zero as $r_0$ goes to zero, Lip$(R(\tau_0))$ goes to zero as $r_0$ goes to zero and the condition \eqref{eq:RHcus} is satisfied provided $r_0$ is chosen small enough. Likewise the conditions  \eqref{eq:RtildeHcsu} and \eqref{eq:HA5dernier} hold, provided $r_0$ is chosen small enough. {} From now on, we fix  $r_0 >0$ small enough so that these conditions are satisfied and we choose $r=r_0$ in  \eqref{eq:Auphi0mod}. 
\medskip

We have seen that, for $r_0>0$ small enough, $\bar{S}_{\alpha}(t)$ satisfies the hypotheses of  Theorems  \ref{th:FoliaA2} and \ref{th:FoliaA4}. 
We can thus state the following result concerning the invariant manifolds of  
$\bar{S}_{\alpha}(t)$.  For the notations and definitions of the different invariant manifolds, we refer the reader to Appendix A below.\\
As in the assumption (HA.5.1), we denote by
 $P_{i}$  the spectral (continuous) projection associated to the spectral set 
$\sigma^{i}$  and let $\mathcal{H}_{rad,i}$ be the image 
$\mathcal{H}_{rad,i} = P_i \mathcal{H}_{rad}$, 
where $i= cu, cs , u, s, c$.

\begin{theorem} \label{AVarieteCenInst} Let $\alpha>0$ be fixed. \\
1) There exists a $C^1$ globally Lipschitz continuous map 
$g_{cu}: \mathcal{H}_{rad,cu} \to \mathcal{H}_{rad,s}$ so that the $C^1$ center unstable manifold $W^{cu}((0,0))$ of $\bar{S}_{\alpha}(t)$ at  $(0,0)$
$$
W^{cu}((0,0)) = \{\vec v_{cu} +g_{cu}(\vec v_{cu}) \, | \, \vec v_{cu} \in 
\mathcal{H}_{rad,cu} \}
$$ 
satisfies all the properties given in Theorem \ref{th:FoliaA1}. 

\smallskip

2) There exists a $C^1$ globally Lipschitz continuous map 
$g_{u}: \mathcal{H}_{rad,u} \to \mathcal{H}_{rad,cs}$ so that the $C^1$ (strongly) unstable manifold $W^{u}((0,0))$ of $\bar{S}_{\alpha}(t)$ at  $(0,0)$
$$
W^{u}((0,0)) = \{\vec v_{u} +g_{u}(\vec v_{u}) \, | \, \vec v_{u} \in 
\mathcal{H}_{rad,u} \}
$$ 
satisfies all the properties described in the statement (2) of  Theorem \ref{th:FoliaA4}.

\smallskip

3) Moreover, there exists a continuous mapping $\ell: \mathcal{H}_{rad} \times 
\mathcal{H}_{rad,s} \to \mathcal{H}_{rad,cu}$, such that, for any $\vec v 
\in \mathcal{H}_{rad}$, the  manifold $\mathcal{M}_{\vec v}= \{ \vec v + \ell( \vec v, \vec v_s) \, | \,  \vec v_s \in \mathcal{H}_{rad,s}\}$ satisfies all the properties in Theorem \ref{th:FoliaA2}. In particular, $\{ \mathcal{M}_{\vec \xi} \, | \, 
\vec \xi \in W^{cu}((0,0))\}$ is a foliation of $\mathcal{H}_{rad}$ over $W^{cu}((0,0))$.

\smallskip

4) In particular, there exist $\tilde{c} >1$, $0 <\rho_0 <1$, and, for any $\vec v_0 \in \mathcal{H}_{rad}$, a unique element 
$\vec \xi( \vec v_0) \in W^{cu}((0,0))$ such that, for $t \geq 0$,
 \begin{equation}
\label{eq:Aattracexp1}
\| \bar S_{\alpha}(t) \vec v_0 -  \bar S_{\alpha}(t) \vec \xi( \vec v_0) \|_{\mathcal{H}} \leq \tilde{c} \rho_0^t \| \vec v_0 - 
 \vec \xi( \vec v_0)\|_{\mathcal{H}}~.
\end{equation}
Moreover, the map $\vec v_0 \in \mathcal{H}_{rad} \mapsto \vec\xi( \vec v_0) \in 
W^{cu}((0,0))$ is continuous.

\smallskip

5) There exists a $C^1$ globally Lipschitz continuous map $g_{c}: \mathcal{H}_{rad,c} \to \mathcal{H}_{rad,s}\oplus \mathcal{H}_{rad,u}$ with $g_{c}(0)=0$, so that the center  manifold $W^{c}(0)$ of $\bar{S}_{\alpha}(t)$ at  $(0,0)$
$$ 
W^{c}((0,0)) = \{ x_c + g_{c}(x_c) \, | \, x_c \in \mathcal{H}_{rad,c}\} =
 W^{cu}((0,0)) \cap W^{cs}((0,0))
$$
satisfies  all the  properties given in statement (4) of Theorem \ref{th:FoliaA4}. 
\end{theorem}

Let us go back to the ``actual''  variable $\vec u= \vec v + (\varphi_0, 0)^t$. We set
$$
S^*_{\alpha}(t) \vec u_0= (\varphi_0,0)^t + \bar{S}_{\alpha}(t) 
(\vec u_0  - (\varphi_0,0))~.
$$ 
Then the invariant manifolds of $S_{\alpha}^*(t)$ are defined by
\begin{equation}
\label{eq:Wi}
W^{i *}((\varphi_0,0)) = (\varphi_0,0)^t + W^{i}((0,0))~, i=c
u, c, u, s~.
\end{equation}

 \begin{remark} \label{rem:SEgalSbarre}
We emphasize that the proof given in Step 1 above shows that if, for example, $r=r_0$,
$m= (8 C(\alpha,\tau_0))^{-1}$, and $\| \vec u_0\|_{ \mathcal{H}} \leq mr_0$,  then, for $0 \leq t \leq \tau_0$, 
$$
 \| \bar{S}_{\alpha}(t) \vec u_0\|_{Y} \leq r_0/2~,
 $$
 which implies that, for $0 \leq t \leq \tau_0$, $\bar{S}_{\alpha}(t) \vec u_0 
 = S_{\alpha}(t) \vec u_0$. In other terms, if $\vec u_0$ belongs to the 
 ball  $B_{ \mathcal{H}_{rad}}((\varphi_0,0),r_1)$ of center 
 $(\varphi_0,0)$ and radius $r_1 \leq (8 C(\alpha,\tau_0))^{-1}r_0$, then 
 $S^*_{\alpha}(t) \vec u_0 = S_{\alpha}(t) \vec u_0$.  This allows one to define the local invariant manifolds $W^i_{loc}((\varphi_0,0))$  of $S_{\alpha}(t)$ about $(\varphi_0,0)$ as
 \begin{equation}
\label{eq:4Wiloc}
W^i_{loc}((\varphi_0,0)) = W^{i *}((\varphi_0,0)) \cap 
 B_{ \mathcal{H}_{rad}}((\varphi_0,0),r_1) ~, i=cu, c, u, s~.
\end{equation}
 \end{remark}

\begin{remark} \label{prop:AVarieteInst}
1) In the above theorem, $\mathcal{M}_0$ coincides with the (strongly) stable manifold 
$\tilde{W}^{s}((0,0))$. 
2) If $\ker(\calL)= \{0\}$, then the center unstable manifold $W^{cu}((0,0))$ coincides with the unstable manifold $W^{u}((0,0))$ of $(0,0)$, while $\mathcal{M}_0$ coincides with the stable manifold $W^{s}((0,0))$. 
\end{remark}

\begin{remark} \label{prop:AVarietealpha0} 
In the case where $\alpha=0$, we can also apply Theorems  \ref{th:FoliaA1} and 
\ref{th:FoliaA2} below in order to prove the existence of the strong unstable manifold and the existence of a center stable manifold around any equilibrium point of \ref{KGalpha} as well as the existence of a foliation of $\mathcal{H}_{rad}$ over the unstable manifold. This gives an alternative proof to the construction of a center stable manifold, by the Hadamard method  in \cite{NaS} (for more details, see \cite{BRS2}). 
\end{remark}


\appendix
\section{Global invariant manifolds and foliations by the Lyapunov-Perron method}
\label{sec:varietes}

 In this appendix, we recall the basic properties of  invariant manifold theory 
that we applied to the equation \ref{KGalpha} in Section \ref{sec:IMforKG}. We reproduce the theorems of Chen, Hale and Tan about global invariant manifolds and foliations as given in \cite{ChHaBT}.
 For classical results on invariant manifolds, we also refer the reader to the books \cite{Carr}, \cite{Henry}, \cite{HPS77}, and \cite{Palis-Melo} for example as well as to \cite{BaJon} and  to \cite{ChowLL}.
\vskip 3mm

Let $X$ be a Banach space with norm $\| \cdot\|_{X}$ and $S(t) : X \to X$ be a non-linear semigroup, satisfying the following hypotheses:
\begin{description}
\item [{\bf (HA.1)} ]$S(.).: (t,x) \in [0, + \infty) \times X \mapsto S(t)x \in X$ is continuous and there exists a constant $\tau_0 >0$ such that, 
$$
 \sup_{0 \leq t \leq \tau_0} \mathrm{Lip}(S(t)) =D < +\infty~.
$$

\item [{\bf (HA.2)}] There exists $\tau$, $0 < \tau \leq \tau_0$ such that $S(\tau)$ can be decomposed as 
$$
 S(\tau) = L + R ~, 
 $$
 where $L: X \to X$ is a bounded linear operator and $R : X \to X$ is a global Lipschitz continuous map, satisfying the following properties.

\item [{\bf (HA.2.1)}] There are subspaces $X_{i}$, $i=1, 2$, of $X$ and continuous projections 
$P_{i}: X \to X_{i}$ such that $P_1 + P_2 = I$, $X = X_{1} \oplus X_{2}$, $L$ leaves $X_{i}$, $i =1, 2$, invariant and $L$ commutes with  $P_{i}$, $i =1, 2$. The restrictions $L_{i}$ of $L$ to $X_{i}$ satisfy the following properties. The map $L_1$ has a bounded inverse and there exist constants $ 0 \leq \beta_2 < \beta_1$, $C_i \geq 1$, $i =1, 2$, such  that,
for $k \geq 0$,
\begin{equation}
\label{eq:AL1L2}
\begin{split}
\| L_1^{-k}P_1 \|_{L(X,X)} \leq C_1\beta_1^{-k}~, \cr
\| L_2^{k}P_2\|_{L(X,X)} \leq C_2 \beta_2^{k}~.
\end{split}
\end{equation}

\item [{\bf (HA.2.2)}] The maps $L$ and $R$ satisfy the condition 
\begin{equation}
\label{eq:ALRinf1}
\frac{(\sqrt{C_1} + \sqrt{C_2})^2}{\beta_1 - \beta_2} \mathrm{Lip}(R) < 1~.
\end{equation}
\end{description}

 Chen, Hale and Tan considered the following quantity, for $\gamma \in (\beta_2,\beta_1)$,
 \begin{equation}
\label{eq:Alambda}
 \lambda(\gamma) = \frac{C_1}{\beta_1 - \gamma} + \frac{C_2}{\gamma - \beta_2}~.
\end{equation}
 A short computation shows that, under the condition \eqref{eq:ALRinf1}, there exist 
 $\gamma_i$, $i=1, 2$, with $\beta_2 < \gamma_2 <\gamma_1< \beta_1$ such that,
 \begin{equation}
 \label{AChoixgamma}
 \lambda(\gamma_1) \mathrm{Lip}(R)=  \lambda(\gamma_2) \mathrm{Lip}(R)=1~, 
 \hbox{ and } \lambda(\gamma) \mathrm{Lip}(R) <1~, \quad \forall \gamma \in (\gamma_2, \gamma_1)~.
\end{equation}
In the trivial case, where $\mathrm{Lip}(R) =0$, one sets $\gamma_1=\beta_1$ and $\gamma_2=\beta_2$.

We are now able to state the first theorem, concerning the existence of an invariant  manifold, which is a graph over $X_1$. 

\begin{theorem} \label{th:FoliaA1}
Assume that the hypotheses (HA.1), (HA.2) hold and that $R(0)=0$. Then there exists a globally Lipschitz  map $g: X_1 \to X_2$ with $g(0)=0$, and 
\begin{equation}
\label{eq:Lipg}
\mathrm{Lip}(g) \leq \min_{\gamma_2 \leq \gamma \leq \gamma_1} \frac{C_1 C_2  \mathrm{Lip}(R)
 \gamma} {\beta_1 (\gamma - \beta_2) (1 -\lambda(\gamma) \mathrm{Lip}(R))}~,
\end{equation} 
so that the Lipschitz  submanifold 
$$ 
G = \{ x_1 + g(x_1) \, | \, x_1 \in X_1\}
$$
satisfies the following properties: 
\begin{description}
\item[(i)] \textrm{(Invariance)} The restriction to $G$ of the semi-flow $S(t)$, $t \geq 0$, can be extended to a Lipschitz continuous flow on $G$. In particular, $S(t)G=G$, for any 
$t \geq 0$, and for any $\xi \in G$, there exists a unique negative semi-orbit $u(t) \in G$ of $S(.)$, $t \leq 0$, so that $u(0)=\xi$. 

\item[(ii)] \textrm{(Lyapunov exponent)} If a negative semi-orbit  $u(t)$, $ t\leq 0$, of 
$S(.)$ is contained in $G$, then,
\begin{equation}
\label{eq:ALyExp1}
\limsup_{t \to -\infty} \frac{1}{|t|}\ln |u(t)| \leq - \frac{1}{\tau}\ln \gamma_1~.
\end{equation}
Conversely, if a negative semi-orbit  $u(t)$, $ t\leq 0$, of $S(.)$ is contained in $X$ satisfies
\begin{equation}
\label{eq:ALyExp2}
\limsup_{t \to -\infty} \frac{1}{|t|}\ln |u(t)| < - \frac{1}{\tau}\ln \gamma_2~.
\end{equation}
then, it is contained in $G$.
\item[(iii)] \textrm{(Smoothness)} If the map $S(\tau): X \to X$ is of class $C^1$, then $g: X_1 \to X_2$ is of class $C^1$, that is, $G$ is a $C^1$-submanifold of $X$. 
\end{description}
\end{theorem}

The second theorem states the existence of a foliation of $X$ over the invariant manifold 
$G$. 

\begin{theorem} \label{th:FoliaA2}
Assume that the hypotheses (HA.1), (HA.2) hold and that $R(0)=0$. Then, there exists an invariant foliation of $X$ over $G$ as follows.
\begin{description}
\item[(i)] \textrm{(Invariance)} There exists a continuous mapping $\ell: X \times X_2 \to X_1$ such that, for any $\xi \in G$, $\ell(\xi,P_2\xi)=P_1 \xi$ and the manifold 
$\mathcal{M}_{\xi}= \{x_2 + \ell(\xi,x_2) \, | \, x_2 \in X_2 \}$ passing through $\xi$ satisfies:
\begin{equation}
\label{eq:AfoliaInv}
S(t) \mathcal{M}_{\xi} \subset \mathcal{M}_{S(t)\xi}~, \quad t \geq 0~,
\end{equation}
and
 \begin{equation}
\label{eq:AfoliaLya}
  \mathcal{M}_{\xi}  = \{ y \in X\, | \, \limsup_{t \to \infty} \frac{1}{t} 
  \ln |S(t)y - S(t)\xi |  \leq \frac{1}{\tau} \ln \gamma_2 \} ~.
 \end{equation}
Moreover, the map $\ell: X \times X_2 \to X_1$ is uniformly Lipschitz continuous  in the $X_2$ direction.
\item[(ii)] \textrm{(Completeness)} Suppose in addition that
\begin{equation}
\label{eq:AgamRinf1}
\begin{split}
\big[\min_{\gamma_2 \leq \gamma \leq \gamma_1} \frac{C_1 C_2  \mathrm{Lip}(R) }
{ (\beta_1 -\gamma) (1 -\lambda(\gamma) \mathrm{Lip}(R))}    \big] 
\cdot
 \big[\min_{\gamma_2 \leq \gamma \leq \gamma_1} \frac{C_1 C_2  \mathrm{Lip}(R)
 \gamma} {\beta_1 (\gamma - \beta_2) (1 -\lambda(\gamma) \mathrm{Lip}(R))} \big] <1.
\end{split}
\end{equation}
Then, for any $x \in X$, $\mathcal{M}_x \cap G$ consists of a single point. In particular,
\begin{equation}
\label{eq:AfoliaProp}
\mathcal{M}_{\xi} \cap \mathcal{M}_{\eta} = \emptyset~, \quad \forall \xi, \eta \in G~, \quad \quad 
X = \bigcup_{\xi \in G} \mathcal{M}_{\xi}~. 
\end{equation}
In other terms, $\{\mathcal{M}_{\xi} \, | \, \xi \in G\}$ is a foliation of $X$ over $G$.\\
Moreover, the mapping $x \in X \mapsto \xi(x)= \mathcal{M}_x \cap G$ is a continuous map from $X$ into $G \subset X$. 

\item[(iii)] \textrm{(Smoothness)} If the map $S(\tau): X \to X$ is of class $C^1$, then 
$\ell: X \times X_2 \to X_1$ is of class $C^1$ in the $X_2$ direction. Hence, $\mathcal{M}_{\xi}$ is a $C^1$-submanifold of $X$, for any $\xi \in G$.
\end{description}
\end{theorem}

\noindent {\bf Comments on the proof of Theorems \ref{th:FoliaA1} and \ref{th:FoliaA2}:}\\
Theorems \ref{th:FoliaA1} and \ref{th:FoliaA2} are proved in \cite{ChHaBT} by first showing the corresponding results for the map $S(\tau)$ and at the end  coming back to the continuous dynamical system. This means that Theorems \ref{th:FoliaA1} and \ref{th:FoliaA2} still hold for iterates of maps $S(\tau)$. It suffices to replace $t \in \R$ by $n\tau$, $n \in \mathbb{N}$. Theorems \ref{th:FoliaA1} and \ref{th:FoliaA2} are proved in
\cite{ChHaBT} by using the Lyapunov-Perron method. 
\vskip 1mm

The property that the mapping $x \in X \mapsto \xi(x) = \mathcal{M}_x \cap G$  is a continuous map from $X$ into $G \subset X$ is not stated in the main Theorem 1.1 of
\cite{ChHaBT}. It is merely a consequence of the proof of \cite[Lemma 3.4]{ChHaBT}. Indeed, given $x \in X$, the intersection points $\xi(x)$ of $\mathcal{M}_x$ with $G$ are the solutions of 
 \begin{equation}
\label{eq:xix}
\xi(x) \equiv y_2 + \ell(x,y_2) = \ell(x,y_2) + g(\ell(x,y_2))~,
\end{equation}
where $y_2 \in X_2$.  This leads to study the fixed points of the map $F_x(y_2) \equiv F(x, y_2) = g(\ell(x,y_2))$, depending on the parameter $x \in X$. One can check that the condition \eqref{eq:AgamRinf1} implies that $F_x: X_2 \to X_2$ is a strict contraction and therefore has a unique fixed point $y_2(x)$. The continuity property of $y_2(x)$ with respect to $x \in X$ is a direct consequence of the continuity of $F$ with respect to the variable 
$x \in X$ and of the {\em uniform contraction principle} (see 
\cite[Theorem 2.2 on Page 25]{ChowHale}). 
 It follows that $\xi(x) = y_2(x) + \ell(x, y_2(x)) \in G$ is also continuous with respect to $x \in X$.

\begin{remark} \label{rem:A3}
 If the equilibrium point $0$ of $S(.)$ is hyperbolic, then we may choose 
$\beta_2 < 1 <\beta_1$. In this case, $G$ is the classical unstable manifold $W^{u}(0)$ and $M_{\xi}$, $\xi \in G$, defines an invariant foliation of $X$ over $W^{u}(0)$, with 
$\mathcal{M}_0$ being the classical stable manifold $W^{s}(0)$. And the solutions on 
$\mathcal{M}_0$ decay exponentially to $0$, as $t$ goes to $+\infty$. \\
 If $0$ is a non-hyperbolic equilibrium point and $\beta_2  <\beta_1 < 1$ with $\beta_1$ close to $1$, then Theorems 
\ref{th:FoliaA1} and \ref{th:FoliaA2} allow for the construction of the center-unstable manifold 
$G = W^{cu}(0)$ of $0$ and a foliation over it.  If $0$ is a non-hyperbolic equilibrium point and $1 <\beta_2  <\beta_1$ with $\beta_2$ close to $1$, then Theorems 
\ref{th:FoliaA1} and \ref{th:FoliaA2} give the strongly unstable manifold
$G = W^{su}(0)$ of $0$ and a foliation over it. If $\gamma_2 <1$, the existence of the foliation implies that each positive semi-orbit of $S(t)$ converges exponentially to an orbit of $G$ and is synchronized with this orbit in time. This property is often called ``attraction" of $G$ with asymptotic phase". \\
We emphasize that the construction in Theorems \ref{th:FoliaA1} and \ref{th:FoliaA2}
is also interesting in the case where $S_{\alpha}(.)$ depends on a parameter $\alpha$ and  
$\beta_2(\alpha)< 1  <\beta_1(\alpha)$ with $\beta_2(\alpha)$ arbitrarily close to $1$ as $\alpha$ converges  say to $\alpha_0=0$. 
\end{remark}

\medskip
Mutatis mutandis, repeating the  arguments of the proofs of Theorems 
\ref{th:FoliaA1} and \ref{th:FoliaA2}, one can also show the existence of a 
Lipschitz  manifold $\tilde{G}= \{ x_2 + \tilde{g}(x_2) \, | \, x_2 \in X_2\}$ where 
$\tilde{g}: X_2 \to X_1$  is a globally Lipschitz  map  with $\tilde{g}(0)=0$, such that 
$\tilde{G}$ is  invariant and such that, if a semi-orbit $u(t)$, $ t\geq 0$, of $S(.)$ is contained in 
$\tilde{G}$, then,
\begin{equation}
\label{eq:ALyExp1tilde}
\limsup_{t \to \infty} \frac{1}{t}\ln |u(t)| \leq  \frac{1}{\tau}\ln \tilde{\gamma}_2^{-1}~,
\end{equation}
where $\beta_2 < \tilde{\gamma}_2^{-1} < \tilde{\gamma}_1^{-1} < \beta_1$ is made more precise below, 
and also the existence of a foliation $\tilde{\mathcal{M}}_{\xi}$ (in reverse time) of $X$ 
over $\tilde{G}$. 

If $S(t)$ is a non-linear group, these properties can be proved by reversing the time in 
Theorems \ref{th:FoliaA1} and \ref{th:FoliaA2}. In Section~\ref{sec:core}, the existence of a center manifold  played an important role. 
We can derive this existence by defining the center manifold as the intersection of the center stable and center unstable manifolds. The center stable manifold is constructed like the Lipschitz  manifold $\tilde{G}= \{ x_2 + \tilde{g}(x_2) \, | \, x_2 \in X_2\}$ described above. Since throughout the paper we are only dealing with groups,  we will quickly show the existence of  $\tilde{G}$ by reversing the time in Theorem \ref{th:FoliaA1}. The constants appearing in the proof below are maybe not optimal, but we are not looking here for optimality.

\medskip

In addition to the hypothesis (HA.2), we assume now that 
\begin{description}
\item [{\bf (HA.3)} ] $S(.).: (t,x) \in (-\infty, + \infty) \times X \mapsto S(t)x \in X$ is continuous and there exists a constant $\tau_0 >0$ such that, 
$$
 \sup_{-\tau_0 \leq t \leq \tau_0} \mathrm{Lip}(S(t)) =D < +\infty~.
$$

\item [{\bf (HA.4)}] $S(-\tau)$ can be decomposed as 
$$
 S(-\tau) = L ^{-1}+ \tilde {R}  ~, 
 $$
 where $\tau$ and $L: X \to X$ have been introduced in the hypothesis (HA.2) and 
 where $\tilde{ R} : X \to X$ is a global Lipschitz continuous map, satisfying the following property:
\begin{equation}
\label{eq:RtildeH}
\frac{(\sqrt{C_1} + \sqrt{C_2})^2}{\beta_1 - \beta_2} \beta_1 \beta_2
\mathrm{Lip}(\tilde{R}) < 1~.
\end{equation}
\end{description}

We remark that the linear map $L^{-1}$ satisfies the hypothesis (HA.2.1) with $P_1$ (resp.\ 
$P_2$) replaced by $P_2$ (resp.\   $P_1$),  $C_1$ (resp.\ 
$C_2$) replaced by $C_2$ (resp.\  $C_1$), and $\beta_1$ (resp.\  $\beta_2$) replaced by
$\beta_2^{-1}$ (resp.\  $\beta_1^{-1}$).  Indeed, we have
\begin{equation}
\label{eq:AL1L2moins}
\begin{split}
\|( L^{-1})^{-k}P_2 \|_{L(X,X)} \leq C_2 (\beta_2^{-1})^{-k}~, \cr
\| ( L^{-1})^{k}P_1\|_{L(X,X)} \leq C_1 (\beta_1^{-1})^{k}~.
\end{split}
\end{equation}
We next set
\begin{equation}
\label{eq:Alambdatilde}
 \tilde{\lambda}(\tilde{\gamma}) = \frac{C_2}{\beta_2^{-1} -  \tilde{\gamma}} + 
 \frac{C_1}{\tilde{\gamma} - \beta_1^{-1}}~.
\end{equation}
As above, a short computation shows that, under the condition \eqref{eq:RtildeH}, there exist 
 $\tilde{\gamma_i}$, $i=1, 2$, with $\beta_1^{-1} < \tilde{\gamma_1} <
 \tilde{\gamma_2}< \beta_2^{-1}$ such that,
 \begin{equation}
 \label{AChoixgammatilde}
\tilde{\lambda}( \tilde{\gamma_1}) \mathrm{Lip}(\tilde{R})=  \tilde{\lambda}( \tilde{\gamma_2}) \mathrm{Lip}(\tilde{R})=1~, 
 \hbox{ and } \tilde{\lambda}(\tilde{\gamma_1}) \mathrm{Lip}(\tilde{R}) <1~, \quad \forall \tilde{\gamma}\in (\tilde{\gamma_1}, \tilde{\gamma_2})~.
\end{equation}
We may now apply Theorem \ref{th:FoliaA1} to the nonlinear semigroup $\tilde{S}(t) = S(-t)$ and we obtain the following result.

\begin{theorem} \label{th:FoliaA3}
Assume that the hypotheses (HA.2), (HA.3), and (HA.4) hold and that $R(0)= 
\tilde{R}(0)= 0$. Then there exists a globally Lipschitz  map $\tilde{g}: X_2 \to X_1$ with $\tilde{g}(0)=0$ and 
\begin{equation}
\label{eq:Lipgtilde}
\mathrm{Lip}(\tilde{g}) \leq  \min_{\tilde{\gamma_1} \leq  \tilde{\gamma} \leq 
\tilde{\gamma_2}} \frac{C_1 C_2  \mathrm{Lip}(\tilde{R}) \beta_1 \beta_2} 
{(\beta_1 - 1/\tilde{\gamma} ) (1- \tilde{\lambda}(\tilde{\gamma})\mathrm{Lip}(\tilde{R}))}~,
\end{equation}
so that the Lipschitz  submanifold 
$$ 
\tilde{G} = \{ x_2 + \tilde{g}(x_2) \, | \, x_2 \in X_2\}
$$
satisfies the following properties: 
\begin{description}
\item[(i)] \textrm{(Invariance)}  $\tilde{G}$ is invariant under $S(t)$, i.e.,  $S(t)\tilde{G}=
\tilde{G}$, for any $t \geq 0$. \\

\item[(ii)] \textrm{(Lyapunov exponent)} If a positive semi-orbit  $u(t)$, $ t\geq 0$, of 
$S(.)$ is contained in $\tilde{G}$, then,
\begin{equation}
\nn
\limsup_{t \to \infty} \frac{1}{t}\ln |u(t)| \leq  \frac{1}{\tau}\ln 
\frac{1}{\tilde{\gamma_2}}~.
\end{equation}
Conversely, if a positive semi-orbit  $u(t)$, $ t\geq 0$, of $S(.)$ in $X$, satisfies
\begin{equation}
\label{eq:ALyExp2tilde}
\limsup_{t \to \infty} \frac{1}{t}\ln |u(t)| < \frac{1}{\tau}\ln \frac{1}{\tilde{\gamma_1}}~.
\end{equation}
then, it is contained in $\tilde{G}$.
\item[(iii)] \textrm{(Smoothness)} If the map $S(\tau): X \to X$ is of class $C^1$, then 
$\tilde{g}: X_2 \to X_1$ is of class $C^1$, that is, $\tilde{G}$ is a $C^1$-submanifold of 
$X$. 
\end{description}
\end{theorem}

We next consider the classical case where $S(.)$ is a non-linear group satisfying the assumption (HA.3) as well as  
\begin{description}
\item [{\bf (HA.5)}] The point $0$ is an equilibrium point of $S(.)$. And there exists 
$\tau$, $0 < \tau \leq \tau_0$ such that $S(\tau)$ and $S(-\tau)$ can be decomposed as follows
$$
S(\tau) = L+ R~, \quad S(-\tau) = L^{-1} + \tilde{R}~,
$$
where $L: X \to X$ is a bounded linear operator, $R : X \to X$  and $\tilde{R} : X \to X$
are global Lipschitz continuous maps, satisfying the following properties.

\item [{\bf (HA.5.1)}] The spectrum $\sigma(L)$ of $L$ can be written as 
$$
\sigma(L) =  \sigma^{s} \cup \sigma^{c} \cup \sigma^{u}~,
$$
where $\sigma^{s}$, $\sigma^{c}$ and  $\sigma^{u}$ are closed subsets of 
$\{ \lambda \in \mathbb{C} \, | \,
 |\lambda| < 1\}$, $\{ \lambda \in \mathbb{C} \, | \, |\lambda| = 1\}$, and 
 $\{ \lambda \in \mathbb{C} \, | \, |\lambda| > 1\}$. 
 \end{description}
 There exists $\eta >0$ such that 
\begin{equation}
\label{eq:SousEnsheta}
\sigma^{s} \subset  \{ \lambda \in \mathbb{C} \, | \,  |\lambda| < 1 -\eta\}~, \quad
\sigma^{u} \subset  \{ \lambda \in \mathbb{C} \, | \, |\lambda| > 1 + \eta\}
\end{equation}
We set: $\sigma^{cu} = \sigma^{c} \cup \sigma^{u}$ and $\sigma^{cs}=
 \sigma^{c} \cup \sigma^{s}$.
Let $P_{i}$  be the spectral (continuous) projector associated to the spectral set 
$\sigma^{i}$  and let $X_i$ be the image $X_i = P_i X$, 
where $i= cu, cs , u, s, c$. We  have that $P_{cu} +P_s=I = P_{cs} + P_u$. The linear map 
$L$ leaves $X_{i}$ invariant and  commutes with $P_i$, $i = cu, cs, u, s,c$. \\
Now we choose $0 <\varepsilon < \eta/2$. The restrictions $L_i$ of $L$ to $X_i$ satisfy the following properties. There exist  constants $C_1 \geq1$ and
$C_2 \geq 1$ such that, for $k \geq 0$,
\begin{equation}
\label{eq:estcus}
\begin{split}
\| L_{cu}^{-k}P_{cu} \|_{L(X,X)} \leq  & C_1 (1-\varepsilon)^{-k}~, \cr
\| L_{s}^{k}P_s\|_{L(X,X)} \leq & C_2 (1-\eta)^{k}~,
\end{split}
\end{equation}
and
\begin{equation}
\label{eq:AL1L2moins'}
\begin{split}
\|( L_{cs}^{-1})^{-k}P_{cs} \|_{L(X,X)} \leq C_2 ((1+\varepsilon)^{-1})^{-k}~, \cr
\| ( L_{u}^{-1})^{k}P_u\|_{L(X,X)} \leq C_1 (( 1+ \eta )^{-1})^{k}~.
\end{split}
\end{equation}
We further assume that the maps $R$ and $\tilde{R}$ satisfy the conditions.
\begin{description}
\item [{\bf (HA.5.2)}] 
The following inequalities hold
\begin{equation}
\label{eq:RHcus}
\frac{(\sqrt{C_1} + \sqrt{C_2})^2}{\eta -\varepsilon}
\mathrm{Lip}(R) < 1~,
\end{equation}
and
\begin{equation}
\label{eq:RtildeHcsu}
\frac{(\sqrt{C_1} + \sqrt{C_2})^2}{\eta -\varepsilon} (1+\varepsilon)(1+ \eta)
\mathrm{Lip}(\tilde{R}) < 1~.
\end{equation}
\item [{\bf (HA.5.3)}] 
We define the function $\lambda(\gamma)$ as  in \eqref{eq:Alambda}, that is,
\begin{equation}
\label{eq:Alambdacus}
 \lambda(\gamma) = \frac{C_1}{1-\varepsilon - \gamma} + \frac{C_2}{\gamma - 
1+\eta}~,
\end{equation}
and the quantities $\gamma_i$, $i=1,2$, with 
$1-\eta < \gamma_2 < \gamma_1 < 1-\varepsilon$, satisfying  \eqref{AChoixgamma}. 
Likewise, we define  the function $\tilde{\lambda}(\tilde{\gamma})$ as  in 
\eqref{eq:Alambdatilde}, that is,
\begin{equation}
\label{eq:Alambdacsu}
 \tilde{\lambda}(\tilde{\gamma}) = \frac{C_2}{(1+\varepsilon)^{-1} -  \tilde{\gamma}} +  \frac{C_1}{\tilde{\gamma} - (1+\eta)^{-1}}~.
\end{equation}
and the quantities $\tilde{\gamma}_i$, $i=1,2$, with 
$(1+\eta)^{-1} < \tilde{\gamma}_1 < \tilde{\gamma}_2 < (1+\varepsilon)^{-1}$, satisfying  \eqref{AChoixgammatilde}. \\
We next introduce the function  $\lambda^*(\gamma^*)$: 
\begin{equation}
\label{eq:Alambdacus*}
 \lambda^*(\gamma^*) = \frac{C_1}{1+\eta - \gamma^*} + \frac{C_2}{\gamma^* - 
1- \varepsilon}~,
\end{equation}
and the quantities $\gamma^*_i$, $i=1,2$, with 
$1+\varepsilon < \gamma^*_2 < \gamma^*_1 < 1+\eta$, satisfying
\begin{equation}
 \label{AChoixgamma*}
 \lambda^*(\gamma^*_1) \mathrm{Lip}(R)=  \lambda(\gamma^*_2) \mathrm{Lip}(R)=1~, 
 \hbox{ and } \lambda^*(\gamma) \mathrm{Lip}(R) <1~, \quad \forall \gamma^* \in (\gamma^*_2, \gamma^*_1)~.
\end{equation} 
We finally require that the following inequality holds:
\begin{equation}
\label{eq:HA5dernier}
\begin{split}
\min_{\gamma_2 \leq \gamma \leq \gamma_1} \frac{C_1 C_2  
\mathrm{Lip}(R) \gamma}
{ (1-\varepsilon) (\gamma - 1 +\eta) (1 -\lambda(\gamma) \mathrm{Lip}(R))}     
\times
\min_{\tilde{\gamma_1} \leq  \tilde{\gamma} \leq 
\tilde{\gamma_2}} \frac{C_1 C_2  \mathrm{Lip}(\tilde{R}) (1+\varepsilon)(1+\eta)} 
{(1+\eta - 1/\tilde{\gamma} ) (1- \tilde{\lambda}(\tilde{\gamma})\mathrm{Lip}(\tilde{R}))} <1.
\end{split}
\end{equation}
\end{description}

Applying  Theorems \ref{th:FoliaA1} and \ref{th:FoliaA3} to the above flow $S(.)$, we obtain the following properties, which are used in Sections~\ref{sec:core} and \ref{sec:IMforKG}.

\begin{theorem} \label{th:FoliaA4}
Assume that the hypotheses (HA.3) and (HA.5) are satisfied. Then, the following properties hold. 
\\
\begin{enumerate}
\item There exists a globally Lipschitz   map $g_{cu}: X_{cu} \to X_s$ with 
$g_{cu}(0)=0$, so that the Lipschitz  center unstable manifold $W^{cu}(0)$
$$ 
W^{cu}(0) = \{ x_c + x_{u} + g_{cu}(x_c + x_{u}) \, | \, x_c \in X_c, x_u \in X_u\}
$$
satisfies all the properties described in Theorem \ref{th:FoliaA1}. 
In particular, if $S(\tau)$ is of class $C^1$, then $g_{cu}: X_{cu} \to X_s$ is of class $C^1$.

\item There exists a globally Lipschitz   map $g_{u}: X_{u} \to X_{cs}$ with 
$g_{u}(0)=0$, so that the Lipschitz    unstable (also called strongly unstable) 
manifold $W^{u}(0)$
$$ 
W^{u}(0) = \{ x_{u} + g_{u}(x_{u}) \, | \, x_u \in X_u\}
$$
satisfies all the properties described in Theorem \ref{th:FoliaA1} with $\gamma$ replaced by $\gamma^*$ and $\gamma_{i}$ replaced by $\gamma^*_{i}$, $i=1,2$. In particular, if 
$S(\tau)$ is of class $C^1$, then $g_{u}: X_{u} \to X_{cs}$ is of class $C^1$. \\
And,  if a negative semi-orbit  $u(t)$, $ t\leq 0$, of  $S(.)$ is contained in $W^{u}(0)$, then,
\begin{equation}
\label{eq:ALyExpetoile1}
\limsup_{t \to -\infty} \frac{1}{|t|}\ln |u(t)| \leq - \frac{1}{\tau}\ln \gamma^*_1~.
\end{equation}

\item There exists a globally Lipschitz  map $g_{cs}: X_{cs} \to X_u$ with 
$g_{cs}(0)=0$ so that the Lipschitz  center stable manifold $W^{cs}(0)$ 
$$ 
W^{cs}(0) = \{ x_c  + x_{s} + g_{cs}(x_c + x_{s}) \, | \, x_c \in X_c, x_s \in X_s\}
$$
satisfies all the properties described in Theorem \ref{th:FoliaA3}. In particular, if $S(\tau)$ is of class $C^1$, then $g_{cs}: X_{cs} \to X_u$ is of class $C^1$. 
\item There exists a globally Lipschitz   map $g_{c}: X_{c} \to X_s \oplus X_u$ with 
$g_{c}(0)=0$, so that the Lipschitz  center  manifold $W^{c}(0)$
$$ 
W^{c}(0) = \{ x_c + g_{c}(x_c) \, | \, x_c \in X_c\} = W^{cu}(0) \cap W^{cs}(0)
$$
satisfies  the following properties: \\
(i) $W^{c}(0)$ is invariant under $S(t)$, i.e., $S(t)W^{c}(0) = W^{c}(0)$, for any $t \geq 0$. \\
(ii) The properties (ii) of Theorem \ref{th:FoliaA1} and the properties (ii) of Theorem
\ref{th:FoliaA3}  hold. In particular, if a trajectory $u(t)$, $t \in (-\infty, \infty)$ of $S(.)$ is contained in $W^{c}(0)$, then
\begin{equation}
\label{eq:propcenter}
\limsup_{t \to -\infty} \frac{1}{|t|}\ln |u(t)| \leq - \frac{1}{\tau}\ln \gamma_1~, \quad
\limsup_{t \to \infty} \frac{1}{t}\ln |u(t)|  \leq  \frac{1}{\tau}\ln 
\frac{1}{\tilde{\gamma_2}}~.
\end{equation}
Moreover, $W^{c}(0)$ contains all the equilibria of $S(t)$. \\
(iii)  If the map $S(\tau): X \to X$ is of class $C^1$, then 
$g_{c}:  X_{c} \to X_s \oplus X_u$ is of class $C^1$, that is, $W^{c}(0)$ is a $C^1$-submanifold of $X$. 

\item If moreover the condition \eqref{eq:AgamRinf1} holds with $\beta_1= 1- \varepsilon$
and $\beta_{2}= 1 -\eta$, then one has a foliation of $X$ over $W^{cu}(0)$ as defined in 
Theorem \ref{th:FoliaA2}.
\end{enumerate}
\end{theorem}

\begin{proof}
(1) Statements (1) and (5)  are direct consequences of Theorem \ref{th:FoliaA1} 
and Theorem \ref{th:FoliaA2} respectively, applied to the case where $\beta_1=1-\varepsilon$ and 
$\beta_2=1-\eta$. \\
(2) Statement (2) is a  direct consequence of Theorem \ref{th:FoliaA1}, applied to the case
where $\beta_1= 1+\eta$ and $\beta_2 = 1+\varepsilon$. \\
(3) Statement (3) is a  direct consequence of Theorem \ref{th:FoliaA3}, applied to the case
where $\beta_2^{-1}= (1+ \varepsilon)^{-1}$ and $\beta_1^{-1} = (1+ \eta)^{-1}$. \\
Let us next prove the statement (4).
We are looking for the trajectories $u(t)$, which satisfy both properties of 
\eqref{eq:propcenter}. These two properties together are satisfied only by the elements in
$W^{cu}(0) \cap W^{cs}(0)$. \\
Thus, we are looking for the elements $x= x_c +x_s + x_u$ so that
\begin{equation}
\label{eq:wcegalite1}
x_c + x_u + g_{cu}(x_c+x_u)= x_c +x_s +  g_{cs}(x_c+x_s)= x_c + g_{cu}(x_c+x_u) +
g_{cs}(x_c+ g_{cu}(x_0+x_u))~,
\end{equation}
or also for the elements $x_u \in X_u$ satisfying
\begin{equation}
\label{eq:wcegalite2}
x_u = g_{cs}(x_c+ g_{cu}(x_c+x_u))~.
\end{equation}
In other terms, given $x_c \in X_c$, we are looking for the fixed point of the map
$x_u \in X_u \mapsto F(x_c, x_u) = g_{cs}(x_c+ g_{cu}(x_c+x_u)) \in X_u$.
We notice that the Lipschitz constant of $F(x_c, .)$ satisfies
$$
\mathrm{Lip}(F(x_c,.)) \leq \mathrm{Lip}(g_{cs}) \times \mathrm{Lip}(g_{cu})~.
$$
By Theorems \ref{th:FoliaA1} and \ref{th:FoliaA3} and the assumption 
\eqref{eq:HA5dernier}, we have, for any $x_0 \in X_0$
\begin{equation}
\label{eq:ALipF0}
\begin{split}
&\mathrm{Lip}(F(x_c,.) \leq \cr
&\min_{\gamma_2 \leq \gamma \leq \gamma_1} \frac{C_1 C_2  
\mathrm{Lip}(R) \gamma}
{ (1-\varepsilon) (\gamma - 1 +\eta) (1 -\lambda(\gamma) \mathrm{Lip}(R))}     
\times
\min_{\tilde{\gamma_1} \leq  \tilde{\gamma} \leq 
\tilde{\gamma_2}} \frac{C_1 C_2  \mathrm{Lip}(\tilde{R}) (1+\varepsilon)(1+\eta)} 
{(1+\eta - 1/\tilde{\gamma} ) (1- \tilde{\lambda}(\tilde{\gamma})\mathrm{Lip}(\tilde{R}))} <1.
\end{split}
\end{equation}
Therefore, $x_u \in X_u \to F(x_c, x_u) \in X_u$ is a strict contraction, uniformly in $x_c$. Thus, for any $x_c \in X_c$, there exists a unique fixed point  $h(x_c) \in X_u$  
of $F(x_c,.)$. And $g_c(x_c)$ is given by
$$
g_c(x_c)= x_c +h(x_c) + g_{cu}(x_c+h(x_c))~.
$$
The regularity of the map $g_c$ is proved by using the regularity of the mappings $g_{cu}$ and $g_{cs}$ and by applying the uniform contraction principle of 
\cite[Theorem 2.2 on Page 25]{ChowHale}.
\end{proof}

\begin{remark} \label{rem:WuWs} 
1. If the equilibrium point is hyperbolic (that is, $\sigma^{c} = \emptyset $), then one can choose $\varepsilon = \eta $ in the hypotheses  (HA.5.1) and (HA.5.2).
The center unstable manifold $W^{cu}(0)$ and the (strongly) unstable manifold $W^{u}(0)$ coincide (that is, $g_{cu} =g_{u}$). And the center manifold  $W^{c}(0)$ reduces to $0$. 
\\
2. In the above theorem, we have only stated those properties which are used in this paper.  We leave it  to the reader to state the existence of the (strongly) stable manifold. 
\end{remark}

\section{Classical convergence results} \label{sec:ThconvBP}

In the study of  asymptotic behaviour of dynamical systems, one often encounters the following question: 
knowing that the $\omega$-limit set  of a relatively compact trajectory contains an equilibrium point $x_0$, 
does this $\omega$-limit set reduce to the point $x_0$, i.e., does the trajectory converge to $x_0$? This question is especially 
interesting in the case of gradient systems (that is, systems which admit a strict Lyapunov functional).  In fact, consider a gradient system 
with a  hyperbolic equilibrium $x_0$. Then $x_0$ is isolated and the whole trajectory converges to this point $x_0$. 
If the equilibrium $x_0$ is not hyperbolic and  the spectrum of the linearized dynamical system around $x_0$ 
intersects the unit circle, then $x_0$ could lie in a continuum of equilibria, which could be contained in the $\omega$-limit set. 
If $x_0$ belongs to a normally hyperbolic manifold of equilibria, we can still have convergence to $x_0$, under additional hypotheses.

In the proof of Theorem \ref{ThBRS1}, we use the convergence property to an equilibrium point in order to prove the boundedness of the orbits, which are global in forward time. We recall here the general convergence property in the form proved by 
Brunovsk\'{y} and Pol\'{a}\v{c}ik in \cite{BrP97b}, who extended  earlier  
convergence results, proved for example by Aulbach \cite{Aulb} in the finite-dimensional frame, or by Hale and Raugel \cite{HaRa}, who generalised  the convergence property 
of Aulbach to the infinite-dimensional setting (see also the paper \cite{HaMa} of 1982, 
 and \cite{Rau94} for applications). In the case of the one-dimensional parabolic equation with separate boundary conditions, convergence proofs had been given before in 
 \cite{Mata78} and \cite{Zele}.

Let $X$ be a Banach space and $\Phi: X \to X$ be a continuous map admitting a fixed point $y_0$. Without loss of generality, we may choose $y_0=0$.
 Brunovsk\'{y} and Pol\'{a}\v{c}ik assumed the following hypotheses:
\begin{itemize}
\item {\bf (HB.1)}  There exists a neighborhood $U$ of $0$ in $X$ so that the restriction 
$\Phi {\big |}_{U}: U \to X$ is of class $C^1$.

\item {\bf (HB.2)}  The spectrum $\sigma (DF(0))$ can be written as 
$\sigma (DF(0)) = \sigma^{s} \cup \sigma^{c} \cup \sigma^{u}$,  where $\sigma^{s}$, $\sigma^{c}$ and   $\sigma^{u}$ are closed subsets of $\{ \lambda \in \mathbb{C} \, | \,
 |\lambda| < 1\}$, $\{ \lambda \in \mathbb{C} \, | \, |\lambda| = 1\}$, and 
 $\{ \lambda \in \mathbb{C} \, | \, |\lambda| > 1\}$.
\end{itemize}
As in Appendix A, we introduce  the spectral projectors $P_{i}$ of $B= DF(0)$ associated with the spectral sets $\sigma^{i}$, $i= s, c, u$ and the images $X_i= P_{i}X$. We recall that these spaces are all $B$-invariant  and $X= X_{s} \oplus X_{c} \oplus X_{u}$. We also denote $X_{cu} = X_{c} \oplus X_{u}$. 

\medskip

As we have seen in Appendix A, the hypotheses (HB.1) and (HB.2) allow one to construct Lipschitz continuous local center unstable and local center manifolds  $W^{cu}_{loc}(0)$, $W^{c}_{loc}(0)$ of $\Phi$ at 
$0$ as graphs over 
$X_{cu}$ and $X_{c}$, respectively, and also the local unstable manifold $W^{u}_{loc}(0)$ as a graph over $X_{u}$, by  extending the map $\Phi$ into a global Lipschitz continuous and $C^1$ mapping  $\tilde{\Phi}$, which coincides with $\Phi$ on the ball $B_X(0,\delta)$
of center $0$ and radius $\delta >0$ ($\delta$ being small enough), and by applying Theorems \ref{th:FoliaA1} and \ref{th:FoliaA4}. These local invariant manifolds are defined in the following way
\begin{equation}
\label{eq:BvarietesInv}
W^{i}_{loc}(0) = \tilde{W}^{i}_{\delta}(0)~, \quad i= cu, c, u~,
\end{equation}
where $\tilde{W}^{cu}_{\delta}(0)$, $\tilde{W}^{c}_{\delta}(0)$ and 
$\tilde{W}^{u}_{\delta}(0)$ are the global center stable, center and unstable manifolds of 
$\tilde{\Phi}$ around $0$.

On the other hand, Theorem \ref{th:FoliaA2} in Appendix~A on the invariant foliations implies 
that $W^{cu}_{loc}(0)$  is exponentially attractive in $X$ with asymptotic phase (see Appendix \ref{sec:varietes} for more details).
Likewise, one can show that $W^{c}_{loc}(0)$ is exponentially attractive in backward time in $W^{cu}_{loc}(0)$ with asymptotic phase. These asymptotic phase properties are among the key  arguments in the proof of the main convergence theorem  \ref{th:ConvB1} below.

\begin{remark} \label{rem:PhiTilde}
Actually, 
the hypothesis (HB.1) can be replaced by the weaker hypothesis: \\
{\bf (HB.1bis)} There exists a neighborhood $U$ of $0$ in $X$ so that the restriction 
$\Phi {\big |}_{U}: U \to X$ is Lipschitz continuous and differentiable at every fixed point contained in $U$.
\end{remark}

Before stating the main convergence result of \cite{BrP97b}, we introduce the concept of stability restricted to $W^{c}_{loc}(0)$. 
We say that $0$ is stable for the map  $\Phi{\big |}_{W^{c}_{loc}(0)}$, if, for any $\varepsilon >0$, there exists $\eta >0$ 
such that, for any $y \in W^{c}_{loc}(0)$  with $\| y\|_{X} \leq \eta$, we have
\begin{equation}
\label{AppendBStable}
\| \Phi^n(y)\|_{X} \leq \varepsilon~, \quad \forall \; n=0,1,2, \ldots~.
\end{equation}

As pointed out in \cite{BrP97b}, this stability is independent of the choice of the local center manifold $W^{c}_{loc}(0)$. 
The independence of this stability on the choice  of the local center manifold can be proved by using foliations as in the paper of 
 \cite{BuDeLu}, who actually showed that the flows on different local center manifolds are conjugated (under some more restrictive hypotheses, which can be easily removed). As also remarked in  \cite{BrP97b}, the fact that the stability is independent of the choice of the local centre manifold, is not needed in the proof of Theorem~\ref{th:ConvB1} below.

\begin{theorem} \label{th:ConvB1}
Assume that the hypotheses (HB.1) (or (HB.1bis)) and (HB.2) hold. Let $x_0 \in X$ 
be such that the fixed point $0$ belongs to the $\omega$-limit set $\omega(x_0)$ 
of $x_0$. Assume that either $X_{cu}$ is finite-dimensional or that the trajectory 
$\Phi^{n}(x_0)$, $n=1,2, \cdots$, of $x_0$ is relatively compact. 
Assume, moreover, that $0$ is stable for the map $\Phi{\big |}_{W^{c}_{loc}(0)}$, where 
$W^{c}_{loc}(0)$ is a local center manifold of $0$. \\
Then  either $\Phi^{n}(x_0)$ converges to $0$ as $n\to\infty$, or 
$\omega(x_0)$ contains a point of the local unstable manifold $W^{u}_{loc}(0)$ of $0$, distinct from $0$.
\end{theorem}

Theorem \ref{th:ConvB1} generalises the above mentioned convergence result of \cite{HaRa} 
 in two ways. Firstly, the hypotheses do not require that  
 $\omega(x_0)$ consists only of fixed points. Secondly, it does not require that the trajectory $\Phi^{n}(x_0)$, $n=1,2, \cdots$, of 
$x_0$ be relatively compact.  But, of course, it requires the additional stability property defined above. 

In \cite{BrP97b}, Brunovsk\'{y} and Pol\'{a}\v{c}ik have proved the following lemma (see \cite[Lemma 1]{BrP97b}) and have obtained Theorem \ref{th:ConvB1} as a direct consequence of it. We emphasize that  Lemma \ref{lemmeB1} is really a local result anf that Lemma \ref{lemmeB1} will hold for any mapping 
$\Phi^{*}: y \in \mathcal{U} \mapsto  \Phi^{*} y \in X$ coinciding with  $\Phi$ in 
$\mathcal{U}$. In particular, $\Phi^{*}$ need not be well defined outside $\mathcal{U}$, which is the case in our application in Section 3.

\begin{lemma} \label{lemmeB1}
Assume that the hypotheses (HB.1) (or (HB.1bis)) and (HB.2) hold, that $\delta>0$ is small enough so that $B_X(0, \delta) \subset \mathcal{U}$ and that $0$ is stable for the map 
$\Phi{\big |}_{W^{c}_{loc}(0)}$. Let $x_k \in X$ and $p_k \in \mathbb{N}$ be sequences satisfying the following properties:
\begin{enumerate}
\item $x_k \to 0$ as $k \to +\infty$.
\item $\Phi^j(x_k) \in B_X(0,\beta)$ for $j= 0,1,2, \ldots, p_k$ and 
$\Phi^{p_k +1}(x_k) \notin B_X(0,\beta)$, where $0< \beta < \delta$.
\item In the case, where $\dim X_{cu} = \infty$, the set $\{\Phi^j_*(x_k) \, | \, k \in 
\mathbb{N}, j= 0,  \ldots, p_k \}$ is relatively compact.
\end{enumerate}
Then $\Phi^{p_k}(x_k)$  contains a subsequence converging to an element of $W^u_{loc}(0) \setminus \{0\}$.
\end{lemma}

As an easy consequence of Theorem \ref{th:ConvB1},  Brunovsk\'{y} and Pol\'{a}\v{c}ik have obtained the following more classical theorem. 

\begin{theorem} \label{th:ConvB2}
Assume that the hypotheses (HB.1) (or (HB.1bis)) and (HB.2) hold. Let $x_0$ be a point in 
$X$ such that the fixed point $0$ belongs to the $\omega$-limit set $\omega(x_0)$ of 
$x_0$ and such that $\omega(x_0)$ is contained in the set $\mathrm{Fix}(\Phi)$ of fixed points of $\Phi$. Assume that either $X_{cu}$ is finite-dimensional or that the trajectory 
$\Phi^{n}(x_0)$, $n=1,2, \cdots$, of $x_0$ is relatively compact. 
Assume moreover that one of the following two properties holds: 
\begin{enumerate}
\item $\dim X^{c} = 1$ and the trajectory 
$\Phi^{n}(x_0)$, $n=1,2, \cdots$, of $x_0$ is relatively compact. 
\item $\dim X^{c} = m < \infty$ and there is a submanifold $M \subset X$ with 
$\dim M=m$ such that $0 \in M \subset \mathrm{Fix}(\Phi)$.
\end{enumerate}
Then $\omega(x_0) = \{0\}$.
\end{theorem}

\begin{proof} We give the proof, because it is    short. 

First assume that (2) holds. Then, if $\delta >0$ is small enough, the sets $M$ and $W^u_{loc}(0)$ coincide since 
$M \subset W^u_{loc}(0)$,  and  they both have the same dimension $m$. The  assumption 
$M  \subset \mathrm{Fix}(\Phi)$ thus implies that $0$ is stable for  
 the map $\Phi{\big |}_{W^{c}_{loc}(0)}$. Since $W^u_{loc}(0) \setminus \{0\}$ contains no 
 fixed point if $\delta >0$ is small enough and since $\omega(x_0) \in \mathrm{Fix}(\Phi)$,
 Theorem \ref{th:ConvB1} implies that $\omega(x_0) = \{0\}$.
 
In the case (1), we first remark that, since  the trajectory $\Phi^{n}(x_0)$, $n=1,2, \cdots$, of $x_0$ is relatively compact and since $\omega(x_0)$ consists only of fixed points, the omega-limit set $\omega(x_0)$ is connected (see for example 
\cite[Lemma 2.7]{HaRa}). If $\omega(x_0)$  contains more than one fixed point, then all fixed points near $0$ are contained in $W^c_{loc}(0)$ and thus $0$ belongs to a curve of fixed points. If $0$ belongs to the relative interior of this curve, one applies the case (2), which leads to a contradiction. If $0$ does not belong to the relative interior of this curve, we consider a fixed point $y^*$ near $0$, contained in the relative interior of this curve of fixed points and in $\omega(x_0)$. Replacing $\Phi$ by $\Phi(y^* +x)$, we are now back to the case (2). Applying the case (2), we obtain that  $\omega(x_0)=y^*$, which also leads to a contradiction. 
\end{proof}

Suppose we consider an element $x_0 \in X$ such that we do not a priori know that the trajectory $\{\Phi^n(x_0) \, | \, n \in \mathbb{N}\}$ is bounded.  Then, even if $\dim X^{c} = 1$, we cannot directly apply case~(1) of Theorem \ref{th:ConvB2}. Indeed, the proof of case~(1) uses the connectedness property of $\omega(x_0)$.
One can then try to apply the more general Theorem \ref{th:ConvB1} in order to obtain a convergence result.

 In Section~\ref{sec:conv} we encountered such a case. We did not know there  that the forward trajectory $\{S_{\alpha}(t) \vec u_0 \, | \, t  \geq 0\}$ is bounded. 
 Thus, as in the proof of Lemma \ref{lemmeB1}, we used the property that 
 $W^{cu}_{loc}(0)$  is exponentially attractive in $X$ with asymptotic phase together with the fact that $\dim X^{c} = 1$, to obtain that $S_{\alpha}(t)$ has the  stability property \eqref{eq:2.16} (or \eqref{AppendBStable}). Then, we  applied  Theorem \ref{th:ConvB1}  to the  time $\tau$-map $\Phi =S_{\alpha}(\tau)$, where $\tau >0$ is small enough, in order to obtain the convergence result.
 
Since the arguments in Section~\ref{sec:conv} did not use the particular properties of 
 $S_{\alpha}(t)$, it is also valid  in the case of a more general semi-flow and allows us to state the following general result. 

 \begin{coro}\label{th:CoroB5}
Assume that the map  $\Phi = S(\tau)$ where $S(t): \mathbb{R} \times X \to X$ is a continuous dynamical system and that $\tau >0$ is a small enough positive time, so that
$\Phi = S(\tau)$  satisfies  the hypotheses (HB.1) (or (HB.1bis)) and (HB.2). Let $x_0$ be a point in $X$ such that the 
equilibrium point $0$ belongs to the $\omega$-limit set $\omega(x_0)$ of $x_0$ and such that $\omega(x_0)$ 
is contained in the set of equilibrium  points of $S(t)$. 
Assume that either $X_{cu}$ is finite-dimensional or that the trajectory 
$\Phi^{n}(x_0)$, $n=1,2, \cdots$, of $x_0$ is relatively compact. 
Assume moreover that $\dim X^{c} = 1$. Then $\omega(x_0) = \{0\}$.
\end{coro}

Let us finally notice that, in the case of gradient systems generated by some evolutionary equations with an analytic non-linearity or satisfying the hypotheses (1) or (2) of 
Theorem \ref{th:ConvB2}, one also obtains convergence results based on the \L ojasiewicz inequality (see \cite{Si96, HaJen99, HaJen07})  and on the construction of appropriate energy functionals, when the trajectories are relatively compact. 
The  convergence proofs given \cite{Si96, HaJen99, HaJen07}  require in an essential way that the trajectory 
$S(t) x_0$, $t \geq 0$, be relatively compact and thus do not seem to be applicable in Section~\ref{sec:conv} above.

\end{document}